\documentclass[10pt, oneside]{article}
\usepackage{mathrsfs}
\usepackage{amsfonts}
\usepackage{amsmath,amsfonts,amssymb,amsthm}
\usepackage{caption}
\usepackage{subfigure}
\usepackage{cite}
\usepackage{color}
\usepackage{graphicx}
\usepackage{subfigure}
\usepackage{float}
\usepackage{paralist}
\usepackage{indentfirst}
\usepackage{authblk}

\usepackage[colorlinks,
linkcolor=red,
citecolor=blue
]{hyperref}

\usepackage[T1]{fontenc}
\DeclareMathOperator{\dive}{div}

\newtheorem{theorem}{Theorem}[section]
\newtheorem{lemma}{Lemma}[section]
\newtheorem{prop}{Proposition}[section]

\newtheorem{remark}{Remark}[section]

\newtheorem{defn}{Definition}[section]

\def\bma#1\ema{{\allowdisplaybreaks\begin{split}#1\end{split}}}

\numberwithin{equation}{section}

\begin{document}
\title{{\LARGE \textbf{Global existence of weak solutions to the drift-flux system for general pressure laws}}}
\author[a,b]{Hai-Liang  Li  \thanks{
E-mail:		hailiang.li.math@gmail.com (H.-L. Li).}}

\author[a,b]{Ling-Yun Shou \thanks{Corresponding author. E-mail: shoulingyun11@gmail.com(L.-Y Shou).}}
    \affil[a]{School of Mathematical Sciences,
	Capital Normal University, Beijing 100048, P.R. China}
\affil[b]{Academy for Multidisciplinary Studies, Capital Normal University, Beijing 100048, P.R. China}

\date{}
\renewcommand*{\Affilfont}{\small\it}
\maketitle
\begin{abstract}
The initial value problem of the multi-dimensional drift-flux model for two-phase flow is investigated in this paper, and the global existence of weak solutions with finite energy is established for general pressure-density functions without the monotonicity assumption.
\end{abstract}
\noindent{\textbf{Key words:} Two-phase flow, drift-flux model, global weak solutions, mon-monotone pressure laws, quantitative regularity estimate }
\section{Introduction}

The drift-flux model for two-phase flow can be used widely in many applied scientific areas, such as nuclear, chemical-process, petroleum, cryogenic, medicine, oil-and-gas, microtechnology, sprays \cite{brennen1,desv1,evje2,gidaspow1,ishii1,ishii2,wallis1,zuber1,mellet1} and so on. In this paper, we consider the initial value problem (IVP) for the drift-flux system in the periodic domain $\mathbb{T}^d:=\mathbb{R}^{d}/\mathbb{Z}^d$ ($d\geq2$):
\begin{equation}\label{two}
\left\{
\begin{split}
&\partial_{t}\rho+\dive(\rho u)=0,\\
&\partial_{t}n+\dive(n u)=0,\\
&\partial_{t}\big{(}(\rho+n) u\big{)}+\dive\big{(}(\rho+n) u\otimes u\big{)}+\nabla P(\rho,n)=\mu\Delta u+(\mu+\lambda)\nabla\dive u,\quad x\in \mathbb{T}^d,\quad t>0,
\end{split}
\right.
\end{equation}
with the initial data
\begin{equation}
\begin{split}
(\rho, n, (\rho+n)u)(x,0)=(\rho_{0},n_{0},m_{0})(x),\quad x\in \mathbb{T}^d,\label{d}
\end{split}
\end{equation}
where $\rho=\rho(x,t)\in\mathbb{R}_{+}$ and $n=n(x,t)\in\mathbb{R}_{+}$ denote the densities of two fluids, respectively, $u=u(x,t)\in \mathbb{R}^d$ stands for the mixed velocity, the shear viscosity coefficient $\mu$ and the bulk viscosity coefficient $\lambda$ satisfy
\begin{equation}\label{mu}
\begin{split}
\mu>0,\quad 2\mu+\lambda>0,
\end{split}
\end{equation}
and the pressure $P(\rho,n)$ is assumed to be the general laws
\begin{equation}\label{P}
\left\{
\begin{split}
&P(\rho,n)\in C^1(\mathbb{R}_{+}\times\mathbb{R}_{+}),\\
&\frac{1}{C_{0}}(\rho^{\gamma}+n^{\alpha})-C_{1}\leq P(\rho,n)\leq C_{0}(\rho^{\gamma}+n^{\alpha})+C_{1},\\
&|\partial_{\rho} P(\rho,n)|+ |\partial_{n} P(\rho,n)|\leq C_{2}((\rho^{\widetilde{\gamma}-1}+n^{\widetilde{\alpha}-1}+1),
\end{split}
\right.
\end{equation}
with the constants $C_{i}$ $(i=0,1,2,3)$, $\widetilde{\gamma}, \widetilde{\alpha}$ and $\beta$ satisfying
\begin{equation}
\begin{split}
C_{0}\geq1,\quad C_{1}\geq 0,\quad C_{2}>0,\quad C_{3}>0,\quad  \widetilde{\gamma},\widetilde{\alpha}\geq 1.
\end{split}
\end{equation}

There are many significant progress made recently on the mathematical analysis of global solutions of the drift-flux system (\ref{two})  \cite{evje1,evje4,hcc1,guo1,yu1,yao1,yao2,yao3,yao4,evje3,evje2,yao2,yao3,bresch3,vasseur1,wen1,wen2,wen20,novotny2,zhang1,wangw1,chens1}. Among them, for the pressure 
$$
P(\rho,n)=C(-b(\rho,n)+\sqrt{b^2(\rho,n)+c(\rho,n)}),
$$
 the global existence of weak solutions in the one-dimensional case for general initial data was established in \cite{evje1,evje4}, and the global well-posedness and large time behaviors of weak and strong solutions around the constant equilibrium state in the multi-dimensional case were investigated in \cite{hcc1,guo1,yu1,yao1,yao4,wen20}. For the pressure 
 $$
 P(\rho,n)=C\rho_{l}^{\gamma}(\frac{n}{\rho_{l}}-\rho)^{\gamma},
 $$
   the well-posedness and dynamics of global weak solutions to the one-dimensional free boundary value problem were studied in \cite{evje3,evje2,yao2,yao3}. As for the pressure 
   $$
   P(\rho,n)=\rho^{\gamma}+n^{\alpha},
   $$
    the global weak solutions with finite energy for the three-dimensional drift-flux equations (\ref{two}) was obtained by Bresch-Mucha-Zatorska \cite{bresch3} for $\gamma, \alpha>1$ in a semi-stationary Stokes regime, by Vasseur-Wen-Yu \cite{vasseur1} for $\gamma> \frac{9}{5}$ and $\alpha$ satisfying that either $\alpha$ is close to $\gamma$ enough or $\alpha\geq 1$ if the following domination condition holds:
\begin{align}
\underline{c}\rho_{0}\leq n_{0}\leq \overline{c}\rho_{0},\label{dom}
\end{align}
with two constants $0<\underline{c}\leq \overline{c}<\infty$, and then by Wen \cite{wen1} for two independent adiabatic constants $\gamma,\alpha \geq\frac{9}{5}$. For more general pressures laws $P(\rho,n)$ which can be non-monotone on a compact set of the two variables $\rho$ and $n$, the global existence of weak solutions to (\ref{two}) has been proved by Novotn$\rm{\acute{y}}$-Pokorn$\rm{\acute{y}}$ \cite{novotny2} in 3D under the domination condition (\ref{dom}) with two constants $0\leq \underline{c}\leq \overline{c}<\infty$, and by Wen-Zhu \cite{wen2} in 1D without the domination condition (\ref{dom}). In addition, there are many interesting studies on the global existence of weak solutions to some models related to (\ref{two}), such as compressible Oldroyd-B model \cite{barrett1}, two-dimensional non-resistive compressible MHD equations \cite{ly1}, compressible Navier-Stokes equations with entropy transport \cite{maltese1}, compressible Navier-Stokes-Vlasov-Fokker-Planck system \cite{mellet1}, etc. The reader can refer to the review papers \cite{wen0,wen2}. 

The drift-flux system (\ref{two}) can be viewed as a simplified model of the compressible nonconservative two-fluid equations (cf. \cite{bresch4,bresch2,ishii2,wen0}):
\begin{equation}\label{ntf}
\left\{
\begin{split}
&\alpha^{+}+\alpha^{-}=1,\\
&\partial_{t}(\alpha^{\pm}\rho^{\pm})+\dive (\alpha^{\pm}\rho^{\pm} u^{\pm})=0,\\
&\partial_{t}(\alpha^{\pm}\rho^{\pm}u^{\pm})+\dive (\alpha^{\pm} \rho^{\pm} u^{\pm}\otimes u^{\pm})+\alpha^{\pm}\nabla P^{\pm}(\rho^{\pm})=\dive (\alpha^{\pm}\tau^{\pm})+\sigma^{\pm}\alpha^{\pm}\rho^{\pm}\nabla\Delta(\alpha^{\pm}\rho^{\pm}),\\
\end{split}
\right.
\end{equation}
where $\alpha^{\pm}$, $\rho^{\pm}, u^{\pm},$ $P^{\pm}(\rho^{\pm})$ and $\tau^{\pm}$ denote the volume fractions, densities, velocities, pressures and stress tensors of two fluids, respectively. The global existence of weak solutions to (\ref{ntf}) with degenerate viscosity coefficients was proved in \cite{bresch4,bresch5} for $P^{+}(\rho^{+})=P^{-}(\rho^{-})$, and the global well-posedness and optimal time-decay rates of strong solutions near the constant equilibrium state to (\ref{ntf}) were studied in \cite{evje6,wang1} for $P^{+}(\rho^{+})\neq P^{-}(\rho^{-})$.

The drift-flux system (\ref{two}) for either $\rho= 0$ or $\rho= n$ can reduce to the baratropic compressible Navier-Stokes equations:
\begin{equation}\label{ns}
\left\{
\begin{split}
&\partial_{t}\varrho+\dive(\varrho v)=0,\\
&\partial_{t}(\varrho v)+\dive(\varrho v\otimes v)+\nabla \pi(\varrho)=\mu\Delta v+(\mu+\lambda)\nabla \dive v,
\end{split}
\right.
\end{equation}
where $\varrho, v,$ and $\pi(\varrho)$ denote the density, the velocity and the pressure function, respectively. There are many important investigations about global weak solutions to compressible Navier-Stokes equations (\ref{ns}) with constant viscosity coefficients \cite{lions2,feireisl1,jiang1,plotnikov1,hu1,feireisl2,bresch1}. For instance, concerning the $\gamma$-law $\pi(\varrho)=\varrho^{\gamma}$, the existence of global weak solutions with finite energy to (\ref{ns}) has been established by Lions \cite{lions2} for $\gamma\geq \frac{3d}{d+2}$ $(d=2,3)$ and $\gamma>\frac{d}{2}$ $(d\geq 4)$, then by Feireisl-Novotn$\rm{\acute{y}}$-Petzeltov$\rm{\acute{a}}$ \cite{feireisl1} for $\gamma>\frac{d}{2}$ $(d\geq2)$, by Jiang-Zhang \cite{jiang1} for $\gamma>1$ subject to spherically symmetric initial data and by Plotnikov-Weigant \cite{plotnikov1} for $\gamma=1$ in two dimension. Indeed, the concentration phenomenon of the convective term $\rho u\otimes u$ may occur for $\gamma\in [1,\frac{d}{2}]$ \cite{hu1}. Feireisl \cite{feireisl2} proved the global existence of weak solutions to (\ref{ns}) for the pressure laws allowed to be non-monotone on a compact set. Bresch-Jabin \cite{bresch1} developed new compactness tools to obtain global weak solutions to (\ref{ns}) for more general stress tensors, including thermodynamically unstable pressure laws and anisotropic viscosity coefficients. Moreover, some important progress has been made about global weak solutions to compressible Navier-Stokes equations with degenerate viscosity coefficients, refer to \cite{bresch10,bresch20,bresch30,bresch40,l1,mellet10,vasseur10} and references therein.

However, there is no any result about the global existence problem of the multi-dimensional drift-flux model (\ref{two}) without the monotonicity assumption of the pressure as we know so far. The purpose of this paper is to extend the compactness techniques in \cite{bresch1} to the two-phase case and establish the global existence of weak solutions to the IVP (\ref{two})-(\ref{d}) for the general pressure laws (\ref{P}).

First, we give the definition of global weak solutions to the IVP $(\ref{two})$-$(\ref{d})$ as follows.
\begin{defn}\label{defn11}
 $(\rho,n,(\rho+n)u)$ is said to be a global weak solution to the IVP $(\ref{two})$-$(\ref{d})$ if for any time $T>0$, the following properties hold$:$

{\rm(1)}. It holds 
\begin{equation}\nonumber
\left\{
\begin{split}
&0\leq\rho\in C([0,T];L^{\gamma}_{weak}(\mathbb{T}^{d} )),\quad 0\leq n\in C([0,T];L^{\alpha}_{weak}(\mathbb{T}^{d}) ), \\
&u\in L^2(0,T;H^1(\mathbb{T}^{d}) ), \quad \sqrt{\rho+n} u\in L^{\infty}(0,T, L^2(\mathbb{T}^{d})), \quad (\rho+n) u \in C([0,T];L^{ \frac{2\min{\{\gamma,\alpha\}}}{\min{\{\gamma,\alpha\}}+1}}_{weak}(\mathbb{T}^{d}) ),\\
&(\rho,n,(\rho+n)u)(x,0)=(\rho_{0},n_{0},m_{0})(x),\quad\text{a.e.}~x\in\mathbb{T}^{d}.
\end{split}
\right.
\end{equation}

{\rm(2)}. The continuity equations $(\ref{two})_{1}$-$(\ref{two})_{2}$ are satisfied in the sense of renormalized solutions$:$
 \begin{equation}
 \left\{
 \begin{split}
 &\partial_{t}b(\rho)+\dive(b(\rho) u)+[b'(\rho)\rho-b(\rho)]\dive u=0\quad~\text{in}~\mathcal{D}(\mathbb{T}^{d}\times(0,T)),\\
  &\partial_{t}b(n)+\dive(b(n) u)+[b'(n)n-b(n)]\dive u=0\quad\text{in}~\mathcal{D}(\mathbb{T}^{d}\times(0,T)),\nonumber
 \end{split}
 \right.
 \end{equation}
where $b\in C^{1}(\mathbb{R})$ satisfies $b'(z)=0$ for all $z\in\mathbb{R}$ large enough.

{\rm(3)}. The momentum equation $(\ref{two})_{3}$ is satisfied in $\mathcal{D}'(\mathbb{T}^{d}\times(0,T))$.

{\rm(4)}. The energy inequality holds for a.e. $t\in(0,T)$$:$
\begin{equation}
\begin{split}
&\int_{\mathbb{T}^d} [\frac{1}{2}(\rho+n) |u|^2+G(\rho,n)] dx+\int_{0}^{t}\int_{\mathbb{T}^d}[\mu|\nabla u|^2+(\mu+\lambda)(\dive{u})^2]dxd\tau\\
&\quad\leq \int_{\mathbb{T}^d} [\frac{1}{2}\frac{|m_{0}|^2}{\rho_{0}+n_{0}}+G(\rho_{0},n_{0})] dx,\label{energy}
\end{split}
\end{equation}
where the Helmholtz free energy function $G(\rho,n)$ is defined by
\begin{eqnarray}
\label{G} G(\rho,n):=
\begin{cases}
(\rho+n)\int_{1}^{\rho+n}P(\frac{\rho s}{\rho+n},\frac{n}{\rho+n}s)s^{-2}ds,\quad\quad
& \mbox{if\quad$ \rho+ n>0,$ } \\
0,
& \mbox{if ~~$\rho=n=0.$}
\end{cases}
\end{eqnarray}
\end{defn}

\vspace{2ex}

Then, we have the global existence of weak solutions to the IVP (\ref{two})-(\ref{d}) below.
\begin{theorem}\label{theorem11}
Suppose that the initial data $(\rho_{0},n_{0},m_{0})$ satisfies
\begin{equation}\label{a1}
\left\{
\begin{split}
&0\leq \underline{c}\rho_{0}(x)\leq n_{0}(x)\leq \overline{c}\rho_{0}(x),\quad\text{a.e.}~x\in\mathbb{T}^{d},\\
&\frac{m_{0}(x)}{\sqrt{(\rho_{0}+n_{0})(x)}}=0,\quad\text{if}~ (\rho_{0}+n_{0})(x)=0,\quad\text{a.e.}~x\in\mathbb{T}^{d},\\
&(\rho_{0},n_{0},\frac{m_{0}}{\sqrt{\rho_{0}+n_{0}}})\in L^{\gamma}(\mathbb{T}^{d})\times L^{\alpha}(\mathbb{T}^{d})\times L^2(\mathbb{T}^{d}),
\end{split}
\right.
\end{equation}
with two constants $0<\underline{c}\leq \overline{c}<\infty$. Then, if it holds for the pressure $P(\rho,n)$ given by $(\ref{P})$ that
\begin{equation}\label{gamma1}
\left\{
\begin{split}
&\gamma\geq \frac{3d}{d+2}~(d=2,3),\quad\quad~ \gamma>\frac{d}{2}~(d>4),\quad\quad  \alpha\geq 1,\\
& \widetilde{\gamma},\widetilde{\alpha}\leq\max{\{\gamma,\alpha\}}+\theta,\quad \quad\theta:=\frac{2}{d}\max{\{\gamma,\alpha\}}-1>0,
\end{split}
\right.
\end{equation}
the IVP $(\ref{two})$-$(\ref{d})$ admits a global weak solution $(\rho, n, (\rho+n)u)$ in the sense of Definition \ref{defn11}.
\end{theorem}

\vspace{2ex}

Without the domination condition $(\ref{dom})$, we also have the global existence of weak solutions to the IVP (\ref{two})-(\ref{d}) as follows.
\begin{theorem}\label{theorem12}
 Suppose that the initial data $(\rho_{0},n_{0},m_{0})$ satisfies
\begin{equation}\label{a2}
\left\{
\begin{split}
&\rho_{0}(x)\geq 0,\quad  n_{0}(x)\geq 0,\quad\text{a.e.}~x\in\mathbb{T}^{d},\\
&\frac{m_{0}(x)}{\sqrt{(\rho_{0}+n_{0})(x)}}=0,\quad\text{if}~ (\rho_{0}+n_{0})(x)=0,\quad\text{a.e.}~x\in\mathbb{T}^{d},\\
&(\rho_{0},n_{0},\frac{m_{0}}{\sqrt{\rho_{0}+n_{0}}})\in L^{\gamma}(\mathbb{T}^{d})\times L^{\alpha}(\mathbb{T}^{d})\times L^2(\mathbb{T}^{d}).
\end{split}
\right.
\end{equation}
 Then, if it holds for the pressure $P(\rho,n)$ given by $(\ref{P})$ that
\begin{equation}\label{gamma2}
\left\{
\begin{split}
&\gamma,\alpha\geq \frac{3d}{d+2}~(d=2,3),\quad\quad \gamma,\alpha>\frac{d}{2}~(d\geq4),\\
&1\leq \widetilde{\gamma}\leq \gamma+\theta_{1},  \quad\quad \quad~~  \quad \quad \theta_{1}:=\frac{2}{d}\gamma-\frac{\gamma}{\min{\{\gamma,\alpha\}}}>0,\\
&1\leq \widetilde{\alpha}\leq\alpha+\theta_{2}, \quad   \quad\quad  ~~ \quad \quad\theta_{2}:=\frac{2}{d}\alpha-\frac{\alpha}{\min{\{\gamma,\alpha\}}}>0,
\end{split}
\right.
\end{equation}
the IVP $(\ref{two})$-$(\ref{d})$ admits a global weak solution $(\rho, n, (\rho+n)u)$ in the sense of Definition \ref{defn11}.
\end{theorem}

\begin{remark}
There are many interesting pressure laws satisfying $(\ref{P})$, for instance,

(1) We can take
 $$
 P(\rho,n)=\rho^{\gamma}+n^{\alpha}+\sum_{i=1}^{N} c_{i}\rho^{\gamma_{i}}n^{\alpha_{i}},
 $$
where the constants $N$, $\gamma_{i}$ and $\alpha_{i}$ $(i=1,...,N)$ satisfy $N\geq1$, $0\leq \gamma_{i}<\gamma$, $0\leq \alpha_{i}<\alpha$ and $\gamma_{i}+\alpha_{i}<\max\{\gamma,\alpha\}$ and the constants $c_{i}$ $(i=1,...,N)$ are allowed to be negative such that $P(\rho,n)$ can be non-monotone. In particular, one can choose the monotone pressure
$$
P(\rho,n)=\rho^{\gamma}+n^{\alpha}.
$$

(2)  Let~$\Pi_{1}(\rho)$ and $\Pi_{2}(n)$ be two non-monotone pressure laws for compressible Navier-Stokes equations given in {\rm{\cite{bresch1}}}, for instance, the virial expansion with high order terms, which can be thermodynamically unstable. Then the pressure 
$$
P(\rho,n)=\Pi_{1}(\rho)+\Pi_{2}(n)
$$
satisfies {\rm{(\ref{P})}}. In particular, the pressure function $P(\rho,n)=\Pi_{1}(\rho)+n^2$ corresponds to the two-dimensional non-resistive compressible MHD equations {\rm\cite{ly1}}.

(3) An example of pressure laws is
$$
P(\rho,n)=\rho^{\gamma}(1+2\cos \rho)+n^{\alpha}(1+2\cos n),
$$
which is oscillatory even for large $\rho$ and $n$.
\end{remark}

\begin{remark}
If the domination condition $(\ref{dom})$ holds, then one can show $0\leq \underline{c}n(x,t)\leq \rho(x,t)\leq \overline{c}\rho(x,t)$ for a.e. $(x,t)\in\mathbb{T}^{d}\times(0,T)$ due to the comparsion principle for $(\ref{two})_{1}$-$(\ref{two})_{2}$ so that $G(\rho,n)$ satisfies 
\begin{equation}\label{P1}
\begin{split}
\frac{1}{C_{\underline{c},\overline{c}}}(\rho+n)^{\max\{\gamma,\alpha\}}-C_{\underline{c},\overline{c}}\leq G(\rho,n)\leq C_{\underline{c},\overline{c}} (\rho+n)^{\max\{\gamma,\alpha\}}+C_{\underline{c},\overline{c}},
\end{split}
\end{equation}
where $C_{\underline{c},\overline{c}}\geq 1$ is a constant. Indeed, we can prove under the assumptions $(\ref{a1})$-$(\ref{gamma1})$ that both $\rho$ and $n$ are bounded in $L^{\max\{\gamma,\alpha\}+\theta}(0,T;L^{\max\{\gamma,\alpha\}+\theta}(\mathbb{T}^{d}))$ for $\max\{\gamma,\alpha\}+\theta\geq 2$ so as to apply the arguments of renormalized solutions.
\end{remark}

\begin{remark}
Theorem \ref{theorem12} is the first result on the global existence of weak solutions to the multi-dimensional drift-flux model $(\ref{two})$ for non-monotonous pressure laws without any requirement on the relation between $\rho_{0}$ and $n_{0}$ as in $(\ref{dom})$. As pointed in {\rm \cite{evje4,wen1}}, the system $(\ref{two})$ without the condition $(\ref{dom})$ is more realistic in some physical situations has and more "two-phase" properties from mathematical points of view.
\end{remark}

\begin{remark}
The adiabatic constants $\gamma$, $\alpha$ can take $\frac{3d}{d+2}$ for $d=2$, $3,$ which extends the previous work by Bresch-Jabin {\rm\cite{bresch1}}  about global existence of weak solutions to compressible Navier-Stokes equations $(\ref{ns})$ for thermodynamically unstable pressure laws.
\end{remark}

\begin{remark}
If $(\ref{P})$ holds, then we have 
\begin{equation}\label{P2}
\begin{split}
\frac{1}{C_{0}}(\frac{\rho^{\gamma}}{\gamma-1}+\frac{n^{\alpha}}{\alpha-1})\leq G(\rho,n)\leq C_{0}(\frac{\rho^{\gamma}}{\gamma-1}+\frac{n^{\alpha}}{\alpha-1}).
\end{split}
\end{equation}
This implies under the assumptions $(\ref{a2})$-$(\ref{gamma2})$ that $\rho$ and $n$ belong to  $L^{\gamma+\theta_{1}}(0,T;L^{\gamma+\theta_{1}}(\mathbb{T}^{d}))$ for $\gamma+\theta_{1}\geq 2$ and $L^{\alpha+\theta_{2}}(0,T;L^{\alpha+\theta_{2}}(\mathbb{T}^{d}))$ for $\alpha+\theta_{2}\geq 2$, respectively.
\end{remark}

\vspace{2ex}

We explain the main strategies to prove Theorem \ref{theorem11} and Theorem \ref{theorem12}. By the Faedo-Galerkin approximation and vanishing artificial viscosity, one can construct the approximate sequence $(\rho_{\delta},n_{\delta},(\rho_{\delta}+n_{\delta})u_{\delta})$ of the IVP $(\ref{two})$-$(\ref{d})$ with the artificial pressure.

Then, we pass to the limit as $\delta\rightarrow 0$. As emphasized in many related papers \cite{novotny2, vasseur1,wen2}, the key point is to show the strong convergence of the two densities $\rho_{\delta}$ and $n_{\delta}$. Due to the possible non-monotonicity of $P(\rho,n)$ in the two variables $\rho$ and $n$, it is difficult to apply the techniques by Bresch-Jabin {\rm\cite{bresch1}} to estimate $\rho_{\delta}$ and $n_{\delta}$ separately. To overcome this difficulty, we first make use of the ideas inspired by \cite{vasseur1} to deduce 
\begin{equation}\nonumber
\begin{split}
&\quad\quad~(\rho_{\delta},n_{\delta})\rightarrow (\rho,n)\quad\text{in}~L^1(0,T;L^1(\mathbb{T}^{d}))\times L^1(0,T;L^1(\mathbb{T}^{d}))\\
&\Longleftrightarrow \rho_{\delta}+n_{\delta}\rightarrow \rho+n\quad\text{in}~L^1(0,T;L^1(\mathbb{T}^{d}) )\quad\text{as}~\delta\rightarrow 0.
\end{split}
\end{equation}
By virtue of the compactness criterion introduced in \cite{belgacem1,bresch1}, one needs to derive the quantitative regularity estimate of the sum $\vartheta_{\delta}:=\rho_{\delta}+n_{\delta}$:
\begin{equation}\nonumber
\begin{split}
&\frac{1}{\|\mathcal{K}_{h}\|_{L^1}}\int_{0}^{T}\int_{\mathbb{T}^{2d}} \mathcal{K}_{h}(x-y) |\vartheta_{\delta}^{x}-\vartheta_{\delta}^{y}|dxdydt \rightarrow 0 \quad \text{uniformly in}~\delta\quad\text{as}~h\rightarrow 0,
\end{split}
\end{equation}  
with $\mathcal{K}_{h}$ the periodic symmetric kernel given by $(\ref{mathcalK})$. By the structures of the mass equations, we turn to estimate the term associated with $\dive u_{\delta}^{x}-\dive u_{\delta}^{y}$, which can be decomposed into the pressure part and the effective viscous flux part.

The key pressure part can be analyzed as follows:
\begin{equation}\label{ErPdelta}
\begin{split}
&P_{\delta}(\rho_{\delta}^{x},n_{\delta}^{x})-P_{\delta}(\rho_{\delta}^{y},n_{\delta}^{y})=P_{\delta}(\rho_{\delta}^{x},n_{\delta}^{x})-P_{\delta}(A_{\rho^x,n^x}\vartheta _{\delta}^x,B_{\rho^x,n^x}\vartheta _{\delta}^x)\\
&\quad\quad\quad\quad\quad\quad\quad\quad\quad\quad~+P_{\delta}(A_{\rho^y,n^y}\vartheta _{\delta}^y,B_{\rho^y,n^y}\vartheta _{\delta}^y)-P_{\delta}(\rho_{\delta}^{y},n_{\delta}^{y})\\
&\quad\quad\quad\quad\quad\quad\quad\quad\quad\quad~+P_{\delta}(A_{\rho^x,n^x}\vartheta _{\delta}^y,B_{\rho^x,n^x}\vartheta _{\delta}^y)-P_{\delta}(A_{\rho^y,n^y}\vartheta _{\delta}^y,B_{\rho^y,n^y}\vartheta _{\delta}^y)\\
&\quad\quad\quad\quad\quad\quad\quad\quad\quad\quad~+P_{\delta}(A_{\rho^x,n^x}\vartheta _{\delta}^x,B_{\rho^x,n^x}\vartheta _{\delta}^x)-P_{\delta}(A_{\rho^x,n^x}\vartheta _{\delta}^y,B_{\rho^x,n^x}\vartheta _{\delta}^y),
\end{split}
\end{equation}
where $(A_{\rho,n},B_{\rho,n})$ is defined by
\begin{equation}\label{AB}
(A_{\rho,n},B_{\rho,n}):=
\begin{cases}
(\frac{\rho}{\rho+n},\frac{n}{\rho+n}),
& \mbox{if $\rho+n>0,$ } \\
(0,0),
& \mbox{if $\rho+n=0.$}
\end{cases}
\end{equation}
We can prove that for some domain $\mathbb{Q}\subset \mathbb{T}^{2d}\times(0,T)$ with $\mathbb{T}^{2d}\times(0,T)/\mathbb{Q}$ arbitrarily small, the first three terms on the right-hand side of (\ref{ErPdelta}) on $\mathbb{Q}$ tend to $0$ as $h\rightarrow 0$ uniformly with respect to $\delta$ and the last term associated with $P_{\delta}(A_{\rho,n} \vartheta_{\delta}, B_{\rho,n}\vartheta_{\delta})$ in the one variable $\vartheta_{\delta}$ can be estimated by the arguments as used in\cite{bresch1} with some modifications.

Meanwhile, to deal with the effective viscous flux part, we make use of the structures of the momentum equation and the commutator estimates of Riesz operator, which is different from the previous analysis \cite{bresch1} and may be applicable to other related models in fluid dynamics.

\vspace{2ex}

The rest part of this paper is arranged as follows. The global existence of weak solutions for the approximate system with two small parameters $\varepsilon\in(0,1)$ and $\delta\in (0,1)$ is obtained in Section 2. In Section 3, we pass to the limit as $\varepsilon\rightarrow 0$ by the compactness theorem of Lions-Feireisl. In Section 4, we prove the strong convergence of two densities by deriving quantitative regularity estimate of their sum and show the convergence of approximate sequence to an expected weak solution as $\delta\rightarrow 0$. In Appendix, we recall and prove some technical lemmas which are used in this paper.

\section{Faedo–Galerkin approximation}

 We are ready to solve the following approximate system for $\varepsilon\in(0,1)$ and $\delta\in(0,1)$:
\begin{equation}\label{twoapp}
\left\{
\begin{split}
&\partial_{t}\rho+\dive(\rho u)=\varepsilon \Delta \rho,\\
&\partial_{t}n+\dive(n u)=\varepsilon\Delta n,\\
&\partial_{t}\big{(}(\rho+n) u\big{)}+\dive\big{(}(\rho+n)  u\otimes u\big{)}+\nabla P_{\delta}(\rho,n)+\varepsilon\nabla u\cdot\nabla (\rho+n)\\
&\quad=\mu\Delta u+(\mu+\lambda)\nabla\dive u,\quad x\in\mathbb{T}^{d},\quad t>0,
\end{split}
\right.
\end{equation}
with the initial data
\begin{equation}\label{dapp}
\begin{split}
(\rho,n, u)(x,0)=(\rho_{0,\delta},n_{0,\delta},u_{0,\delta})(x),\quad x\in\mathbb{T}^{d},
\end{split}
\end{equation}
where $P_{\delta}(\rho,n)$ is the artificial pressure
 \begin{equation}\label{Pvar}
\begin{split}
P_{\delta}(\rho,n):=\mathbf{1}_{\rho+n\geq \delta}P(\rho,n)+\delta (\rho+n)^{p_{0}},\quad p_{0}>\gamma+\widetilde{\gamma}+\alpha+\widetilde{\alpha}+1,
\end{split}
\end{equation}
and $(\rho_{0,\delta},n_{0,\delta},u_{0,\delta})$ is the regularized initial data 
 \begin{equation}\label{dd}
 \begin{split}
 (\rho_{0,\delta},n_{0,\delta}, u_{0,\delta})(x):=(\rho_{0}\ast j_{\delta}+\delta, n_{0}\ast j_{\delta}+\delta, \frac{\frac{m_{0}}{\sqrt{\rho_{0}+n_{0}}}\ast j_{\delta}}{\sqrt{\rho_{0}\ast j_{\delta}+n_{0}\ast j_{\delta}+2\delta}} )(x).
 \end{split}
 \end{equation}
 Here $\mathbf{1}_{s\geq k}$ is a cut-off function satisfying
 \begin{equation}\nonumber
\mathbf{1}_{s\leq k}:=
\begin{cases}
1,
& \mbox{if\quad $0\leq s \leq k,$ } \\
\text{smooth},
& \mbox{if\quad$k\leq s\leq 2k,$}\\
0,
& \mbox{if\quad$s\geq 2k$,}
\end{cases}
\quad \quad \mathbf{1}_{s\geq k}:=1-\mathbf{1}_{s\leq k},
\end{equation}
and $j_{\delta}$ is a function such that
 \begin{equation}\label{jdelta}
 \left\{
 \begin{split}
 &j_{\delta}\in C^{\infty}(\mathbb{T}^{d}),\quad \|j_{\delta}\|_{L^1 }=1,\quad 0\leq j_{\delta} \leq \delta ^{-\frac{1}{2p_{0}}},\\
 & j_{\delta}\ast f\rightarrow f \quad\text{in}~ L^{p}(\mathbb{T}^{d})\quad\text{as}\quad \delta\rightarrow 0,\quad\forall f\in L^{p}(\mathbb{T}^{d}) ,\quad p\in [1,\infty).
 \end{split}
 \right.
 \end{equation}
 It is easy to verify under the assumptions of either Theorem \ref{theorem11} or Theorem \ref{theorem12} that $(\rho_{0,\delta},n_{0,\delta}, (\rho_{0,\delta}+n_{0,\delta})u_{0,\delta})$ converges to $(\rho_{0},u_{0}, m_{0})$ strongly in $L^{\gamma}(\mathbb{T}^{d})\times L^{\alpha}(\mathbb{T}^{d})\times L^{\frac{2\min\{\gamma,\alpha\}}{\min\{\gamma,\alpha\}+1}}(\mathbb{T}^{d})$ and satisfies
\begin{equation}\label{dappu}
\left\{
\begin{split}
&\|\rho_{0,\delta}\|_{L^{\gamma}}\leq \|\rho_{0}\|_{L^{\gamma}},\quad \|n_{0,\delta}\|_{L^{\alpha}}\leq \|n_{0}\|_{L^{\alpha}},\quad \|\sqrt{\rho_{0,\delta}+ n_{0,\delta}} u_{0,\delta}\|_{L^2}\leq \|\frac{m_{0}}{\sqrt{\rho_{0}+n_{0}}}\|_{L^2},\\
& 0<\delta\leq \rho_{0,\delta}(x),n_{0,\delta}(x)\leq C\delta^{-\frac{1}{2p_{0}}},\quad  \frac{1}{c_{*,\delta}}\leq \frac{n_{0,\delta}(x)}{\rho_{0,\delta}(x)}\leq c_{*,\delta},\quad~ x\in\mathbb{T}^{d},\quad c_{*,\delta}:=C\delta^{-\frac{1}{2p_{0}}-1}>1,\\
&\underline{c}\rho_{0,\delta}(x)\leq n_{0,\delta}(x)\leq \overline{c}\rho_{0,\delta}(x),\quad\text{if}~~ \underline{c}\rho_{0}(x)\leq n_{0}(x)\leq \overline{c} \rho_{0}(x),\quad x\in\mathbb{T}^{d}.
\end{split}
\right.
\end{equation}

Without loss of generalization, we assume that $P(\rho,n)$ satisfies
\begin{align}
P(\rho,n)\in C^2(\mathbb{R}_{+}\times\mathbb{R}_{+}),\quad\quad |\partial_{nn}^2 P(\rho,n)|\leq C_{\delta}(1+\rho^{p_{0}-3}+n^{p_{0}-3}),\label{pppaa}
\end{align}
where $C_{\delta}>0$ is a constant independent of $\varepsilon$.  Indeed, if this is not the case, one may approximate $P(\rho,n)$ by a regular sequence with respect to the parameter $\delta$. We can compute a constant $c_{\delta}\rightarrow \infty$ as $\delta\rightarrow 0$ such that both $\partial_{\rho}P_{\delta}(\rho,n)> 0$ and $\partial_{n}P_{\delta}(\rho,n)> 0$ hold for any $(\rho,n)\in\mathbb{R}_{+}\times\mathbb{R}_{+}$ satisfying $\rho+n\geq c_{\delta}>0$ and then set set
\begin{equation}\label{P1P2}
\left\{
\begin{split}
&P_{1,\delta}(\rho,n):=P_{\delta}(\rho,n)+C_{*}\mathbf{1}_{\rho+n\leq C_{*}c_{\delta}}\big{(} (\rho+n)^{\widetilde{\gamma}+\widetilde{\alpha}}+\rho+n\big{)},\\
&P_{2,\delta}(\rho,n):=C_{*}\mathbf{1}_{\rho+n\leq C_{*}c_{\delta}}\big{(} (\rho+n)^{\widetilde{\gamma}+\widetilde{\alpha}}+\rho+n\big{)},
\end{split}
\right.
\end{equation}
with $C_{*}>1$ a sufficiently large constant. Then the pressure $P_{\delta}(\rho,n)$ can be re-written as
\begin{equation}\label{PP1P2}
\begin{split}
P_{\delta}(\rho,n)=P_{1,\delta}(\rho,n)-P_{2,\delta}(\rho,n).
\end{split}
\end{equation}
Note that $P_{1,\delta}(\rho,n)\in C^2(\mathbb{R}_{+}\times\mathbb{R}_{+})$ is monotonically increasing with respect to both $\rho\geq0$ and $n\geq0$ and $P_{2,\delta}(\rho,n)\in C^{\infty}(\mathbb{R}_{+}\times\mathbb{R}_{+})$ satisfies $P_{2,\delta}(\rho,n)\geq0$ and $P_{2,\delta}(\rho,n)=0$ for $\rho+n\geq 2C_{*}c_{\delta}$, so we can employ the classical Lions-Feireisl approach \cite{lions2,feireisl2} to solve the approximate system (\ref{twoapp}).

Let $\psi_{\ell}\in C^{\infty}(\mathbb{T}^{d})$ $(\ell=1,2,…)$ be an orthonormal basis in $L^2$, and denote the finite dimensional spaces $X_{\ell}\subset L^2$ and its projector $\mathbf{P}_{\ell}: L^2\rightarrow X_{\ell}$ by
\begin{equation}\nonumber
\begin{split}
&X_{\ell}:=\text{span}\{\psi_{j}\}_{j=1}^{\ell},\quad \mathbf{P}_{\ell}f:=\sum_{i=1}^{\ell} \psi_{i}\int_{\mathbb{T}^{d}} \psi_{i} fdx,\quad\forall f\in X_{\ell}. 
\end{split}
\end{equation}
\begin{prop}\label{prop21}
Let $p_{0}>\gamma+\widetilde{\gamma}+\alpha+\widetilde{\alpha}+1$, $\delta\in (0,1)$, $\varepsilon\in (0,1)$ and integer $\ell\geq 1$. Then, under the assumptions of either Theorem \ref{theorem11} or Theorem \ref{theorem12}, there exists a unique global regular solution $(\rho_{\ell},n_{\ell}, u_{\ell})$ for $(\ref{twoapp})$ with the initial data $(\rho_{0,\delta},n_{0,\delta},\mathbf{P}_{\ell} u_{0,\delta})$. In addition, it holds for any $T>0$ that
\begin{equation}\label{rk1}
\left\{
\begin{split}
&\|(\rho_{\ell},n_{\ell})\|_{L^{\infty}(0,T;L^{p_{0}})}+\varepsilon^{\frac{1}{2}}\| (\nabla \rho_{\ell},\nabla n_{\ell})\|_{L^2(0,T;L^2)}\leq C_{\delta},\\
&\|\sqrt{\rho_{\ell}+n_{\ell}}u_{\ell}\|_{L^{\infty}(0,T;L^2)}+\|u_{\ell}\|_{L^2(0,T;H^1)}\leq C_{\delta},\\
&0\leq \frac{1}{c_{*,\delta}}\rho_{\ell}(x,t)\leq n_{\ell}(x,t)\leq c_{*,\delta}\rho_{\ell}(x,t),\quad \text{a.e.}~(x,t)\in \mathbb{T}^{d}\times(0,T),\\
&0\leq\underline{c}\rho_{\ell}(x,t)\leq n_{\ell}(x,t)\leq \overline{c}\rho_{\ell}(x,t), \quad\text{if}~~\underline{c}\rho_{0}(x)\leq n_{0}(x)\leq \overline{c}\rho_{0}(x),\quad\text{a.e.}~(x,t)\in \mathbb{T}^{d}\times(0,T),
\end{split}
\right.
\end{equation}
where  $C_{\delta}>0$ is a constant independent of $\varepsilon$ and $\ell$ and $c_{*,\delta}>1$ is the constant given by $(\ref{dappu})_{2}$.

 Furthermore, we have
\begin{equation}\label{energyk}
\begin{split}
&\int_{\mathbb{T}^{d}}[\frac{1}{2}(\rho_{\ell}+n_{\ell})|u_{\ell}|^2+G_{\delta}(\rho_{\ell},n_{\ell})]dx+\int_{0}^{t}\int_{\mathbb{T}^d}[\mu|\nabla u_{\ell}|^2+(\mu+\lambda)(\dive{u_{\ell}})^2]dxd\tau\\
&\quad \leq \int_{\mathbb{T}^{d}}[\frac{1}{2}(\rho_{0,\delta}+n_{0,\delta})|u_{0,\delta}|^2+G_{\delta}(\rho_{0,\delta},n_{0,\delta})]dx+C_{\delta}\varepsilon,
\end{split}
\end{equation}
where $G_{\delta}(\rho,n)$ is defined by
\begin{eqnarray}
\label{Gdelta} G_{\delta}(\rho,n):=(\rho+n)\int_{1}^{\rho+n}P(\frac{\rho}{\rho+n} s,\frac{n}{\rho+n}s)s^{-2}ds=\rho\int_{\frac{\rho}{\rho+n}}^{\rho}P(s, \frac{n}{\rho} s)s^{-2}ds ,
\end{eqnarray}
\end{prop}
\begin{proof}
By the Faedo-Galerkin approximation and fixed point arguments \cite{feireisl1,feireisl3,novotny1}, there is a time $T_{\ell}\in (0,T]$ such that the system $(\ref{twoapp})$ with the initial data $(\rho_{0,\delta},n_{0,\delta},\mathbf{P}_{\ell} u_{0,\delta})$ can be solved uniquely on $[0,T_{\ell}]$. The details are omitted here.

To show $T_{\ell}=T$, we need to establish the a-priori estimates (\ref{rk1}) on $\mathbb{T}^{d}\times (0,T_{\ell})$ uniformly in $\ell$. Note that $u_{\ell}$ is regular since all the Sobolev norms in the finite dimensional space $X_{\ell}$ are equivalent. By the comparison principle for $(\ref{twoapp})_{1}$-$(\ref{twoapp})_{2}$, we have
\begin{equation}\label{rk11}
\begin{split}
&0\leq \frac{1}{c_{*,\delta}}\rho_{\ell}(x,t)\leq n_{\ell}(x,t)\leq c_{*,\delta}\rho_{\ell}(x,t),\quad  (x,t)\in\mathbb{T}^{d}\times (0,T_{\ell}).
\end{split}
\end{equation}
Multiplying $(\ref{two})_{1}$, $(\ref{two})_{2}$ and $(\ref{two})_{3}$ by $\partial_{\rho_{\ell}}G_{\delta}(\rho_{\ell},n_{\ell})$, $\partial_{n_{\ell}}G_{\delta}(\rho_{\ell},n_{\ell})$ and $u_{\ell}$, respectively, and using the fact
$$
\rho\partial_{\rho}G_{\delta}(\rho,n)+n\partial_{n}G_{\delta}(\rho,n)-G_{\delta}(\rho,n)=P_{\delta}(\rho,n),
$$
we show for $t\in[0,T_{\ell}]$ that
\begin{equation}\label{energyk1}
\begin{split}
&\frac{d}{dt}\int_{\mathbb{T}^d}\Big{(} \frac{1}{2}(\rho_{\ell}+n_{\ell}) |u_{\ell}|^2+G_{\delta}(\rho_{\ell},n_{\ell})\Big{)}dx\\
&\quad+\int_{\mathbb{T}^d}\Big{(}\mu|\nabla u_{\ell}|^2+(\mu+\lambda)(\dive{u_{\ell}})^2+\varepsilon\delta p_{0}  \rho_{\ell}^{p_{0}-2}|\nabla \rho_{\ell}|^2+\varepsilon\delta p_{0} n_{\ell}^{p_{0}-2}|\nabla n_{\ell}|^2\Big{)} dx\\
&=-\varepsilon\int_{\mathbb{T}^{d}}\Big{(}\partial^2_{\rho_{\ell}\rho_{\ell}}G_{\delta}(\rho_{\ell},n_{\ell})|\nabla \rho_{\ell}|^2+2\partial_{\rho_{\ell}n_{\ell}}^2 G_{\delta}(\rho_{\ell},n_{\ell})\nabla \rho_{\ell}\cdot \nabla n_{\ell}+\partial_{n_{\ell}n_{\ell}}^2G_{\delta}(\rho_{\ell},n_{\ell})|\nabla n_{\ell}|^2\Big{)}dx.
\end{split}
\end{equation}
 It follows from $(\ref{P1})$, $(\ref{Pvar})$, \eqref{pppaa} and $(\ref{rk11})$ that
\begin{equation}\label{energyk11}
\left\{
\begin{split}
&G_{\delta}(\rho_{\ell},n_{\ell})\geq \frac{\delta}{p_{0}-1}(\rho_{\ell}+n_{\ell})^{p_{0}}-C_{\delta},\\
&\big{(} |\partial^2_{\rho_{\ell}\rho_{\ell}}G_{\delta}|+|\partial_{\rho_{\ell}n_{\ell}}^2 G_{\delta}|+|\partial_{n_{\ell}n_{\ell}}^2G_{\delta}|\big{)}(\rho_{\ell},n_{\ell})\leq C_{\delta} (\rho_{\ell}+n_{\ell})^{p_{0}-2}+C_{\delta}.
\end{split}
\right.
\end{equation}
By $(\ref{energyk11})_{2}$ and the Young inequality, the right-hand side of (\ref{energyk1}) can be estimates as
\begin{equation}\label{rightk}
\begin{split}
&-\varepsilon\int_{\mathbb{T}^{d}}\Big{(}\partial^2_{\rho_{\ell}\rho_{\ell}}G_{\delta}(\rho_{\ell},n_{\ell})|\nabla \rho_{\ell}|^2+2\partial_{\rho_{\ell}n_{\ell}}^2 G_{\delta}(\rho_{\ell},n_{\ell})\nabla \rho_{\ell}\cdot \nabla n_{\ell}+\partial_{n_{\ell}n_{\ell}}^2G_{\delta}(\rho_{\ell},n_{\ell})|\nabla n_{\ell}|^2\Big{)}dx\\
&\quad\leq \frac{\varepsilon\delta p_{0} }{2}  \int_{\mathbb{T}^{d}} \big{(} \rho_{\ell}^{p_{0}-2}|\nabla \rho_{\ell}|^2+n_{\ell}^{p_{0}-2}|\nabla n_{\ell}|^2\big{)}dx +C_{\delta}\varepsilon \int_{\mathbb{T}^{d}}(|\nabla \rho_{\ell}|^2+|\nabla n_{\ell}|^2)dx+C_{\delta}\varepsilon.
\end{split}
\end{equation}
To control the second term on the right-hand side of $(\ref{rightk})$, we deduce from $(\ref{twoapp})_{1}$-$(\ref{twoapp})_{2}$ that
\begin{equation}\label{energyk1111}
\begin{split}
&C_{\delta}\frac{d}{dt}\int_{\mathbb{T}^{d}}(|\rho_{\ell}|^2+|n_{\ell}|^2)dx+2C_{\delta}\varepsilon\int_{\mathbb{R}^{d}}(|\nabla \rho_{\ell}|^2+|\nabla n_{\ell}|^2)dx\\
&\quad=-2C_{\delta}\int_{\mathbb{T}^{d}} (\rho_{\ell}+n_{\ell})\dive u_{\ell}dx\\
&\quad\leq \frac{1}{2}\int_{\mathbb{T}^{d}}\big{(}     \mu|\nabla u_{\ell}|^2+(\mu+\lambda)(\dive{u_{\ell}})^2    \big{)}dx+C_{\delta} \int_{\mathbb{T}^{d}}(|\rho_{\ell}|^2+|n_{\ell}|^2)dx.
\end{split}
\end{equation}
Adding $(\ref{energyk1})$ and (\ref{rightk})-(\ref{energyk1111}) together and making use of (\ref{mu}), $(\ref{dappu})$, $(\ref{energyk11})_{1}$ and the Gr${\rm{\ddot{o}}}$nwall inequality, we obtain
\begin{equation}\label{r111k}
\begin{split}
&\sup_{t\in[0,T]}\int_{\mathbb{T}^d}[ (\rho_{\ell}+n_{\ell}) |u_{\ell}|^2+\rho_{\ell}^{p_{0}}+n_{\ell}^{p_{0}}]dx\\
&\quad+\int_{0}^{T}\int_{\mathbb{T}^{d}}[|\nabla u_{\ell}|^2+\varepsilon(1+\rho_{\ell}^{p_{0}-2}+n_{\ell}^{p_{0}-2})(|\nabla \rho_{\ell}|^2+|\nabla n_{\ell}|^2)]dxdt\leq C_{\delta},
\end{split}
\end{equation}
which together with \cite[Lemma 3.2]{feireisl3} leads to
\begin{equation}\label{ukL2}
\begin{split}
\|u_{\ell}\|_{L^2(0,T;L^2)}\leq C_{\delta}.
\end{split}
\end{equation}
By the Sobolev inequality, it also holds that
\begin{equation}\label{r112k} 
\begin{split}
&\|\rho_{\ell}+n_{\ell}\|_{L^{p_{0}+1}(0,T;L^{p_{0}+1})}^{p_{0}-1}=\|(\rho_{\ell}+n_{\ell})^{p_{0}-1}\|_{L^{\frac{p_{0}+1}{p_{0}-1}}(0,T;L^{\frac{p_{0}+1}{p_{0}-1}})}\\
&\quad\quad\quad\quad\quad\quad\quad\quad\quad\quad\quad~\leq \|(\rho_{\ell}+n_{\ell})^{p_{0}-1}\|_{L^{1}(0,T;L^{1})}+C\|\nabla (\rho_{\ell}+n_{\ell})^{p_{0}-1}\|_{L^2(0,T;L^2)}.
\end{split}
\end{equation}
The combination of (\ref{mu}), (\ref{rk11}) and (\ref{r111k})-(\ref{r112k}) gives rise to $T_{\ell}=T$ and $(\ref{rk1})$. By $(\ref{energyk1})$, (\ref{rightk}) and $(\ref{r111k})$, one can show $(\ref{energyk})$. The proof of Proposition \ref{prop21} is complete.
\end{proof}

With the help of Proposition \ref{prop21} and standard arguments (cf. \cite{novotny2,vasseur1}), we have the following global existence of weak solutions to the IVP $(\ref{twoapp})$-$(\ref{dapp})$.
\begin{prop}\label{prop22}
Let $p_{0}>\gamma+\widetilde{\gamma}+\alpha+\widetilde{\alpha}+1$, $\delta\in (0,1)$ and $\varepsilon\in (0,1)$. Then, under the assumptions of either Theorem \ref{theorem11} or Theorem \ref{theorem12}, there exists a global weak solution $(\rho_{\varepsilon}, n_{\varepsilon}, (\rho_{\varepsilon}+n_{\varepsilon})u_{\varepsilon})$ to the IVP $(\ref{twoapp})$-$(\ref{dapp})$ satisfying for any $T>0$ that
\begin{equation}\label{rvar1}
\left\{
\begin{split}
&\|(\rho_{\varepsilon},n_{\varepsilon})\|_{L^{\infty}(0,T;L^{p_{0}} )}+\varepsilon^{\frac{1}{2}}\| (\nabla \rho_{\varepsilon}, \nabla n_{\varepsilon})\|_{L^2(0,T;L^2)}\leq C_{\delta},\\
&\|\sqrt{\rho_{\varepsilon}+n_{\varepsilon}}u_{\varepsilon}\|_{L^{\infty}(0,T;L^2)}+\|u_{\varepsilon}\|_{L^2(0,T;H^1)}\leq C_{\delta},\\
&\|(\rho_{\varepsilon},n_{\varepsilon})\|_{L^{p_{0}+1}(0,T;L^{p_{0}+1} )}\leq C_{\delta},\\
&0\leq \frac{1}{c_{*,\delta}}\rho_{\varepsilon}(x,t)\leq n_{\varepsilon}(x,t)\leq c_{*,\delta}\rho_{\varepsilon}(x,t),\quad\text{a.e.}~(x,t)\in\mathbb{T}^{d}\times(0,T),\\
&0\leq \underline{c}\rho_{\varepsilon}(x,t)\leq n_{\varepsilon}(x,t)\leq \overline{c}\rho_{\varepsilon}(x,t), \quad\text{if}~~\underline{c}\rho_{0}(x)\leq n_{0}(x)\leq \overline{c}\rho_{0}(x),\quad\text{a.e.}~(x,t)\in\mathbb{T}^{d}\times(0,T),
\end{split}
\right.
\end{equation}
where $C_{\delta}>0$ is a constant independent of $\varepsilon$ and $c_{*,\delta}>1$ is the constant given by $(\ref{dappu})_{2}$.

Furthermore, we have
\begin{equation}\label{energyvar}
\begin{split}
&\int_{\mathbb{T}^{d}}[\frac{1}{2}(\rho_{\varepsilon}+n_{\varepsilon})|u_{\varepsilon}|^2+G_{\delta}(\rho_{\varepsilon},n_{\varepsilon})]dx+\int_{0}^{t}\int_{\mathbb{T}^d}[\mu|\nabla u_{\varepsilon}|^2+(\mu+\lambda)(\dive{u_{\varepsilon}})^2]dxd\tau\\
&\quad \leq \int_{\mathbb{T}^{d}}[\frac{1}{2}(\rho_{0,\delta}+n_{0,\delta})|u_{0,\delta}|^2+G_{\delta}(\rho_{0,\delta},n_{0,\delta})]dx+C_{\delta}\varepsilon,\quad \text{a.e.} ~t\in (0,T),
\end{split}
\end{equation}
with $G_{\delta}(\rho,n)$ defined by $(\ref{Gdelta})$.
\end{prop}

\section{Vanishing artificial viscosity}

In this section, we aim to study the following IVP: 
\begin{equation}\label{twodelta}
\left\{
\begin{split}
&\partial_{t}\rho+\dive(\rho u)=0,\\
&\partial_{t}n+\dive(n u)=0,\\
&\partial_{t}\big{(}(\rho+n) u\big{)}+\dive\big{(}(\rho+n)  u\otimes u\big{)}+\nabla P_{\delta}(\rho,n)=\mu\Delta u+(\mu+\lambda)\nabla\dive u,\quad x\in\mathbb{T}^{d},\quad t>0,
\end{split}
\right.
\end{equation}
with the initial data
\begin{equation}\label{ddelta}
\begin{split}
(\rho,n, u)(x,0)=(\rho_{0,\delta},n_{0,\delta},u_{0,\delta})(x),\quad x\in\mathbb{T}^{d},
\end{split}
\end{equation}
where $P_{\delta}(\rho,n)$ and $(\rho_{0,\delta},n_{0,\delta},u_{0,\delta})$ are given by $(\ref{Pvar})$ and $(\ref{dd})$, respectively.

First, by (\ref{rvar1}) and standard compactness arguments, we have the following convergence of weak solutions to the IVP $(\ref{twodelta})$-$(\ref{ddelta})$.

\begin{lemma}\label{lemma31}
Let $T>0$, $p_{0}>\gamma+\widetilde{\gamma}+\alpha+\widetilde{\alpha}+1$ and $(\rho_{\varepsilon},n_{\varepsilon},(\rho_{\varepsilon}+n_{\varepsilon})u_{\varepsilon})$ be the weak solution to the IVP $(\ref{twodelta})$-$(\ref{ddelta})$ for $\varepsilon\in(0,1)$ given by Proposition \ref{prop22}. Then, under the assumptions of either Theorem \ref{theorem11} or Theorem \ref{theorem12}, there is a limit $(\rho,n,(\rho+n)u)$ such that as $\varepsilon\rightarrow 0$, it holds up to a subsequence (still denoted by $(\rho_{\varepsilon},n_{\varepsilon},(\rho_{\varepsilon}+n_{\varepsilon})u_{\varepsilon}))$ that
\begin{equation}\label{limitvar1}
\left\{
\begin{split}
&(\rho_{\varepsilon}, n_{\varepsilon})\rightharpoonup (\rho, n)\quad\quad\quad\quad\text{in}~L^{p_{0}+1}(0,T;L^{p_{0}+1}(\mathbb{T}^{d}) )\times L^{p_{0}+1}(0,T;L^{p_{0}+1}(\mathbb{T}^{d}) ),\\
&\varepsilon (\nabla\rho_{\varepsilon}, \nabla n_{\varepsilon})\rightarrow 0\quad\quad\quad~~\text{in}~L^2(0,T;L^2(\mathbb{T}^{d})),\\
&u_{\varepsilon}\rightharpoonup u\quad\quad\quad\quad\quad\quad\quad\quad\text{in}~L^2(0,T;H^1(\mathbb{T}^{d}) ),\\
&(\rho_{\varepsilon}, n_{\varepsilon})\rightarrow (\rho, n)\quad\quad\quad\quad\text{in}~C([0,T];L^{p_{0}}_{weak}(\mathbb{T}^{d}) )\times C([0,T];L^{p_{0}}_{weak}(\mathbb{T}^{d}) ),\\
&(\rho_{\varepsilon}, n_{\varepsilon})\rightarrow (\rho, n)\quad\quad\quad\quad\text{in}~C([0,T];H^{-1}(\mathbb{T}^{d}))\times C([0,T];H^{-1}(\mathbb{T}^{d}) ),\\
&(\rho_{\varepsilon}+n_{\varepsilon}) u_{\varepsilon}\rightarrow (\rho+n) u\quad\text{in}~C([0,T];L^{\frac{2p_{0}}{p_{0}+1}}_{weak}(\mathbb{T}^{d}) )\cap C([0,T];H^{-1}(\mathbb{T}^{d}) ).
\end{split}
\right.
\end{equation}
\end{lemma}

Before showing the strong convergence of two densities, we need to have the following lemma.
\begin{lemma}\label{lemma33}
Let $T>0$, $(\rho_{\varepsilon},n_{\varepsilon},(\rho_{\varepsilon}+n_{\varepsilon})u_{\varepsilon})$ be the weak solution to the IVP $(\ref{twoapp})$-$(\ref{dapp})$ for $\varepsilon\in(0,1)$ given by Proposition \ref{prop22}, and $(\rho,n,(\rho+n)u)$ be the limit obtained by Lemma \ref{lemma31}. Then, under the assumptions of either Theorem \ref{theorem11} or Theorem \ref{theorem12}, it holds
\begin{equation}\label{strongABvar}
\left\{
\begin{split}
&\lim_{\varepsilon\rightarrow0}\int_{0}^{T}\int_{\mathbb{T}^{d}}(\rho_{\varepsilon}+n_{\varepsilon})|A_{\rho_{\varepsilon},n_{\varepsilon}}-A_{\rho,n}|^{p}dxdt=0,\quad\quad p\in[1,\infty),\\
&\lim_{\varepsilon\rightarrow0}\int_{0}^{T}\int_{\mathbb{T}^{d}}(\rho_{\varepsilon}+n_{\varepsilon})|B_{\rho_{\varepsilon},n_{\varepsilon}}-B_{\rho,n}|^{p}dxdt=0,\quad\quad p\in[1,\infty),
\end{split}
\right.
\end{equation}
where $(A_{\rho,n},B_{\rho,n})$ is defined by $(\ref{AB})$.

Furthermore, as $\varepsilon\rightarrow0$, we have
\begin{equation}\label{rhonifvar}
\begin{split}
&\quad\quad~(\rho_{\varepsilon},n_{\varepsilon})\rightarrow (\rho,n)\quad\text{in}~L^1(0,T;L^1 (\mathbb{T}^{d}))\times L^1(0,T;L^1(\mathbb{T}^{d}) )\\
&\Longleftrightarrow \rho_{\varepsilon}+n_{\varepsilon}\rightarrow \rho+n\quad\text{in}~L^1(0,T;L^1(\mathbb{T}^{d}) ).
\end{split}
\end{equation}
\end{lemma}

\begin{proof}
The proof of (\ref{strongABvar}) can be found in \cite{vasseur1}. Denote $F_{\sigma}(\rho,n):=\frac{\rho^2}{\rho+n+\sigma}$ for any $\sigma>0$ so that it holds
\begin{equation}\label{F}
\left\{
\begin{split}
&|\rho\partial_{\rho} F_{\sigma}(\rho,n)+n\partial_{n} F_{\sigma}(\rho,n)-F_{\sigma}(\rho,n)|=\frac{\sigma \rho^2}{(\rho+n+\sigma)^2}\leq \sigma,\\
& \partial_{\rho \rho}^2 F_{\sigma}(\rho,n)|\nabla \rho|^2+2\partial_{\rho n}^2F_{\sigma}(\rho,n) \nabla \rho\cdot \nabla n+\partial_{n n}^2 F_{\sigma}(\rho,n)|\nabla n|^2\geq0,\\
&|\partial_{\rho} F_{\sigma}(\rho,n)|+|\partial_{n} F_{\sigma}(\rho,n)|\leq 4.
\end{split}
\right.
\end{equation}
 Let $\eta_{\sigma}\in C^{\infty}_{c}(\mathbb{T}^{d}) $ for any $\sigma>0$ be the Friedrichs mollifier. Applying $\eta_{\sigma}\ast $ to both sides of $(\ref{twoapp})_{1}$-$(\ref{twoapp})_{2}$, we obtain
\begin{equation}\label{renrho}
\left\{
\begin{split}
&\partial_{t}\eta_{\sigma}\ast\rho_{\varepsilon}+\dive\big{(} (\eta_{\sigma}\ast\rho_{\varepsilon}) u_{\varepsilon}\big{)}=\varepsilon \Delta (\eta_{\sigma}\ast\rho_{\varepsilon})+r_{1,\sigma},\\
& r_{1,\sigma}:=\dive\big{(}(\eta_{\sigma}\ast\rho_{\varepsilon}) u_{\varepsilon}\big{)}-\eta_{\sigma}\ast\dive (\rho_{\varepsilon} u_{\varepsilon}),\\
&\partial_{t}\eta_{\sigma}\ast n_{\varepsilon}+\dive\big{(}(\eta_{\sigma}\ast n_{\varepsilon}) u_{\varepsilon}\big{)}=\varepsilon \Delta (\eta_{\sigma}\ast n_{\varepsilon})+r_{2,\sigma},\\
&r_{2,\sigma}:=\dive\big{(}(\eta_{\sigma}\ast n_{\varepsilon}) u_{\varepsilon}\big{)}-\eta_{\sigma}\ast\dive (n_{\varepsilon} u_{\varepsilon}).
\end{split}
\right.
\end{equation}
By the commutator estimates of Friedrichs mollifier (cf. \cite{diperna1,lions1}), the terms $r_{1,\sigma}$ and $r_{2,\sigma}$ converge to $0$ strongly in $L^1(0,T;L^1(\mathbb{T}^{d}) )$. Multiplying $(\ref{renrho})_{1}$ and $(\ref{renrho})_{3}$ by $\partial_{\eta_{\sigma}\ast\rho_{\varepsilon}}F_{\sigma}(\eta_{\sigma}\ast\rho_{\varepsilon},\eta_{\sigma}\ast n_{\varepsilon})$ and $\partial_{\eta_{\sigma}\ast n_{\varepsilon}}F_{\sigma}(\eta_{\sigma}\ast\rho_{\varepsilon},\eta_{\sigma}\ast n_{\varepsilon})$, respectively, adding the resulting equations together and then making use of (\ref{F}), we have
\begin{equation}\nonumber
\begin{split}
&\partial_{t}F_{\sigma}(\eta_{\sigma}\ast\rho_{\varepsilon},\eta_{\sigma}\ast n_{\varepsilon})+\dive \big{(} F_{\sigma}(\eta_{\sigma}\ast\rho_{\varepsilon},\eta_{\sigma}\ast n_{\varepsilon}) u_{\varepsilon}\big{)} \\
&\quad\leq  \sigma|\dive u_{\varepsilon}|+ \varepsilon\Delta F_{\sigma}(\eta_{\sigma}\ast\rho_{\varepsilon},\eta_{\sigma}\ast n_{\varepsilon})+4(|r_{1,\sigma}|+|r_{2,\sigma}|) ,
\end{split}
\end{equation}
which gives rise to
\begin{equation}\label{Fsigmalim}
\begin{split}
&\int_{0}^{T}\int_{\mathbb{T}^{d}}F_{\sigma}(\eta_{\sigma}\ast\rho_{\varepsilon},\eta_{\sigma}\ast n_{\varepsilon})dxdt\\
&\quad\leq T\int_{\mathbb{T}^{d}} F_{\sigma}(\eta_{\sigma}\ast\rho_{0,\delta},\eta_{\sigma}\ast n_{0,\delta})dx+\sigma T^{\frac{3}{2}}\|\dive u_{\varepsilon}\|_{L^2(0,T;L^2)}+4T\|(r_{1,\sigma},r_{2,\sigma})\|_{L^1(0,T;L^1 )}.
\end{split}
\end{equation}
By the dominated convergence theorem, we take the limit as $\sigma\rightarrow 0$ in (\ref{Fsigmalim}) to derive
\begin{equation}\label{rho2n}
\begin{split}
&\int_{0}^{T}\int_{\mathbb{T}^{d}}\frac{\rho_{\varepsilon}^2}{\rho_{\varepsilon}+n_{\varepsilon}}dxdt\leq T\int_{\mathbb{T}^{d}} \frac{\rho_{0,\delta}^2}{\rho_{0,\delta}+ n_{0,\delta}}dx.
\end{split}
\end{equation}
Similarly, one has
\begin{equation}\label{rho2n1}
\begin{split}
&\int_{0}^{T}\int_{\mathbb{T}^{d}}\frac{\rho^2}{\rho+n}dxdt= T\int_{\mathbb{T}^{d}} \frac{\rho_{0,\delta}^2}{\rho_{0,\delta}+ n_{0,\delta}}dx.
\end{split}
\end{equation}
 By using (\ref{limitvar1}), $(\ref{rho2n})$-$(\ref{rho2n1})$ and
 \begin{equation}\label{AB1}
 \left\{
 \begin{split}
 &\rho=(\rho+n)A_{\rho,n},\quad \frac{\rho^2}{\rho+n}=\rho A_{\rho,n}= (\rho+n) A_{\rho,n}^2,\\
 &n=(\rho+n)B_{\rho,n},\quad \frac{n^2}{\rho+n}=n B_{\rho,n}=(\rho+n) B_{\rho,n}^2,
 \end{split}
 \right.
 \end{equation}
 we obtain
\begin{equation}\nonumber
\begin{split}
& \lim_{\varepsilon\rightarrow 0}\int_{0}^{T}\int_{\mathbb{T}^{d}}(\rho_{\varepsilon}+n_{\varepsilon})|A_{\rho_{\varepsilon},n_{\varepsilon}}-A_{\rho,n}|^2dxdt\\
&\quad=\lim_{\varepsilon\rightarrow 0}\int_{0}^{T}\int_{\mathbb{T}^{d}}\Big{(}\frac{\rho_{\varepsilon}^2}{\rho_{\varepsilon}+n_{\varepsilon}}-2\rho_{\varepsilon}A_{\rho,n}+(\rho_{\varepsilon}+n_{\varepsilon})A_{\rho,n}^2\Big{)}dxdt\\
&\quad \leq \int_{0}^{T}\int_{\mathbb{T}^{d}}\Big{(}\frac{\rho^2}{\rho+n}-2\rho A_{\rho,n}+(\rho+n)A_{\rho,n}^2\Big{)}dxdt =0,
\end{split}
\end{equation}
which together with the H$\rm{\ddot{o}}$lder inequality and the fact
 \begin{equation}\label{AB2}
 \begin{split}
 & 0\leq A_{\rho,n}, B_{\rho,n}\leq 1
 \end{split}
 \end{equation}
leads to $(\ref{strongABvar})_{1}$. Similarly, one can verify $(\ref{strongABvar})_{2}$.

Then, we are ready to prove $(\ref{rhonifvar})$. Obviously, as $\varepsilon\rightarrow0$, the strong convergence of both $\rho_{\varepsilon}$ and $n_{\varepsilon}$ in $L^1(0,T;L^1(\mathbb{T}^{d}))$ implies the  strong convergence of $\rho_{\varepsilon}+n_{\varepsilon}$ in $L^1(0,T;L^1(\mathbb{T}^{d}))$. Assume $\rho_{\varepsilon}+n_{\varepsilon}\rightarrow \rho+n$ in $L^1(0,T;L^1(\mathbb{T}^{d}) )$ as $\varepsilon\rightarrow 0$. By virtue of $(\ref{strongABvar})_{1}$ for $p=1$ and $(\ref{AB1})$-$(\ref{AB2})$, we have
\begin{equation}\nonumber
\begin{split}
&\lim_{\varepsilon\rightarrow0}\|\rho_{\varepsilon}-\rho\|_{L^1(0,T;L^1 )}\\
&\quad=\lim_{\varepsilon\rightarrow0}\|(\rho_{\varepsilon}+n_{\varepsilon})A_{\rho_{\varepsilon},n_{\varepsilon}}-(\rho_{\varepsilon}+n_{\varepsilon})A_{\rho,n}+(\rho_{\varepsilon}+n_{\varepsilon})A_{\rho,n}-(\rho+n)A_{\rho,n}\|_{L^1(0,T;L^1 )}\\
&\quad\leq \lim_{\varepsilon\rightarrow0} \big{(} \|\rho_{\varepsilon}-(\rho_{\varepsilon}+n_{\varepsilon})A_{\rho,n}\|_{L^1(0,T;L^1 )}+\|\rho_{\varepsilon}+n_{\varepsilon}-\rho-n\|_{L^1(0,T;L^1 )}\big{)}= 0,
\end{split}
\end{equation}
where we have used $\rho=(\rho+n)A_{\rho,n}$. Similarly, one can show that $n_{\varepsilon}$ converges to $n$ strongly in $L^1(0,T;L^1(\mathbb{T}^{d}) )$ as $\varepsilon\rightarrow 0$. The proof of Lemma \ref{lemma33} is complete.
\end{proof}

We consider the cut-off function
 \begin{equation}\label{Tk}
T_{k}(s):=
\begin{cases}
s,
& \mbox{if\quad$0\leq s \leq k,$ } \\
\text{smooth and concave},
& \mbox{if\quad$k\leq s\leq 3k,$}\\
2k,
& \mbox{if\quad$s\geq 3k$,}
\end{cases}
\end{equation}
which was introduced in \cite{feireisl1}  and references therein. As in \cite{feireisl1,lions2}, we can prove the following property of effect viscous flux.
\begin{lemma}\label{lemma32}
Let $T>0$, $\delta\in(0,1)$, $(\rho_{\varepsilon},n_{\varepsilon},(\rho_{\varepsilon}+n_{\varepsilon})u_{\varepsilon})$ be the weak solution to the IVP $(\ref{twodelta})$-$(\ref{ddelta})$ for $\varepsilon\in(0,1)$ given by Proposition \ref{prop22}, and $(\rho,n,(\rho+n)u)$ be the limit obtained by Lemma \ref{lemma31}. Then, under the assumptions of either Theorem \ref{theorem11} or Theorem \ref{theorem12}, for any $k>0$, it holds up to a subsequence (still labelled by $(\rho_{\varepsilon},n_{\varepsilon},(\rho_{\varepsilon}+n_{\varepsilon})u_{\varepsilon}) )$ that
\begin{equation}\label{effectvar}
\begin{split}
&\lim_{\varepsilon\rightarrow 0}\int_{0}^{T}\int_{\mathbb{T}^{d}}\big{(}P_{\delta}(\rho_{\varepsilon},n_{\varepsilon})-(2\mu+\lambda)\dive u_{\varepsilon}\big{)}T_{k}(\rho_{\varepsilon}+n_{\varepsilon})dxdt\\
&~~=\int_{0}^{T}\int_{\mathbb{T}^{d}}\big{(}\overline{P_{\delta}(\rho,n)}-(2\mu+\lambda)\dive u\big{)}\overline{T_{k}(\rho+n)}dxdt,
\end{split}
\end{equation}
where $\overline{f}$ denotes the weak limit of $f_{\varepsilon}$ as $\varepsilon\rightarrow0$.
\end{lemma}

Finally, we prove the strong convergence of two densities.
\begin{lemma}\label{lemma34}
Let $T>0$, $p_{0}>\gamma+\widetilde{\gamma}+\alpha+\widetilde{\alpha}+1$, $(\rho_{\varepsilon},n_{\varepsilon},(\rho_{\varepsilon}+n_{\varepsilon})u_{\varepsilon})$ be the weak solution to the IVP $(\ref{twoapp})$-$(\ref{dapp})$ for $\varepsilon\in(0,1)$ given by Proposition \ref{prop22}, and $(\rho,n,(\rho+n)u)$ be the limit obtained by Lemma \ref{lemma31}. Then, under the assumptions of either Theorem \ref{theorem11} or Theorem \ref{theorem12}, as $\varepsilon\rightarrow0$, it holds up to a subsequence (still denoted by $(\rho_{\varepsilon},n_{\varepsilon}))$ that
\begin{equation}\label{strongpnvar}
\begin{split}
&(\rho_{\varepsilon}, n_{\varepsilon})\rightarrow (\rho, n)\quad\text{in}~L^{p}(0,T;L^{p}(\mathbb{T}^{d}))\times L^{p}(0,T;L^{p}(\mathbb{T}^{d})),\quad p\in[1, p_{0}+1).
\end{split}
\end{equation}
\end{lemma}
\begin{proof}
Let $\vartheta_{\varepsilon}:=\rho_{\varepsilon}+n_{\varepsilon}$, $\vartheta:=\rho+n$, and $b_{\sigma}(s):=(s+\sigma)\log{(s+\sigma)}$ for $s\geq0$ and $\sigma>0$. By $(\ref{renrho})$, one has
\begin{equation}\label{varthetavar}
\begin{split}
\partial_{t}(\eta_{\sigma}\ast\vartheta_{\varepsilon})+\dive\big{(} (\eta_{\sigma}\ast\vartheta_{\varepsilon}) u_{\varepsilon}\big{)}=\varepsilon \Delta (\eta_{\sigma}\ast\vartheta_{\varepsilon})+r_{1,\sigma}+r_{2,\sigma},
\end{split}
\end{equation}
where $\eta_{\sigma}\in C^{\infty}_{c}(\mathbb{T}^{d}) $ for $\sigma>0$ is the Friedrichs mollifier and $r_{i,\sigma}$ $(i=1,2)$ are given by $(\ref{renrho})$. Multiplying $(\ref{varthetavar})$ by $b'(\eta_{\sigma}\ast\vartheta_{\varepsilon})=1+\log{(\eta_{\sigma}\ast\vartheta_{\varepsilon}+\sigma)}$, we get
\begin{equation}\label{341}
\begin{split}
&\partial_{t}\big{(} b(\eta_{\sigma}\ast\vartheta_{\varepsilon}) \big{)}+\dive\big{(} b(\eta_{\sigma}\ast\vartheta_{\varepsilon}) u_{\varepsilon}\big{)}+\big{(}b'(\eta_{\sigma}\ast\vartheta_{\varepsilon})\eta_{\sigma}\ast\vartheta_{\varepsilon}-b(\eta_{\sigma}\ast\vartheta_{\varepsilon}) \big{)} \dive u_{\varepsilon}\\
&\quad\leq \varepsilon \Delta \big{(} b(\eta_{\sigma}\ast\vartheta_{\varepsilon}) \big{)}+b'(\eta_{\sigma}\ast\vartheta_{\varepsilon})(r_{1,\sigma}+r_{2,\sigma}).
\end{split}
\end{equation}
We integrate $(\ref{341})$ over $\mathbb{T}^{d}\times[0,t]$ for $t\in[0,T]$ and take the limit as $\sigma\rightarrow 0$ to derive
\begin{equation}\label{342}
\begin{split}
&\int_{\mathbb{T}^{d}}\vartheta_{\varepsilon}\log{\vartheta_{\varepsilon}}dx\leq \int_{\mathbb{T}^{d}}(\rho_{0,\delta}+n_{0,\delta})\log{(\rho_{0,\delta}+n_{0,\delta})}dx-   \int_{0}^{t}\int_{\mathbb{T}^{d}}\vartheta_{\varepsilon}\dive u_{\varepsilon} dxd\tau,\quad\text{a.e.}~t\in(0,T).
\end{split}
\end{equation}
Similarly, one has
\begin{equation}\label{343}
\begin{split}
&\int_{\mathbb{T}^{d}} \vartheta\log{\vartheta}dx=\int_{\mathbb{T}^{d}}(\rho_{0,\delta}+n_{0,\delta})\log{(\rho_{0,\delta}+n_{0,\delta})}dx-\int_{0}^{t}\int_{\mathbb{T}^{d}}\vartheta\dive u dxd\tau,\quad~~\quad\text{a.e.}~t\in(0,T).
\end{split}
\end{equation}
Let $\overline{f}$ denote the weak limit of $f_{\varepsilon}$ as $\varepsilon\rightarrow0$. Let $P_{1,\delta}(\rho,n)$, $P_{2,\delta}(\rho,n)$, $(A_{\rho,n},B_{\rho,n})$ and $T_{k}(f)$ be defined by $(\ref{P1P2})_{1}$, $(\ref{P1P2})_{2}$, $(\ref{AB})$ and $(\ref{Tk})$, respectively. It can be shown by (\ref{PP1P2}), (\ref{AB1}), (\ref{effectvar}) and (\ref{342})-(\ref{343}) that
\begin{equation}\label{3433}
\begin{split}
&\int_{\mathbb{T}^{d}}(\overline{\vartheta \log{\vartheta} }-\vartheta\log{\vartheta})dx\\
&\quad\leq  \lim_{\varepsilon\rightarrow 0}\int_{0}^{t}\int_{\mathbb{T}^{d}}(\vartheta\dive u-\vartheta_{\varepsilon} \dive u_{\varepsilon} )dxd\tau=\sum_{i=1}^{4}I_{i}^{1},\quad\text{a.e.}~t\in(0,T),
\end{split}
\end{equation}
where $I_{i}^4$ $(i=1,2,3,4)$ are given by
\begin{equation}\nonumber
\begin{split}
&I_{1}^1:=\lim_{\varepsilon\rightarrow 0} \int_{0}^{t}\int_{\mathbb{T}^{d}}\big{[} \big{(} \vartheta-\overline{T_{k}(\vartheta)} \big{)}\dive u-\big{(} \vartheta_{\varepsilon} -T_{k}(\vartheta_{\varepsilon}) \big{)}\dive u_{\varepsilon}\big{]}dxd\tau,\\
&I_{2}^{1}:=\frac{1}{2\mu+\lambda}\lim_{\varepsilon\rightarrow 0}\int_{0}^{t}\int_{\mathbb{T}^{d}}\big{[}P_{\delta}(\rho_{\varepsilon},n_{\varepsilon})-P_{\delta}(   A_{\rho,n}\vartheta_{\varepsilon}  ,  B_{\rho,n}\vartheta_{\varepsilon} )\big{]} \big{(} \overline{T_{k}(\vartheta)}-T_{k}(\vartheta_{\varepsilon}) \big{)} dxd\tau,\\
&I_{3}^1:=\frac{1}{2\mu+\lambda}\lim_{\varepsilon\rightarrow 0}\int_{0}^{t}\int_{\mathbb{T}^{d}} \big{[}P_{1,\delta}(   A_{\rho,n}\vartheta_{\varepsilon}  ,  B_{\rho,n}\vartheta_{\varepsilon} ) -P_{1,\delta}(\rho, n) \big{]}(\overline{T_{k} \big{(} \vartheta)}-T_{k}(\vartheta_{\varepsilon}) \big{)}dxd\tau,\\
&I_{4}^1:=\frac{1}{2\mu+\lambda}\lim_{\varepsilon\rightarrow 0}\int_{0}^{t}\int_{\mathbb{T}^{d}}\big{[}P_{2,\delta}(   A_{\rho,n}\vartheta_{\varepsilon}  ,  B_{\rho,n}\vartheta_{\varepsilon} ) -P_{2,\delta}(\rho, n) \big{]}\big{(} T_{k}(\vartheta_{\varepsilon})-\overline{T_{k}(\vartheta)} \big{)}dxd\tau.
\end{split}
\end{equation}
It follows from (\ref{rvar1}), (\ref{limitvar1}) and lower semi-continuity of the weak limit $\vartheta-\overline{T_{k}(\vartheta)}$ that
\begin{equation}\nonumber
\begin{split}
& I_{1}^1\leq \underset{\varepsilon\rightarrow 0}{\sup\lim}~\big{(} \|\vartheta-\overline{T_{k}(\vartheta)}\|_{L^2(0,T;L^2)}\|\dive u\|_{L^2(0,T;L^2)}+\|\vartheta_{\varepsilon} -T_{k}(\vartheta_{\varepsilon})\|_{L^2(0,T;L^2)}\|\dive u_{\varepsilon}\|_{L^2(0,T;L^2)} \big{)}\\
&\quad\leq C_{\delta}~\underset{\varepsilon\rightarrow 0}{\sup\lim}~ \|\vartheta_{\varepsilon} -T_{k}(\vartheta_{\varepsilon})\|_{L^2(0,T;L^2)}\\
&\quad\leq \frac{C_{\delta} \|\vartheta_{\delta}\|_{L^{p_{0}+1}(0,T;L^{p_{0}+1})}^{\frac{p_{0}+1}{2}}}{k^{\frac{p_{0}-1}{2}}} \leq \frac{C_{\delta}}{k^{\frac{p_{0}-1}{2}}}. 
\end{split}
\end{equation}
For any $x_{1},x_{2},y_{1},y_{2}\geq 0$ and $F(x,y)\in C^{1}(\mathbb{R}_{+}\times\mathbb{R}_{+})$, we have
\begin{equation}
\begin{split}
&F(x_{1},y_{1})-F(x_{2},y_{2})=(x_{1}-x_{2})\int_{0}^{1}\partial_{x}F\big{(} x_{2}+\theta(x_{1}-x_{2}),y_{2}+\theta(y_{1}-y_{2}) \big{)}d\theta\\
&\quad\quad\quad\quad\quad\quad\quad\quad\quad~+(y_{1}-y_{2})\int_{0}^{1} \partial_{y}F\big{(} x_{2}+\theta(x_{1}-x_{2}),y_{2}+\theta(y_{1}-y_{2}) \big{)}d\theta,\label{P11}
\end{split}
\end{equation}
which together with  (\ref{P}), $(\ref{Pvar})$, $(\ref{rvar1})_{4}$ and $(\ref{strongABvar})$ yields
\begin{equation}\label{I11}
\begin{split}
&I_{2}^{1}\leq Ck\lim_{\varepsilon\rightarrow 0}\int_{0}^{t}\int_{\mathbb{T}^{d}}(\vartheta_{\varepsilon}+1)^{p_{0}-1}\vartheta_{\varepsilon}^{1-\frac{1}{p_{0}}}\vartheta_{\varepsilon}^{\frac{1}{p_{0}}}(|A_{\rho_{\varepsilon},n_{\varepsilon}}-A_{\rho,n}|+|B_{\rho_{\varepsilon},n_{\varepsilon}}-B_{\rho,n}|)dxd\tau\\
&\quad\leq Ck~\underset{\varepsilon\rightarrow 0}{\sup\lim}~\Big{(}\int_{0}^{t}\int_{\mathbb{T}^{d}}(\vartheta_{\varepsilon}+1)^{p_{0}+1 }dxd\tau\Big{)}^{\frac{p_{0}-1}{p_{0}}}\\
&\quad\quad\times \lim_{\varepsilon\rightarrow 0}\Big{(}\int_{0}^{t}\int_{\mathbb{T}^{d}}\vartheta_{\varepsilon} \big{(} |A_{\rho_{\varepsilon},n_{\varepsilon}}-A_{\rho,n}|^{p_{0}}+|B_{\rho_{\varepsilon},n_{\varepsilon}}-B_{\rho,n}|^{p_{0}} \big{)} dxd\tau\Big{)}^{\frac{1}{p_{0}}}=0.
\end{split}
\end{equation}
Noticing that both $T_{k}(s)$ and $P_{1,\delta}(A_{\rho,n}s,B_{\rho,n}s)$ are monotone increasing with respect to $s\geq 0$, we apply \cite[Theorem 10.19]{feireisl4} to have
\begin{equation}\label{I12}
\begin{split}
&I_{3}^{1}\leq 0.
\end{split}
\end{equation}
Substituting the above estimates of $I_{i}^1$ $(i=1,2,3)$ into (\ref{3433}) and taking the limit as $k\rightarrow\infty$, we deduce
\begin{equation}\label{34331}
\begin{split}
&\int_{\mathbb{T}^{d}}(\overline{\vartheta \log{\vartheta} }-\vartheta\log{\vartheta})dx\\
&\quad\leq\frac{1}{2\mu+\lambda}\lim_{\varepsilon\rightarrow 0}\int_{0}^{t}\int_{\mathbb{T}^{d}}\big{[}P_{2,\delta}(   A_{\rho,n}\vartheta_{\varepsilon}  ,  B_{\rho,n}\vartheta_{\varepsilon} ) -P_{2,\delta}(\rho, n) \big{]}\big{(}\vartheta_{\varepsilon}-\vartheta \big{)}dxd\tau\\
&\quad= \frac{1}{2\mu+\lambda}\lim_{\varepsilon\rightarrow 0}\int_{0}^{t}\int_{\mathbb{T}^{d}}\Big( P_{2,\delta}(A_{\rho,n}\vartheta_{\varepsilon} , B_{\rho,n}\vartheta_{\varepsilon} )\vartheta_{\varepsilon}-P_{2,\delta}(A_{\rho,n}\vartheta, B_{\rho,n}\vartheta) \vartheta\\
&\quad\quad\quad\quad\quad\quad\quad\quad\quad\quad+\big{[}P_{2,\delta}(A_{\rho,n}\vartheta, B_{\rho,n}\vartheta) -P_{2,\delta}(A_{\rho,n}\vartheta_{\varepsilon} , B_{\rho,n}\vartheta_{\varepsilon} )\big{]}\vartheta \Big)dxd\tau.
\end{split}
\end{equation}
Since $P_{2,\delta}(sA_{\rho,n},sB_{\rho,n})$ has compact support in $\{s\in\mathbb{R}_{+} ~|~ s\leq 2C_{*}c_{\delta}\}$ due to $(\ref{P1P2})_{2}$, there is a constant $C_{*,\delta}>0$ such that both $C_{*,\delta} s\log{s}-sP_{2,\delta}(A_{\rho,n} s,B_{\rho,n} s)$ and $C_{*,\delta} s\log{s}+P_{2,\delta}(A_{\rho,n} s,B_{\rho,n} s)$ are strictly convex for any $s\geq0$ and therefore we get
\begin{equation}\label{I14}
\begin{split}
&\lim_{\varepsilon\rightarrow 0}\int_{0}^{t}\int_{\mathbb{T}^{d}}  \big{[} P_{2,\delta}(A_{\rho,n}\vartheta_{\varepsilon} , B_{\rho,n}\vartheta_{\varepsilon} )\vartheta_{\varepsilon}-P_{2,\delta}(\rho, n)\vartheta\big{]}dxd\tau\\
&\quad\leq C_{*,\delta} \int_{0}^{t}\int_{\mathbb{T}^{d}}(\overline{\vartheta \log{\vartheta} }-\vartheta\log{\vartheta})dxd\tau,
\end{split}
\end{equation}
and
\begin{equation}\label{I15}
\begin{split}
&\lim_{\varepsilon\rightarrow 0}\int_{0}^{t}\int_{\mathbb{T}^{d}} \big{[} P_{2,\delta}(A_{\rho,n}\vartheta, B_{\rho,n}\vartheta) -P_{2,\delta}(A_{\rho,n}\vartheta_{\varepsilon} , B_{\rho,n}\vartheta_{\varepsilon} )\big{]}\vartheta dxd\tau\\
&\quad\leq \lim_{\varepsilon\rightarrow 0}\int_{\{(x,\tau)\in\mathbb{T}^{d}\times[0,t]~ |~ \vartheta\leq 2C_{*}c_{\delta}\}}\big{[}P_{2,\delta}(A_{\rho,n}\vartheta, B_{\rho,n}\vartheta)- P_{2,\delta}(A_{\rho,n}\vartheta_{\varepsilon} , B_{\rho,n}\vartheta_{\varepsilon} )\big{]}\vartheta dxd\tau\\
&\quad\leq C_{*,\delta} \lim_{\varepsilon\rightarrow 0}\int_{\{(x,\tau)\in\mathbb{T}^{d}\times[0,t]~ |~ \vartheta\leq 2 C_{*}c_{\delta}\}}(\vartheta_{\varepsilon}\log{\vartheta_{\varepsilon}}-\vartheta\log{\vartheta})\vartheta dxd\tau\\
&\quad\leq 2 C_{*,\delta}C_{*}c_{\delta}  \int_{0}^{t}\int_{\mathbb{T}^{d}}(\overline{\vartheta \log{\vartheta} }-\vartheta\log{\vartheta})dxd\tau,
\end{split}
\end{equation}
where one has used (\ref{AB1}), $P_{2,\delta}(sA_{\rho,n},sB_{\rho,n})\geq 0$, the compact support of $P_{2,\delta}(sA_{\rho,n},sB_{\rho,n})$ and the convexity of $ C_{*,\delta} s\log{s}-P_{2,\delta}(sA_{\rho,n},sB_{\rho,n})$. By (\ref{34331})-(\ref{I15}), we have
\begin{equation}\label{345}
\begin{split}
&\int_{\mathbb{T}^{d}}(\overline{\vartheta \log{\vartheta} }-\vartheta\log{\vartheta})dx\leq  \frac{C_{*,\delta}(2C_{*}c_{\delta}+1)}{2\mu+\lambda}\int_{0}^{t}\int_{\mathbb{T}^{d}}(\overline{\vartheta \log{\vartheta} }-\vartheta\log{\vartheta})dx d\tau.
\end{split}
\end{equation}
The combination of $(\ref{rvar1})_{4}$, $(\ref{rhonifvar})$, (\ref{345}), the Gr${\rm{\ddot{o}}}$nwall inequality and the convexity of  $s\log{s}$ gives rise to $(\ref{strongpnvar})$.  The proof of Lemma \ref{lemma34} is complete.
\end{proof}

With the help of Proposition \ref{prop22} and Lemmas \ref{lemma31}-\ref{lemma34}, we have the global existence of weak solutions to the IVP (\ref{twodelta})-(\ref{ddelta}).
\begin{prop}\label{prop31}
Let $p_{0}>\gamma+\widetilde{\gamma}+\alpha+\widetilde{\alpha}+1$ and $\delta\in (0,1)$. Then, under the assumptions of either Theorem \ref{theorem11} or Theorem \ref{theorem12}, there exists a global weak solution $(\rho_{\delta}, n_{\delta}, (\rho_{\delta}+n_{\delta})u_{\delta})$ to the IVP $(\ref{twodelta})$-$(\ref{ddelta})$ satisfying for any $T>0$ that
\begin{equation}\label{rdelta1}
\left\{
\begin{split}
&\|\rho_{\delta}\|_{L^{\infty}(0,T;L^{ \gamma})}+\|n_{\delta}\|_{L^{\infty}(0,T;L^{ \alpha})}+\delta^{\frac{1}{p_{0}}}\|(\rho_{\delta},n_{\delta})\|_{L^{\infty}(0,T;L^{p_{0}})}\leq C,\\
&\|\sqrt{\rho_{\delta}+n_{\delta}}u_{\delta}\|_{L^{\infty}(0,T;L^2)}+\|u_{\delta}\|_{L^2(0,T;H^1)}\leq C,\\
&\rho_{\delta}(x,t)\geq 0,\quad n_{\delta}(x,t)\geq 0,\quad \text{a.e.}~~(x,t)\in\mathbb{T}^{d}\times(0,T),\\
&\underline{c}\rho_{\delta}(x,t)\leq n_{\delta}(x,t)\leq \overline{c}\rho_{\delta}(x,t),\quad\text{if}~~\underline{c}\rho_{0}(x)\leq n_{0}(x)\leq \overline{c}\rho_{0}(x),\quad \text{a.e.}~(x,t)\in\mathbb{T}^{d}\times(0,T),
\end{split}
\right.
\end{equation}
where $C>0$ is a constant independent of $\delta$.

Furthermore, it holds
\begin{equation}\label{energydelta}
\begin{split}
&\int_{\mathbb{T}^{d}}[\frac{1}{2}(\rho_{\delta}+n_{\delta})|u_{\delta}|^2+G_{\delta}(\rho_{\delta},n_{\delta})]dx+\int_{0}^{t}\int_{\mathbb{T}^d}[\mu|\nabla u_{\delta}|^2+(\mu+\lambda)(\dive{u_{\delta}})^2]dxd\tau\\
&\quad \leq\int_{\mathbb{T}^{d}}[\frac{1}{2}(\rho_{0,\delta}+n_{0,\delta})|u_{0,\delta}|^2+G_{\delta}(\rho_{0,\delta},n_{0,\delta})]dx,\quad \text{a.e.} ~t\in (0,T),
\end{split}
\end{equation}
where $G_{\delta}(\rho,n)$ is given by $(\ref{Gdelta})$.
\end{prop}

\section{Vanishing artificial pressure }

\subsection{Higher integrability of densities}

In this section, we prove Theorem \ref{theorem11} and Theorem \ref{theorem12} by taking the limit in $(\ref{twodelta})$ as $\delta\rightarrow 0$.
Since the uniform estimate $(\ref{rdelta1})$ are not enough to show the convergence of $P_{\delta}(\rho_{\delta},n_{\delta})$ to $P(\rho,n)$ as $\delta\rightarrow 0$, we need the higher integrability of $\rho_{\delta}$ and $n_{\delta}$ uniformly in $\delta$ as follows.
\begin{lemma}\label{lemma41}
Let $T>0$, $p_{0}>\gamma+\widetilde{\gamma}+\alpha+\widetilde{\alpha}+1$, $(\rho_{\delta},n_{\delta},(\rho_{\delta}+n_{\delta})u_{\delta})$ be the weak solution to the IVP $(\ref{twodelta})$-$(\ref{ddelta})$ for $\delta\in(0,1)$ given by Proposition \ref{prop31}, and $\vartheta_{\delta}:=\rho_{\delta}+n_{\delta}$. Then, under the assumptions of Theorem \ref{theorem11}, it holds
\begin{equation}\label{rhonLp}
\begin{split}
&\int_{0}^{T}\int_{\mathbb{T}^{d}}(\vartheta_{\delta}^{\max\{\gamma,\alpha\}+\theta}+\delta \vartheta_{\delta}^{p_{0}+\theta})dxdt\leq C,
\end{split}
\end{equation}
with $C>0$ a constant independent of $\delta$ and
$$
\theta:=\frac{2}{d}\max\{\gamma,\alpha\}-1>0.
$$

In addition, under the assumptions of Theorem \ref{theorem12}, we have
\begin{equation}\label{rhonLpc}
\begin{split}
&\int_{0}^{T}\int_{\mathbb{T}^{d}}(\rho_{\delta}^{\gamma+\theta_{1}}+n_{\delta}^{\alpha+\theta_{2}}+\delta \vartheta_{\delta}^{p_{0}+\theta_{1}}+\delta \vartheta_{\delta}^{p_{0}+\theta_{2}})dxdt\leq C,
\end{split}
\end{equation}
with
$$
\theta_{1}:=\frac{2}{d}\gamma-\frac{\gamma}{\min{\{\gamma,\alpha\}}}>0,\quad\quad \theta_{2}:=\frac{2}{d}\alpha-\frac{\alpha}{\min{\{\gamma,\alpha\}}}>0.
$$
\end{lemma}
\begin{proof}
We are ready to show (\ref{rhonLpc}) and (\ref{rhonLp}) can be proved by similar arguments. Denote by $F_{\delta}$ the effect viscous flux as
\begin{equation}\label{effect1}
\begin{split}
&F_{\delta}:=P_{\delta}(\rho_{\delta},n_{\delta})-\frac{1}{|\mathbb{T}^d|}\int_{\mathbb{T}^d}P_{\delta}(\rho_{\delta},n_{\delta})dx-(2\mu+\lambda)\dive u_{\delta}.
\end{split}
\end{equation}
One deduces from the equation $(\ref{twodelta})_{3}$ that
\begin{equation}\label{effect11}
\begin{split}
&F_{\delta}=(-\Delta)^{-1}\dive [\partial_{t} (\vartheta_{\delta} u_{\delta})+\dive(\vartheta_{\delta}u_{\delta}\otimes u_{\delta})],
\end{split}
\end{equation}
where the operator $(-\Delta)^{-1}: W^{k-2,p}(\mathbb{T}^{d}) \rightarrow W^{k,p}(\mathbb{T}^{d})$ for $k\in\mathbb{R}$ and $p\in(1,\infty)$ is defined by
\begin{equation}\label{deltam1}
\begin{split}
&(-\Delta)^{-1} f=g,\quad g~\text{is the solution of the problem}~-\Delta g=f,\quad \int_{\mathbb{T}^{d}}f=0.
\end{split}
\end{equation}
It can be verified by the Sobolev inequality and $L^{p}(\mathbb{T}^{d})$-$L^{p}(\mathbb{T}^{d})$ boundedness of the operator $\nabla(-\Delta)^{-1} \dive $ that
\begin{equation}\label{LpLp}
\begin{split}
&\| (-\Delta)^{-1}\dive f\|_{L^{\frac{pd}{d-p}}}\leq C\|\nabla (-\Delta )^{-1}\dive f\|_{L^{p}}\leq C\|f\|_{L^{p}},\quad\forall f\in L^{p},\quad p\in(1,d).
\end{split}
\end{equation}
By (\ref{twodelta}) and (\ref{effect11}), we obtain
\begin{equation}\label{411}
\begin{split}
&\int_{0}^{T}\int_{\mathbb{T}^{d}}\rho^{\theta_{1}}_{\delta}P_{\delta}(\rho_{\delta},n_{\delta})dxdt =\sum_{i=1}^{4} I_{i}^2,
\end{split}
\end{equation}
with
\begin{equation}\nonumber
\begin{split}
&I_{1}^{2}:=\int_{0}^{T}\int_{\mathbb{T}^{d}}\rho^{\theta_{1}}_{\delta}\Big{(}\frac{1}{|\mathbb{T}^{d}|}\int_{\mathbb{T}^d}P_{\delta}(\rho_{\delta},n_{\delta})dx+(2\mu+\lambda) \dive u_{\delta} \Big{)}dxdt,\\
&I_{2}^{2}:=\int_{\mathbb{T}^d} \rho^{\theta_{1}}_{\delta} (-\Delta)^{-1}\dive ( \vartheta_{\delta}u_{\delta})dx\Big{|}_{0}^{T},\\
&I_{3}^{2}:=(\theta-1)\int_{0}^{T}\int_{\mathbb{T}^{d}} \rho_{\delta}^{\theta_{1}}\dive u_{\delta} (-\Delta)^{-1} \dive (\vartheta_{\delta} u_{\delta})dxdt,\\
&I_{4}^{2}:=\sum_{i,j=1}^{d}\int_{0}^{T}\int_{\mathbb{T}^{d}}\rho^{\theta_{1}}_{\delta} [\mathcal{R}_{i}\mathcal{R}_{j}( \vartheta_{\delta}u_{\delta}^{i}u_{\delta}^{j})-u_{\delta}^{j}\mathcal{R}_{i}\mathcal{R}_{j}(\vartheta_{\delta} u_{\delta}^{i})]dxdt,\quad R_{i}:=(-\Delta)^{-\frac{1}{2}}\partial_{i},\quad i=1,..., d.
\end{split}
\end{equation}
First, one deduces by 
$(\ref{rdelta1})$ and $2\theta_{1}<\gamma+\theta_{1}$ that
\begin{equation}\nonumber
\begin{split}
&I_{1}^{2}\leq T\|\rho_{\delta}^{\theta_{1}}\|_{L^{\infty}(0,T;L^{1})}\|P_{\delta}(\rho_{\delta},n_{\delta})\|_{L^{\infty}(0,T;L^{1})}+(2\mu+\lambda)T^{\frac{1}{2}}\|\rho_{\delta}^{\theta_{1}}\|_{L^{2}(0,T;L^{2})}\|\dive u_{\delta}\|_{L^2(0,T;L^2)}\\
&\quad\leq C(1+\frac{1}{\varepsilon_{0}})+\varepsilon_{0}\int_{0}^{T}\int_{\mathbb{T}^{d}}\rho_{\delta}^{\gamma+\theta_{1}}dxdt,
\end{split}
\end{equation}
with the constant $\varepsilon_{0}>0$ to be chosen later. By $(\ref{rdelta1})$, one has
\begin{equation}\label{r0}
\begin{split}
&\|\vartheta_{\delta}\|_{L^{\infty}(0,T;L^{\gamma_{0}})}+\|\vartheta_{\delta} u_{\delta}\|_{L^{\infty}(0,T;L^{\frac{2\gamma_{0}}{\gamma_{0}+1}})}\leq C,
\end{split}
\end{equation}
where we used the convenient notation
\begin{equation}\nonumber
\begin{split}
\gamma_{0}:=\min\{\gamma,\alpha\}.
\end{split}
\end{equation}
For the term $I_{2}^2$, making use of $(\ref{dappu})$, $(\ref{rdelta1})$, (\ref{LpLp}), (\ref{r0}) and $\frac{1}{d}+1-\frac{\theta_{1}}{\gamma}>\frac{\gamma_{0}+1}{2\gamma_{0}}$, we have
\begin{equation}\nonumber
\begin{split}
&I_{2}^2\leq \|\rho_{\delta}^{\theta_{1}}\|_{L^{\infty}(0,T;L^{\frac{\gamma}{\theta_{1}}})}\| (-\Delta)^{-1}\dive ( \vartheta_{\delta}u_{\delta})\|_{L^{\infty}(0,T;L^{\frac{2\gamma_{0}d}{(d-2)\gamma_{0}+d}})}\\
&\quad\quad+\|\rho_{0,\delta}\|_{L^{\frac{ \gamma}{\theta_{1}}}}\| (-\Delta)^{-1}\dive \big{(} (\rho_{0,\delta}+n_{0,\delta})u_{0,\delta}\big{)}\|_{L^{\frac{2\gamma_{0}d}{(d-2)\gamma_{0}+d}}}\\
&\quad\leq C\|\rho_{\delta}\|_{L^{\infty}(0,T;L^{ \gamma })}^{\theta_{1}} \|  \vartheta_{\delta}u_{\delta}\|_{L^{\infty}(0,T;L^{\frac{2\gamma_{0}}{\gamma_{0}+1}})}+C\|\rho_{0,\delta}\|_{L^{ \gamma}}^{\theta_{1}} \|  (\rho_{0,\delta}+n_{0,\delta})u_{0,\delta}\|_{L^{\frac{2\gamma_{0}}{\gamma_{0}+1}} }\leq C.
\end{split}
\end{equation}
As in \cite{lions2}, we analyse the terms $I_{i}^2$ $(i=3,4)$ for two cases $d\geq 3$ and $d=2$ separately:

\vspace{1ex}
 \begin{itemize}
\item{Case 1: $d\geq 3.$}
\end{itemize}
Since it holds $\frac{\theta_{1}}{\gamma}+\frac{1}{\gamma_{0}}+\frac{d-2}{2d}+\frac{d-2}{2d}=1$, we obtain by $(\ref{rdelta1})$, (\ref{r0}), the Sobolev embedding $H^1(\mathbb{R}^{d}) \hookrightarrow L^{\frac{2d}{d-2}}(\mathbb{R}^{d})$ and the boundedness of Riesz operator that
\begin{equation}\nonumber
\begin{split}
&I_{4}^2\leq C\|\rho_{\delta}^{\theta_{1}}\|_{L^{\infty}(0,T;L^{\frac{\gamma}{\theta_{1}}})}\|\vartheta_{\delta}\|_{L^{\infty}(0,T;L^{ \gamma_{0} })}\|u_{\delta}\|_{L^2(0,T;L^{\frac{2d}{d-2}})}^2\leq C.
\end{split}
\end{equation}
Similarly, the term $I_{3}^2$ for $d\geq4$ can be estimated by
\begin{equation}\nonumber
\begin{split}
&I_{3}^2\leq \|\rho_{\delta}^{\theta_{1}}\|_{L^{\infty}(0,T;L^{ \frac{\gamma}{\theta_{1}} })}\|\dive u_{\delta}\|_{L^2(0,T;L^2)}\|(-\Delta)^{-1}\dive (\vartheta_{\delta} u_{\delta})\|_{L^{2}(0,T;L^{ \frac{2\gamma_{0}d}{(d-4)\gamma_{0}+2d} })}\\
&\quad\leq C\|\rho_{\delta}\|_{L^{\infty}(0,T;L^{ \gamma })}^{\theta_{1}}\|\dive u_{\delta}\|_{L^2(0,T;L^2)}\|\vartheta_{\delta}\|_{L^{\infty}(0,T;L^{ \gamma_{0} })}\|u_{\delta}\|_{L^2(0,T;L^{\frac{2d}{d-2}})}\leq C.
\end{split}
\end{equation}
For $d=3$, it follows from (\ref{rdelta1}), (\ref{LpLp}), (\ref{r0}) and the Sobolev embedding $H^1(\mathbb{T}^3) \hookrightarrow L^{6}(\mathbb{T}^3)$ that
\begin{equation}\nonumber
\begin{split}
&I_{3}^2\leq C \|\rho_{\delta}^{\theta_{1}}\|_{L^{\frac{5\gamma_{0}-3}{2\gamma_{0}-3}}(0,T;L^{\frac{5\gamma_{0}-3}{2\gamma_{0}-3}})}\|\dive u_{\delta}\|_{L^2(0,T;L^2)} \|\vartheta_{\delta} u_{\delta}\|_{L^{\frac{2(5\gamma_{0}-3)}{\gamma_{0}+3}}(0,T;L^{ \frac{6(5\gamma_{0}-3)}{13\gamma_{0}+3}})}\\
&\quad\leq C\|\rho_{\delta}^{\theta_{1}}\|_{L^{\frac{\gamma+\theta_{1}}{\theta_{1}}}(0,T;L^{\frac{\gamma+\theta_{1}}{\theta_{1}}})} \|\dive u_{\delta}\|_{L^2(0,T;L^2)} \|\vartheta_{\delta} \|_{L^{\infty}(0,T;L^{\gamma_{0}})}^{\frac{\gamma_{0}+3}{5\gamma_{0}-3}}\|u_{\delta} \|_{L^2(0,T;L^{6})}^{\frac{\gamma_{0}+3}{5\gamma_{0}-3}}\|\vartheta_{\delta} u_{\delta}\|_{L^{\infty}(0,T;L^{\frac{2\gamma_{0}}{\gamma_{0}+1}})}^{\frac{2(2\gamma_{0}-3)}{5\gamma_{0}-3}}\\
&\quad\leq C(1+\frac{1}{\varepsilon_{0}})+\varepsilon_{0}\int_{0}^{T}\int_{\mathbb{T}^{3}}\rho_{\delta}^{\gamma+\theta_{1}}dxdt,
\end{split}
\end{equation}
 \begin{itemize}
\item{Case 2: $d=2.$}
\end{itemize}
One can show by (\ref{rdelta1}) and (\ref{LpLp}) that
\begin{equation}\nonumber
\begin{split}
&I_{3}^{2}\leq C\|\rho_{\delta}^{\theta_{1}}\|_{L^{\frac{2\gamma_{0}-1}{\gamma_{0}-1}}(0,T;L^{\frac{2\gamma_{0}-1}{\gamma_{0}-1}})} \|\dive u_{\delta}\|_{L^{2}(0,T;L^2)} \|\vartheta_{\delta} u_{\delta}\|_{L^{2(2\gamma_{0}-1)}(0,T;L^{\frac{2\gamma_{0}}{\gamma_{0}+1}})}\\
&\quad\leq C\|\rho_{\delta}\|_{L^{\gamma+\theta_{1}}(0,T;L^{\gamma+\theta_{1}})}^{\theta_{1}} \|\dive u_{\delta}\|_{L^{2}(0,T;L^2)}\|\sqrt{\vartheta_{\delta}}\|_{L^{2(2\gamma_{0}-1)}(0,T;L^{2(2\gamma_{0}-1)})} \|\sqrt{\vartheta_{\delta}} u_{\delta}\|_{L^{\infty}(0,T;L^{2})}\\
&\quad\leq C(1+\frac{1}{\varepsilon_{0}})+\varepsilon_{0}\int_{0}^{T}\int_{\mathbb{T}^{2}}(\rho_{\delta}^{\gamma+\theta_{1}}+n_{\delta}^{\alpha+\theta_{2}})dxdt.
\end{split}
\end{equation}
In addition, one deduces by the critical Sobolev embedding $H^{1}(\mathbb{T}^2)\hookrightarrow \mathcal{BMO}(\mathbb{T}^2)$ for $d=2$ and the commutator estimate (\ref{risz1}) that
\begin{equation}\nonumber
\begin{split}
&I_{4}^2\leq \sum_{i,j=1}^{2} \|\rho_{\delta}^{\theta_{1}}\|_{L^{\frac{2\gamma_{0}-1}{\gamma_{0}-1}}(0,T;L^{\frac{2\gamma_{0}-1}{\gamma_{0}-1}})}\|\mathcal{R}_{i}\mathcal{R}_{j}( \vartheta_{\delta}u_{\delta}^{i}u_{\delta}^{j})-u_{\delta}^{j}\mathcal{R}_{i}\mathcal{R}_{j}(\vartheta_{\delta} u_{\delta}^{i})\|_{L^{\frac{2\gamma_{0}-1}{\gamma_{0}}}(0,T;L^{\frac{2\gamma_{0}-1}{\gamma_{0}}})}\\
&\quad\leq C\|\rho_{\delta}\|_{L^{\gamma+\theta_{1}}(0,T;L^{\gamma+\theta_{1}})}^{\theta_{1}}\|u_{\delta}\|_{L^2(0,T;\mathcal{BMO})}\|\vartheta_{\delta} u_{\delta}\|_{L^{2(2\gamma_{0}-1)}(0,T;L^{\frac{2\gamma_{0}-1}{\gamma_{0}}})}\\
&\quad\leq C \|\rho_{\delta}\|_{L^{\gamma+\theta_{1}}(0,T;L^{\gamma+\theta_{1}})}^{\theta_{1}}\|\sqrt{\vartheta_{\delta}}\|_{L^{2(2\gamma_{0}-1)}(0,T;L^{2(2\gamma_{0}-1)})} \|u_{\delta}\|_{L^2(0,T;H^1)}\|\sqrt{\vartheta_{\delta}} u_{\delta}\|_{L^{\infty}(0,T;L^2)}\\
&\quad\leq C(1+\frac{1}{\varepsilon_{0}})+\varepsilon_{0}\int_{0}^{T}\int_{\mathbb{T}^{2}}(\rho_{\delta}^{\gamma+\theta_{1}}+n_{\delta}^{\alpha+\theta_{2}})dxdt.
\end{split}
\end{equation}
Substituting the above estimates of $I_{i}^2$ $(i=1,2,3,4)$ into (\ref{411}), we obtain
\begin{equation}\label{412}
\begin{split}
&\int_{0}^{T}\int_{\mathbb{T}^{d}}\rho^{\theta_{1}}_{\delta}P_{\delta}(\rho_{\delta},n_{\delta})dxdt\leq C(1+\frac{1}{\varepsilon_{0}})+3\varepsilon_{0}\int_{0}^{T}\int_{\mathbb{T}^{d}}(\rho_{\delta}^{\gamma+\theta_{1}}+n_{\delta}^{\alpha+\theta_{2}})dxdt.
\end{split}
\end{equation}
Similarly, one can show
\begin{equation}\label{4121}
\begin{split}
&\int_{0}^{T}\int_{\mathbb{T}^{d}}n^{\theta_{2}}_{\delta}P_{\delta}(\rho_{\delta},n_{\delta})dxdt\leq C(1+\frac{1}{\varepsilon_{0}})+3\varepsilon_{0}\int_{0}^{T}\int_{\mathbb{T}^{d}}(\rho_{\delta}^{\gamma+\theta_{1}}+n_{\delta}^{\alpha+\theta_{2}})dxdt.
\end{split}
\end{equation}
By (\ref{P}) and (\ref{Pvar}), it holds
\begin{equation}\label{413}
\begin{split}
&(\rho_{\delta}^{\theta_{1}}+n_{\delta}^{\theta_{2}})P_{\delta}(\rho_{\delta},n_{\delta})\\
&\quad\geq (\rho_{\delta}^{\theta_{1}}+n_{\delta}^{\theta_{2}})\big{(}P(\rho_{\delta},n_{\delta})+\delta(\rho_{\delta}+n_{\delta})^{p_{0}}\big{)}-\mathbf{1}_{\vartheta_{\delta}\leq \delta}(\rho_{\delta}^{\theta_{1}}+n_{\delta}^{\theta_{2}})P(\rho_{\delta},n_{\delta})\\
&\quad\geq \frac{1}{2C_{0}} (\rho_{\delta}^{\gamma+\theta_{1}}+n_{\delta}^{\alpha+\theta_{2}})+\delta (\rho_{\delta}^{p_{0}+\theta_{1}}+n_{\delta}^{p_{0}+\theta_{2}})-C.
\end{split}
\end{equation}
Combining (\ref{412})-(\ref{413}) together and choosing $\varepsilon_{0}=\frac{1}{24C_{1}}$, we derive $(\ref{rhonLpc})$. The proof of Lemma \ref{lemma41} is complete.
\end{proof}

By (\ref{rdelta1}) and (\ref{rhonLp}), we have the convergence of weak solutions to the IVP $(\ref{twodelta})$-$(\ref{ddelta})$.
\begin{lemma}\label{lemma42} 
Let $T>0$, $p_{0}>\gamma+\widetilde{\gamma}+\alpha+\widetilde{\alpha}+1$, and $(\rho_{\delta},n_{\delta},(\rho_{\delta}+n_{\delta})u_{\delta})$ be the weak solution to the IVP $(\ref{twodelta})$-$(\ref{ddelta})$ for $\delta\in(0,1)$ given by Proposition \ref{prop31}. Then, under the assumptions of either Theorem \ref{theorem11} or Theorem \ref{theorem12}, there is a limit $(\rho,n,(\rho+n)u)$ such that as $\delta\rightarrow0$, it holds up to a subsequence (still denoted by $(\rho_{\delta},n_{\delta},(\rho_{\delta}+n_{\delta})u_{\delta}))$ that
\begin{equation}\label{limitdelta1}
\left\{
\begin{split}
&(\rho_{\delta}, n_{\delta})\overset{\ast}{\rightharpoonup}(\rho, n)\quad\quad\quad\quad\text{in}~L^{\infty}(0,T;L^{\gamma}(\mathbb{T}^{d}) )\times L^{\infty}(0,T;L^{\alpha}(\mathbb{T}^{d}) ),\\
&\delta  (\rho_{\delta}+ n_{\delta})^{p_{0}}\rightarrow 0\quad\quad\quad~~\text{in}~L^1(0,T;L^1(\mathbb{T}^{d}) ),\\
&u_{\delta}\rightharpoonup u\quad\quad\quad\quad\quad\quad\quad\quad\text{in}~L^2(0,T;H^1(\mathbb{T}^{d}) ),\\
&(\rho_{\delta},n_{\delta})\rightarrow (\rho,n) \quad\quad\quad\quad\text{in}~C([0,T];L^{\gamma}_{weak}(\mathbb{T}^{d}) )\times C([0,T];L^{\alpha}_{weak}(\mathbb{T}^{d}) ),\\
&(\rho_{\delta},n_{\delta})\rightarrow (\rho,n)\quad\quad\quad\quad\text{in}~C([0,T];H^{-1} (\mathbb{T}^{d}))\times C([0,T];H^{-1}(\mathbb{T}^{d}) ),\\
&(\rho_{\delta}+n_{\delta}) u_{\delta}\rightarrow (\rho+n) u\quad\text{in}~C([0,T];L^{\frac{2\min\{\gamma,\alpha\}}{\min\{\gamma,\alpha\}+1}}_{weak}(\mathbb{T}^{d}) )\cap C([0,T];H^{-1}(\mathbb{T}^{d}) ).
\end{split}
\right.
\end{equation}
\end{lemma}

\subsection{Strong convergence of densities}

By similar arguments as used in Lemma \ref{lemma33}, the strong convergence of two densities $\rho_{\delta}$ and $n_{\delta}$ and the strong convergence of their sum $\rho_{\delta}+n_{\delta}$ are equivalent.
\begin{lemma}\label{lemma43}
Let $T>0$, $(\rho_{\delta},n_{\delta},(\rho_{\delta}+n_{\delta})u_{\delta})$ be the weak solution to the IVP $(\ref{twodelta})$-$(\ref{ddelta})$ for $\delta\in(0,1)$ given by Proposition \ref{prop31}, and $(\rho,n,(\rho+n)u)$ be the limit obtained by Lemma \ref{lemma42}. Then, under the assumptions of either Theorem \ref{theorem11} or Theorem \ref{theorem12}, as $\delta\rightarrow0$, it holds
\begin{equation}\label{strongABdelta}
\left\{
\begin{split}
&\rho_{\delta}-(\rho_{\delta}+n_{\delta})A_{\rho,n}\rightarrow 0\quad\text{in}~L^1(0,T;L^1(\mathbb{T}^{d}) ),\\
&n_{\delta}-(\rho_{\delta}+n_{\delta})B_{\rho,n}\rightarrow 0\quad\text{in}~L^1(0,T;L^1(\mathbb{T}^{d}) ),
\end{split}
\right.
\end{equation}
where $(A_{\rho, n},B_{\rho,n})$ is defined by $(\ref{AB})$.

Furthermore, as $\delta\rightarrow0$, we have
\begin{equation}\label{rhonifdelta}
\begin{split}
&\quad\quad~(\rho_{\delta},n_{\delta})\rightarrow (\rho,n)\quad\text{in}~L^1(0,T;L^1(\mathbb{T}^{d}))\times L^1(0,T;L^1(\mathbb{T}^{d}))\\
&\Longleftrightarrow \rho_{\delta}+n_{\delta}\rightarrow \rho+n\quad\text{in}~L^1(0,T;L^1(\mathbb{T}^{d}) ).
\end{split}
\end{equation}
\end{lemma}

As in \cite{belgacem1,bresch1,bresch3}, introduce the symmetric periodic kernel
\begin{equation}\label{mathcalK}
\mathcal{K}_{h}(x):=
\begin{cases}
\frac{1}{(h+|x|)^{d}},
& \mbox{if\quad$0\leq |x|\leq \frac{1}{2},$ } \\
\text{smooth},
& \mbox{if\quad$\frac{1}{2}\leq|x|\leq \frac{2}{3},$}\\
\text{smooth and independent of}~h,
& \mbox{if\quad$\frac{2}{3}\leq|x|\leq \frac{3}{4}$,}\\
0,
& \mbox{if\quad$\frac{3}{4}\leq|x|\leq 1$.}
\end{cases}
\end{equation}
It is easy to show for a constant $C>1$ independent of $h$ and a suitably small constant $h_{0}\in(0,1)$ that
\begin{equation}\label{K}
\begin{split}
&|x||\nabla \mathcal{K}_{h}(x)|\leq C\mathcal{K}_{h}(x),\quad \frac{1}{C}|\log{h}|\leq\|\mathcal{K}_{h}\|_{L^1}\leq C|\log{h}|,\quad h\in(0,h_{0}).
\end{split}
\end{equation}
 Define the functional
\begin{equation}\label{L}
\begin{split}
&L_{h,p}(f):=\int_{\mathbb{T}^{2d}}\overline{\mathcal{K}}_{h}(x-y)|\Lambda[f]|^{p}dxdy,
\end{split}
\end{equation}
where we use the convenient notations
\begin{equation}\label{fx}
\left\{
 \begin{split}
&\Lambda[f]:=f^{x}-f^{y},\\
& f^{x}:=f(x,t),\\
&\overline{\mathcal{K}}_{h}:=\frac{\mathcal{K}_{h}}{\|\mathcal{K}_{h}\|_{L^1}}.
  \end{split}
  \right.
  \end{equation}
By virtue of Lemma \ref{lemma63} below, to derive the strong convergence of $\vartheta_{\delta}:=\rho_{\delta}+n_{\delta}$, we need
$$
\lim_{h\rightarrow 0}\limsup_{\delta\rightarrow 0}~\underset{t\in[0, T]}{{\rm{ess~sup}}}~L_{h,1}(\vartheta_{\delta})=0.
$$
It is shown in Lemma \ref{lemma44} below that the above estimate is the same thing as
\begin{equation}\nonumber
\begin{split}
\lim_{h\rightarrow 0}\limsup_{\delta\rightarrow 0}~\underset{t\in[0, T]}{{\rm{ess~sup}}}~\int_{\mathbb{T}^{2d}}\overline{\mathcal{K}}_{h}(x-y)\chi(\Lambda[\vartheta_{\delta}])(w^{x}_{\delta}+w^{y}_{\delta})dxdy=0,
\end{split}
\end{equation}
where $w_{\delta}$ is given by (\ref{w}) below and $\chi$ is the function
\begin{equation}\label{chi}
\begin{split}
\chi(s):=
\begin{cases}
|s|^2,
& \mbox{if\quad$|s|\leq 1,$ } \\
\text{smooth},
& \mbox{if\quad$1\leq|s|\leq 2,$}\\
|s|,
& \mbox{if\quad$|s|\geq2$,}
\end{cases}
\end{split}
\end{equation}
which satisfies
\begin{equation}\label{chi1}
\left\{
\begin{split}
&|\chi(s)-\frac{1}{2}\chi'(s)s|\leq \frac{1}{2}\chi'(s)s,\quad 0\leq \chi'(s)s\leq C\chi(s)\leq C|s|,\quad  s\in\mathbb{R},\\
&\chi(s)\geq \frac{1}{C}|s|,\quad  |s|\geq1.
\end{split}
\right.
\end{equation}
\begin{lemma}\label{lemma44}
Let $T>0$, $h\in(0,h_{0})$, $(\rho_{\delta},n_{\delta},(\rho_{\delta}+n_{\delta})u_{\delta})$ be the weak solution to the IVP $(\ref{twodelta})$-$(\ref{ddelta})$ for $\delta\in(0,1)$ given by Proposition \ref{prop31}, $\vartheta_{\delta}:=\rho_{\delta}+n_{\delta}$, and $w_{\delta}$ be the solution to the auxiliary problem
\begin{equation}\label{w}
\left\{
\begin{split}
&\partial_{t}w_{\delta}+u_{\delta}\cdot\nabla w_{\delta}+\lambda_{0} \Xi_{\delta}w_{\delta}=0,\quad x\in\mathbb{T}^{d},\quad t\in (0,T],\\
&\Xi_{\delta}:=\vartheta_{\delta} |\dive u_{\delta}|+|\dive u_{\delta}|+M|\nabla u_{\delta}|+\rho_{\delta}^{\gamma}+\rho_{\delta}^{\widetilde{\gamma}}+n_{\delta}^{\alpha}+n_{\delta}^{\widetilde{\alpha}}+1,\\
&w_{\delta}(x,0)=e^{-\lambda_{0} (\rho_{0,\delta}+n_{0,\delta})(x)},\quad \quad~ x\in\mathbb{T}^{d},
\end{split}
\right.
\end{equation}
where $\lambda_{0}\geq 1$ is a constant to be chosen and $M$ is the localized maximal operator defined by $(\ref{M})$. Then, under the assumptions of either Theorem \ref{theorem11} or Theorem \ref{theorem12}, we have
\begin{equation}\label{wbound}
\left\{
\begin{split}
&0\leq w_{\delta}\leq e^{-\lambda_{0} \vartheta_{\delta}}\leq 1,\\
 &\underset{t\in[0, T]}{{\rm{ess~sup}}}\int_{\mathbb{T}^d}\vartheta_{\delta}\log{(1+|\log w_{\delta}|)}dx\leq C(1+\lambda_{0}),
\end{split}
\right.
\end{equation}
and for $\sigma_{*}>0$ to be chosen later,
\begin{equation}
\begin{split}
\underset{t\in[0, T]}{{\rm{ess~sup}}}~\big{(} L_{h,1}(\vartheta _{\delta}) \big{)}^2\leq \frac{C(1+\lambda_{0})}{\log{(1+|\log{\sigma_{*}}|)}}+\frac{1}{\sigma_{*}}~\underset{t\in[0, T]}{{\rm{ess~sup}}}~\int_{\mathbb{T}^{2d}}\overline{\mathcal{K}}_{h}(x-y)\chi(\Lambda[\vartheta_{\delta}])(w^{x}_{\delta}+w^{y}_{\delta})dxdy,\label{Ew}
\end{split}
\end{equation}
where $L_{h,1}(f)$, $\Lambda[f]$, $f^{x}$, $\overline{\mathcal{K}}_{h}$ and $\chi$ are given by $(\ref{L})$, $(\ref{fx})_{1}$, $(\ref{fx})_{2}$, $(\ref{fx})_{3}$ and $(\ref{chi})$, respectively, and $C>0$ is a constant independent of $\delta$, $h$ and $\sigma_{*}$.
\end{lemma}
\begin{proof}
First, it follows from the maximal principle for the transport equation $(\ref{w})_{1}$ that
\begin{equation}
\begin{split}
0\leq w_{\delta}\leq 1,\label{w01}
\end{split}
\end{equation}
which leads to $|\log{w_{\delta}}|=-\log{w_{\delta}}$. Since there holds
\begin{equation}
\begin{split}
\partial_{t}\vartheta_{\delta}+u_{\delta}\cdot\nabla\vartheta_{\delta} +\vartheta_{\delta}\dive u_{\delta}=0\quad\text{in}~\mathcal{D}'(\mathbb{T}^{d}\times(0,T)),\label{varrhoapp}
\end{split}
\end{equation}
one derives by $(\ref{w})_{2}$ and the argument of renormalized solutions as in $(\ref{varthetavar})$-$(\ref{342})$ that
\begin{equation}\label{etheta}
\begin{split}
&\partial_{t}e^{-\lambda_{0} \vartheta_{\delta}}+u_{\delta}\cdot\nabla e^{-\lambda_{0} \vartheta_{\delta}}+\lambda_{0} \Xi_{\delta}e^{-\lambda_{0} \vartheta_{\delta}}=\lambda_{0} (\Xi_{\delta}+ \vartheta_{\delta} \dive u_{\delta})e^{-\lambda_{0} \vartheta_{\delta}}\geq 0 \quad\text{in}~\mathcal{D}'(\mathbb{T}^{d}\times(0,T)).
\end{split}
\end{equation}
From the comparison principle for $(\ref{w})_{1}$ and $(\ref{etheta})$, we have $(\ref{wbound})_{1}$.

Next, by $(\ref{wbound})_{1}$ and the argument of renormalized solutions (formally multiplying $(\ref{w})_{1}$ by $-\vartheta_{\delta}(1+|\log{w_{\delta}|})w_{\delta}^{-1}$), we obtain
\begin{equation}\label{441}
\begin{split}
&\partial_{t}\big{(}\vartheta_{\delta} \log{(1+|\log{w_{\delta}}|)}\big{)}+\dive\big{(}u_{\delta}\vartheta_{\delta} \log{(1+|\log{w_{\delta}}|)} \big{)}\\
&\quad=\frac{\lambda_{0} \vartheta_{\delta} \Xi_{\delta}}{1+|\log{w_{\delta}}|}\leq \frac{\lambda_{0} \vartheta_{\delta} \Xi_{\delta}}{1+\lambda_{0}\vartheta_{\delta}} \leq  \Xi_{\delta}  \quad\text{in}~\mathcal{D}'(\mathbb{T}^{d}\times(0,T)).
\end{split}
\end{equation}
Since we have
$$
\int_{\mathbb{T}^{d}}\vartheta_{\delta} \log{(1+|\log{w_{\delta}}|)}\big{|}_{t=0}dx\leq C\lambda_{0},
$$
and the term $\Xi_{\delta}$ on the right-hand side of (\ref{441}) is uniformly bounded in $L^1(0,T;L^1(\mathbb{T}^{d}) )$ due to $(\ref{gamma1})_{2}$, $(\ref{gamma2})_{2}$-$(\ref{gamma2})_{3}$, $(\ref{rdelta1})$ and $(\ref{rhonLp})$-$(\ref{rhonLpc})$, one shows $(\ref{wbound})_{2}$ after integrating (\ref{441}) over $\mathbb{T}^{d}\times[0,t]$.

Finally, one has by (\ref{chi}) and $(\ref{chi1})_{2}$ that
\begin{equation}\label{4422}
\begin{split}
&L_{h,1}(\vartheta_{\delta})=\Big{(} \int_{\{(x,y)\in \mathbb{T}^{2d}~|~|\Lambda[\vartheta_{\delta}]|\leq 1 \}} +\int_{\{(x,y)\in \mathbb{T}^{2d}~|~|\Lambda[\vartheta_{\delta}]|\geq 1\}}\Big{)}\overline{\mathcal{K}}_{h}(x-y)|\Lambda[\vartheta_{\delta}]| dxdy\\
&\quad\quad~~\quad\leq C\Big{(} \int_{\mathbb{T}^{2d}}\overline{\mathcal{K}}_{h}(x-y)\chi(\Lambda[\vartheta_{\delta}])dxdy\Big{)}^{\frac{1}{2}}.
\end{split}
\end{equation}
Owing to $(\ref{chi})_{1}$, the Young inequality of convolution type and $w_{\delta}^{x}+w_{\delta}^{y}\geq \sigma_{*}$ for $w_{\delta}^{x}\geq \sigma_{*}$ or $w_{\delta}^{y}\geq \sigma_{*}$, we obtain
\begin{equation}\nonumber
\begin{split}
&\int_{\mathbb{T}^{2d}}\overline{\mathcal{K}}_{h}(x-y)\chi(\Lambda[\vartheta_{\delta}])dxdy\\
&\quad=\Big{(} \int_{\{(x,y)\in \mathbb{T}^{2d}~|~w_{\delta}^{x}\leq \sigma_{*}~\text{and}~w_{\delta}^{y}\leq \sigma_{*}\}} +\int_{\{(x,y)\in \mathbb{T}^{2d}~|~w_{\delta}^{x}\geq \sigma~\text{or}~w_{\delta}^{y}\geq \sigma_{*}\}}\Big{)}\overline{\mathcal{K}}_{h}(x-y)\chi(\Lambda[\vartheta_{\delta}]) dxdy\\
&\quad\leq \frac{C}{\log{(1+|\log{\sigma_{*}}|)}}\int_{\mathbb{T}^{d}}\vartheta_{\delta} \log{(1+|\log{w_{\delta}}|)}dx+\frac{1}{\sigma_{*}}\int_{\mathbb{T}^{2d}}\overline{\mathcal{K}}_{h}(x-y)\chi(\Lambda[\vartheta_{\delta}])(w^{x}_{\delta}+w^{y}_{\delta})dxdy,
\end{split}
\end{equation}
which together with $(\ref{wbound})_{2}$ and (\ref{4422}) leads to (\ref{Ew}).
\end{proof}

With the help of the continuity equation (\ref{varrhoapp}), we have the quantitative regularity estimate with weight.

\begin{lemma}\label{lemma45}
Let $T>0$, $h\in(0,h_{0})$, $\lambda_{0}\geq1$, $(\rho_{\delta},n_{\delta},(\rho_{\delta}+n_{\delta})u_{\delta})$ be the weak solution to the IVP $(\ref{twodelta})$-$(\ref{ddelta})$ for $\delta\in(0,1)$ given by Proposition \ref{prop31}, $\vartheta_{\delta}:=\rho_{\delta}+n_{\delta}$, and $w_{\delta}$ be the solution to the IVP $(\ref{w})$. Then, under the assumptions of either Theorem \ref{theorem11} or Theorem \ref{theorem12}, it holds for $t\in[ 0,T]$ that
\begin{equation}\label{Ewdiv}
\begin{split}
& \int_{\mathbb{T}^{2d}}\overline{\mathcal{K}}_{h}(x-y)\chi(\Lambda[\vartheta_{\delta}])(w^{x}_{\delta}+w^{y}_{\delta})dxdy\\
&\quad\leq  L_{h,1}(\rho _{0,\delta}+n_{0,\delta})+\frac{C}{|\log{h}|^{\frac{1}{2}}}+(C-2\lambda_{0})\int_{0}^{t}\int_{\mathbb{T}^{2d}}\overline{\mathcal{K}}_{h}(x-y)\chi(\Lambda[\vartheta_{\delta}]) \Xi_{\delta}^{x}w_{\delta}^{x}dxdyd\tau\\
&\quad\quad-2\int_{0}^{t}\int_{\mathbb{T}^{2d}}\overline{\mathcal{K}}_{h}(x-y)\Lambda[\dive u_{\delta}]\big{(}\chi'(\Lambda[\vartheta_{\delta}])\vartheta_{\delta}^{y}+\chi(\Lambda[\vartheta_{\delta}]) \big{)} w_{\delta}^{x}dxdyd\tau,
\end{split}
\end{equation}
where $L_{h,1}(f)$, $\Lambda[f]$, $f^{x}$, $\overline{\mathcal{K}}_{h}$, $\chi$ and $\Xi_{\delta}$ are defined by $(\ref{L})$, $(\ref{fx})_{2}$, $(\ref{fx})_{1}$, $(\ref{fx})_{3}$, $(\ref{chi})$ and $(\ref{w})_{2}$, respectively, and $C>0$ is a constant independent of $\delta$, $h$ and $\lambda_{0}$.
\end{lemma}
\begin{proof}
It can be verified by $(\ref{varrhoapp})$ that
\begin{equation}
\begin{split}
&\partial_{t}\Lambda[\vartheta_{\delta}]+\dive_{x}(\Lambda[\vartheta_{\delta}]u_{\delta}^x)+\dive_{y}(\Lambda[\vartheta_{\delta}]u_{\delta}^y)\\
&\quad=-\vartheta _{\delta}^y\dive_{x}u_{\delta}^x+\vartheta _{\delta}^x\dive_{y}u_{\delta}^y\\
&\quad=\frac{1}{2}\Lambda[\vartheta_{\delta}](\dive_{x}u_{\delta}^x+\dive_{y}u_{\delta}^y)-\frac{1}{2}(\vartheta _{\delta}^x+\vartheta _{\delta}^y)\Lambda[\dive u_{\delta}]\quad\text{in}~\mathcal{D}'(\mathbb{T}^{d}\times(0,T)).\label{dxy}
\end{split}
\end{equation}
By virtue of (\ref{w}), (\ref{dxy}) and the argument of renormalized solutions as in $(\ref{varthetavar})$-$(\ref{342})$ (formally multiplying $(\ref{dxy})$ by $ \overline{\mathcal{K}}_{h}(x-y)\chi'(\Lambda[\vartheta_{\delta}])(w_{\delta}^x+w_{\delta}^y)$ and integrating the resulting equation over $\mathbb{T}^{2d}\times[0,t]$), one can show
\begin{equation}\nonumber
\begin{split}
&\int_{\mathbb{T}^{2d}}\overline{\mathcal{K}}_{h}(x-y)\chi(\Lambda[\vartheta_{\delta}]) (w^{x}_{\delta}+w^{y}_{\delta})dxdy\Big{|}^{t}_{0}\\
&\quad=\int_{0}^{t}\int_{\mathbb{T}^{2d}}\nabla\overline{\mathcal{K}}_{h}(x-y) \cdot \Lambda[ u_{\delta}] \chi(\Lambda[\vartheta_{\delta}] )  (w_{\delta}^x+w_{\delta}^y)dxdyd\tau\\
&\quad\quad+\int_{0}^{t}\int_{\mathbb{T}^{2d}}\overline{\mathcal{K}}_{h}(x-y) \chi(\Lambda[\vartheta_{\delta}]) (\partial_{t}w_{\delta}^x+u_{\delta}^x\cdot\nabla_{x}w _{\delta}^x+\partial_{t}w_{\delta}^y+u_{\delta}^y\cdot\nabla_{y}w_{\delta}^y)dxdyd\tau\\
&\quad\quad+\frac{1}{2}\int_{0}^{t}\int_{\mathbb{T}^{2d}}\overline{\mathcal{K}}_{h}(x-y) \big{(} 2\chi(\Lambda[\vartheta_{\delta}])-\chi'(\Lambda[\vartheta_{\delta}])\Lambda[\vartheta_{\delta}]\big{)} (\dive_{x}u_{\delta}^x+\dive_{y}u_{\delta}^y)(w_{\delta}^x+w_{\delta}^y)dxdyd\tau\\
&\quad\quad-\frac{1}{2}\int_{0}^{t}\int_{\mathbb{T}^{2d}}\overline{\mathcal{K}}_{h}(x-y)\chi'(\Lambda[\vartheta_{\delta}])(\vartheta_{\delta}^x+\vartheta _{\delta}^y)\Lambda[\dive u_{\delta}](w_{\delta}^x+w_{\delta}^y)dxdyd\tau.
\end{split}
\end{equation}
Making use of symmetry, we have
\begin{equation}
\begin{split}
&\int_{\mathbb{T}^{2d}}\overline{\mathcal{K}}_{h}(x-y)\chi(\Lambda[\vartheta_{\delta}])(w^{x}_{\delta}+w^{y}_{\delta})dxdy\Big{|}^{t}_{0}=\sum_{i=1}^{4}I_{i}^{3},\label{ddte}
\end{split}
\end{equation}
where $I_{i}^{3}$ $(i=1,2,3,4)$ are given by
\begin{equation}\nonumber
\begin{split}
&I_{1}^{3}:=2\int_{0}^{t}\int_{\mathbb{T}^{2d}} \nabla\overline{\mathcal{K}}_{h}(x-y) \cdot \Lambda[ u_{\delta}] \chi(\Lambda[\vartheta_{\delta}] ) w_{\delta}^xdxdyd\tau,\\
&I_{2}^{3}:=2\int_{0}^{t}\int_{\mathbb{T}^{2d}}\overline{\mathcal{K}}_{h}(x-y) \chi(\Lambda[\vartheta_{\delta}] )(\partial_{t}w_{\delta}^x+u_{\delta}^x\cdot\nabla_{x}w _{\delta}^x)dxdyd\tau,\\
&I_{3}^{3}:=\int_{0}^{t}\int_{\mathbb{T}^{2d}}\overline{\mathcal{K}}_{h}(x-y)\big{(} 2\chi(\Lambda[\vartheta_{\delta}])-\chi'(\Lambda[\vartheta_{\delta}])\Lambda[\vartheta_{\delta}]\big{)}(\dive_{x}u_{\delta}^x+\dive_{y}u_{\delta}^y)w_{\delta}^xdxdyd\tau,\\
&I_{4}^{3}:=-\int_{0}^{t}\int_{\mathbb{T}^{2d}}\overline{\mathcal{K}}_{h}(x-y)(\vartheta _{\delta}^x+\vartheta _{\delta}^y)\chi'(\Lambda[\vartheta_{\delta}]) \Lambda[\dive u_{\delta}]w_{\delta}^xdxdyd\tau.
\end{split}
\end{equation}
By (\ref{K}), $(\ref{chi1})_{1}$, $(\ref{wbound})_{1}$, (\ref{D}) and (\ref{DH1}), it holds
\begin{equation}\nonumber
\begin{split}
&I_{1}^3\leq C\int_{0}^{t}\int_{\mathbb{T}^{2d}}\overline{\mathcal{K}}_{h}(x-y)  \big{(} D_{|x-y| } u_{\delta}^x+D_{|x-y|}u_{\delta}^y \big{)} \chi(\Lambda[\vartheta_{\delta}])  w_{\delta}^xdxdyd\tau\\
&\quad\leq C\int_{0}^{t}\int_{\mathbb{T}^{2d}} \overline{\mathcal{K}}_{h}(x-y) \big{(} |\Lambda[ D_{|x-y|}u_{\delta}]|+2D_{|x-y|}u_{\delta}^{x} \big{)}  (\vartheta_{\delta}^{x}+\vartheta_{\delta}^{y}) w_{\delta}^xdxdyd\tau\\
&\quad \leq C\int_{0}^{t}\|\vartheta _{\delta}(\tau)\|_{L^2}\int_{\mathbb{T}^{d}} \overline{\mathcal{K}}_{h}(z) \| \big{(}D_{|z|}u_{\delta}^{\cdot-z}-D_{|z|}u_{\delta}^{\cdot} \big{)}(\tau)\|_{L^2}dzd\tau\\
&\quad\quad+C\int_{0}^{t}\int_{\mathbb{T}^{2d}}\overline{\mathcal{K}}_{h}(x-y) \chi(\Lambda[\vartheta_{\delta}])  M|\nabla_{x} u_{\delta}^x| w_{\delta}^xdxdyd\tau\\
&\quad \leq \frac{C\|\vartheta _{\delta}\|_{L^{2}(0,T;L^2)}\|u_{\delta}\|_{L^2(0,T;H^1)}}{|\log{h}|^{\frac{1}{2}}}+C\int_{0}^{t}\int_{\mathbb{T}^{2d}}\overline{\mathcal{K}}_{h}(x-y) \chi(\Lambda[\vartheta_{\delta}])  M|\nabla_{x} u_{\delta}^x| w_{\delta}^xdxdyd\tau,
\end{split}
\end{equation}
where $D_{|x|}f$ and $Mf$ are defined by (\ref{Maximal}) and (\ref{M}), respectively. For $I_{2}^{3}$, it follows from $(\ref{w})_{1}$ that
\begin{equation}\nonumber
\begin{split}
I_{2}^3=-2\lambda_{0}\int_{0}^{t}\int_{\mathbb{T}^{2d}}\overline{\mathcal{K}}_{h}(x-y) \chi(\Lambda[\vartheta_{\delta}]) \Xi_{\delta}^xw_{\delta}^xdxdyd\tau .
\end{split}
\end{equation}
In addition, by $(\ref{chi1})_{1}$ and the fact
\begin{equation}\nonumber
\begin{split}
&\big{(} 2\chi(\Lambda[\vartheta_{\delta}])-\chi'(\Lambda[\vartheta_{\delta}])\Lambda[\vartheta_{\delta}]\big{)}(\dive_{x}u_{\delta}^x+\dive_{y}u_{\delta}^y)+(\vartheta _{\delta}^x+\vartheta _{\delta}^y)\chi'(\Lambda[\vartheta_{\delta}]) \Lambda[\dive u_{\delta}]\\
&\quad=2 \big{(} 2\chi(\Lambda[\vartheta_{\delta}])-\chi'(\Lambda[\vartheta_{\delta}])\Lambda[\vartheta_{\delta}]\big{)} \dive_{x}u_{\delta}^x-2\big{(} \chi'(\Lambda[\vartheta_{\delta}]) \vartheta_{\delta}^{y}+\chi(\Lambda[\vartheta_{\delta}])\big{)} \Lambda[\dive u_{\delta}],
\end{split}
\end{equation}
 we infer
\begin{equation}\nonumber
\begin{split}
&I_{3}^3+I_{4}^3\leq C\int_{0}^{t}\int_{\mathbb{T}^{2d}}\overline{\mathcal{K}}_{h}(x-y)\chi(\Lambda[\vartheta_{\delta}] ) |\dive_{x}u_{\delta}^x|w_{\delta}^xdxdyd\tau\\
&\quad\quad\quad~- 2\int_{0}^{t}\int_{\mathbb{T}^{2d}}\overline{\mathcal{K}}_{h}(x-y)\Lambda[\dive u_{\delta}] \big{(} \chi'(\Lambda[\vartheta_{\delta}]) \vartheta_{\delta}^{y}+\chi(\Lambda[\vartheta_{\delta}])\big{)} w_{\delta}^{x}dxdyd\tau.
\end{split}
\end{equation}
One obtain (\ref{Ewdiv}) after substituting the above estimates of $I_{i}^{3}$ $(i=1,2,3,4)$ into (\ref{ddte}). 
\end{proof}

 Next, the truncated part of $\Lambda[\dive u_{\delta}]$ can be decomposed the pressure part and the effect viscous part.
\begin{lemma}\label{lemma46}
Let $T>0$, $p_{0}>\gamma+\widetilde{\gamma}+\alpha+\widetilde{\alpha}+1$, $h\in(0,h_{0})$, $(\rho_{\delta},n_{\delta},(\rho_{\delta}+n_{\delta})u_{\delta})$ be the weak solution to the IVP $(\ref{twodelta})$-$(\ref{ddelta})$ for $\delta\in(0,1)$ given by Proposition \ref{prop31}, $\vartheta_{\delta}:=\rho_{\delta}+n_{\delta}$, and $w_{\delta}$ be the solution to the IVP $(\ref{w})$. Then, under the assumptions of either Theorem \ref{theorem11} or Theorem \ref{theorem12}, it holds for $k\geq1$ and $t\in[ 0,T]$ that
\begin{equation}\label{dived}
\begin{split}
&-\int_{0}^{t}\int_{\mathbb{T}^{2d}}\overline{\mathcal{K}}_{h}(x-y)\Lambda[\dive u_{\delta}]\big{(}\chi'(\Lambda[\vartheta_{\delta}])\vartheta_{\delta}^{y}+\chi(\Lambda[\vartheta_{\delta}]) \big{)} w_{\delta}^{x}dxdyd\tau\\
&\quad\leq \frac{C}{k}-\frac{1}{2\mu+\lambda}\int_{0}^{t}\int_{\mathbb{T}^{2d}}\overline{\mathcal{K}}_{h}(x-y)\Lambda[P_{\delta}(\rho_{\delta}, n_{\delta}]\Psi_{\delta,k}dxdyd\tau\\
&\quad\quad\quad~+\frac{1}{2\mu+\lambda} \int_{0}^{t}\int_{\mathbb{T}^{2d}}\overline{\mathcal{K}}_{h}(x-y)\Lambda[F_{\delta}]\Psi_{\delta,k}dxdyd\tau,
\end{split}
\end{equation}
with
\begin{align}
\Psi_{\delta,k}:=\big{(}\chi'(\Lambda[\vartheta_{\delta}])\vartheta_{\delta}^{y}+\chi(\Lambda[\vartheta_{\delta}]) \big{)} \mathbf{1}_{\vartheta_{\delta}^x\leq k}\mathbf{1}_{\vartheta_{\delta}^y\leq k} w_{\delta}^{x},\label{Psi}
\end{align}
where $P_{\delta}$ $F_{\delta}$, $L_{h,1}(f)$, $\Lambda[f]$, $f^{x}$, $\overline{\mathcal{K}}_{h}$ and $\chi$ are given by $(\ref{Pvar})$, $(\ref{effect1})$, $(\ref{L})$, $(\ref{fx})_{1}$, $(\ref{fx})_{2}$, $(\ref{fx})_{3}$ and $(\ref{chi})$, respectively, and $C>0$ is a constant independent of $\delta$, $h$, $k$ and $\lambda_{0}$.
\end{lemma}
\begin{proof}
By $(\ref{rdelta1})$, $(\ref{rhonLp})$-$(\ref{rhonLpc})$, $(\ref{chi})_{1}$ and $(\ref{wbound})_{1}$, we have
\begin{equation}\label{461}
\begin{split}
&-\int_{0}^{t}\int_{\mathbb{T}^{2d}}\overline{\mathcal{K}}_{h}(x-y)\Lambda[\dive u_{\delta}]\big{(}\chi'(\Lambda[\vartheta_{\delta}])\vartheta_{\delta}^{y}+\chi(\Lambda[\vartheta_{\delta}]) \big{)} (1-\mathbf{1}_{\vartheta_{\delta}^x\leq k}\mathbf{1}_{\vartheta_{\delta}^y\leq k} ) w_{\delta}^xdxdyd\tau\\
&\quad\leq C\int_{\{(x,y,\tau)\in\mathbb{T}^{2d}~|~\vartheta_{\delta}^{x}\geq k~\text{or}~\vartheta_{\delta}^{y}\geq k\}}\overline{\mathcal{K}}_{h}(x-y) (|\dive_{x}u_{\delta}^x|+|\dive_{y}u_{\delta}^{y}|) (\vartheta_{\delta}^{x}w_{\delta}^x+1)dxdyd\tau\\
&\quad\leq \frac{C}{ k}\int_{\mathbb{T}^{2d}}\overline{\mathcal{K}}_{h}(x-y) (|\dive_{x}u_{\delta}^x|+|\dive_{y}u_{\delta}^{y}|) (\vartheta_{\delta}^x+\vartheta_{\delta}^{y}) dxdyd\tau\\
&\quad\leq \frac{C}{k} \|\dive u_{\delta}\|_{L^2(0,T;L^2)}\|\vartheta_{\delta}\|_{L^{2}(0,T;L^{2})}\leq  \frac{C}{k},
\end{split}
\end{equation}
where one has used the facts $|\chi'(\Lambda[\vartheta_{\delta}])|\vartheta^{y}_{\delta}\leq C(\vartheta^{x}_{\delta}+1)$ if $|\lambda[\vartheta_{\delta}]|\leq 2$, $\chi'(\Lambda[\vartheta_{\delta}])\vartheta_{\delta}^{y}+\chi(\Lambda[\vartheta_{\delta}])= \vartheta_{\delta}^{x} {\rm{sign}} \Lambda[\vartheta_{\delta}]$ for $|\lambda[\vartheta_{\delta}]|\geq 2$, $\vartheta _{\delta}^x+\vartheta _{\delta}^y\geq k$ for either $\vartheta _{\delta}^x\geq k$ or $\vartheta _{\delta}^y\geq k$ and 
\begin{equation}\nonumber
\begin{split}
\vartheta_{\delta}w_{\delta}\leq \vartheta_{\delta}e^{-\lambda_{0} \vartheta_{\delta}}\leq  \frac{1}{\lambda_{0}}\leq 1.
\end{split}
\end{equation}
By (\ref{effect1}) and (\ref{461}), (\ref{dived}) holds.
\end{proof}

We are ready to estimate the key pressure part.
\begin{lemma}\label{lemma47}
Let $T>0$, $p_{0}>\gamma+\widetilde{\gamma}+\alpha+\widetilde{\alpha}+1$, $h\in(0,h_{0})$, $(\rho_{\delta},n_{\delta},(\rho_{\delta}+n_{\delta})u_{\delta})$ be the weak solution to the IVP $(\ref{twodelta})$-$(\ref{ddelta})$ for $\delta\in(0,1)$ given by Proposition \ref{prop31}, $\vartheta_{\delta}:=\rho_{\delta}+n_{\delta}$, and $w_{\delta}$ be a solution to the IVP $(\ref{w})$. Then, for any $\zeta>0$, there is a constant $\delta_{1}(\zeta)\in(0,1)$ such that it holds for $\delta\in (0,\delta_{1}(\zeta))$, $k\geq1$ and $t\in[0,T]$ that
\begin{equation}
\begin{split}
&-\int_{0}^{t}\int_{\mathbb{T}^{2d}}\overline{\mathcal{K}}_{h}(x-y)\Lambda[P_{\delta}(\rho_{\delta}, n_{\delta}]\Psi_{\delta,k}dxdyd\tau\\
&\quad\leq Ck^{p_{0}}\Big{(} \delta^{\beta}+\zeta+\int_{0}^{t}\big{(} L_{h,1}(A_{\rho,n})+L_{h,1}(B_{\rho,n}) \big{)} d\tau\Big{)} \\
&\quad\quad+C\int_{0}^{t}\int_{\mathbb{T}^{2d}}\overline{\mathcal{K}}_{h}(x-y) \chi(\Lambda[\vartheta_{\delta}])\big{(} (\rho_{\delta}^{x})^{\gamma}+(\rho_{\delta}^{x})^{\widetilde{\gamma}}+(n_{\delta}^{x})^{\alpha}+(n_{\delta}^{x})^{\widetilde{\alpha}}+ 1\big{)}w_{\delta}^xdxdyd\tau,\label{PF}
\end{split}
\end{equation}
where $(A_{\rho,n},B_{\rho,n})$, $L_{h,1}(f)$, $\Lambda[f]$, $f^{x}$, $\overline{\mathcal{K}}_{h}$, $\chi$ and $\Psi_{\delta,k}$ are given by $(\ref{AB})$, $(\ref{L})$, $(\ref{fx})_{1}$, $(\ref{fx})_{2}$, $(\ref{fx})_{3}$, $(\ref{chi})$ and $(\ref{Psi})$, respectively, and $C>0$ is a constant independent of $\delta$, $h$, $k$ and $\lambda_{0}$.
\end{lemma}
\begin{proof}
To begin with, as $\delta\rightarrow0$, we deduce by (\ref{strongABdelta}) that
\begin{equation}\nonumber
\left\{
\begin{split}
&\rho_{\delta}-A_{\rho,n}\vartheta _{\delta}\rightarrow 0 \quad \text{a.e.}\quad~\text{in}~\mathbb{T}^{d}\times(0,T),\\
&n_{\delta}-B_{\rho,n}\vartheta _{\delta}\rightarrow 0 \quad \text{a.e.}\quad~\text{in}~ \mathbb{T}^{d}\times(0,T),
\end{split}
\right.
\end{equation}
which together with the Egorov theorem implies that for any $\zeta>0$, there exist a constant $\delta_{1}(\zeta)\in(0,1)$ and a domain $Q_{\zeta,T}\subset \mathbb{T}^{d}\times (0,T)$ such that
\begin{equation}\label{egg}
\left\{
\begin{split}
&|\mathbb{T}^{d}\times(0,T)/Q_{\zeta,T}|\leq \zeta,\\
& |(\rho_{\delta}-A_{\rho,n}\vartheta _{\delta})(x,t)|+|(n_{\delta}-B_{\rho,n}\vartheta _{\delta})(x,t)|\leq \zeta, \quad (x,t)\in Q_{\zeta,T},~\delta\in (0,\delta_{1}(\zeta)).
\end{split}
\right.
\end{equation}
Denoting
$$
\mathbb{Q}_{\zeta,t}=\big{\{}(x,y,\tau)\in\mathbb{T}^{2d}\times (0,t) ~ |~ (x,\tau)\in Q_{\zeta,t}~\text{and}~(y,\tau)\in Q_{\zeta,t} \big{\}},\quad  t\in [0,T],
$$
we obtain by $(\ref{Pvar})$, $(\ref{wbound})_{1}$ and (\ref{egg}) that
\begin{equation}\nonumber
\begin{split}
&-\int_{\mathbb{T}^{2d}\times (0,t)/ \mathbb{Q}_{\zeta,t}}\overline{\mathcal{K}}_{h}(x-y)\Lambda[P_{\delta}(\rho_{\delta},n_{\delta})]\Psi_{\delta,k}dxdyd\tau\\
&\quad\leq C(k^{\gamma}+k^{\alpha}+\delta k^{p_{0}}+1)\int_{\mathbb{T}^{2d}\times(0,t)/\mathbb{Q}_{\zeta,t}}\overline{\mathcal{K}}_{h}(x-y)dxdyd\tau\leq Ck^{p_{0}}\zeta.
\end{split}
\end{equation}
According to (\ref{ErPdelta}), the integration on $\mathbb{Q}_{\zeta,t}$ can be decomposed by
\begin{equation}\label{I34}
\begin{split}
&-\int_{ \mathbb{Q}_{\zeta,t}}\overline{\mathcal{K}}_{h}(x-y)\Lambda[P_{\delta}(\rho_{\delta},n_{\delta})]\Psi_{\delta,k}dxdyd\tau= \sum_{i=1}^{6}I_{i}^4,
\end{split}
\end{equation}
where $I_{i}^4$ $(i=1,...,6)$ are given by
\begin{equation}\nonumber
\begin{split}
&I_{1}^4:=-\int_{\mathbb{Q}_{\zeta,t}}\overline{\mathcal{K}}_{h}(x-y) \big{(}P_{\delta}(\rho_{\delta}^{x},n_{\delta}^{x})-P_{\delta}(A_{\rho^x,n^x}\vartheta _{\delta}^x,B_{\rho^x,n^x}\vartheta _{\delta}^x) \big{)}\Psi_{\delta,k}dxdyd\tau,\\
&I_{2}^4:=-\int_{\mathbb{Q}_{\zeta,t}}\overline{\mathcal{K}}_{h}(x-y)  \big{(} P_{\delta}(A_{\rho^{y},n^{y}}\vartheta _{\delta}^{y},B_{\rho^{y},n^{y}}\vartheta _{\delta}^{y})-P_{\delta}(\rho_{\delta}^{y},n_{\delta}^{y}) \big{)}\Psi_{\delta,k}dxdyd\tau,\\
&I_{3}^4:=-\int_{\mathbb{Q}_{\zeta,t}}\overline{\mathcal{K}}_{h}(x-y)\big{(} P_{\delta}(A_{\rho^{x},n^{x}}\vartheta _{\delta}^y,B_{\rho^x,n^x}\vartheta _{\delta}^y)-P_{\delta}(A_{\rho^y,n^y}\vartheta _{\delta}^y,B_{\rho^y,n^y}\vartheta _{\delta}^y)\big{)}\Psi_{\delta,k}dxdyd\tau,\\
&I_{4}^4:=-\int_{\mathbb{Q}_{\zeta,t}}\overline{\mathcal{K}}_{h}(x-y)\big{(} \mathbf{1}_{\vartheta_{\delta}^x\leq \delta}P(A_{\rho^x,n^x}\vartheta _{\delta}^x,B_{\rho^x,n^x}\vartheta _{\delta}^x)-\mathbf{1}_{\vartheta_{\delta}^y\leq \delta}P(A_{\rho^x,n^x}\vartheta _{\delta}^y,B_{\rho^x,n^x}\vartheta _{\delta}^y)\big{)}\Psi_{\delta,k}dxdyd\tau,\\
&I_{5}^4:=-\delta\int_{\mathbb{Q}_{\zeta,t}}\overline{\mathcal{K}}_{h}(x-y) \big{(} (A_{\rho^x,n^x}\vartheta _{\delta}^x)^{p_{0}}-(A_{\rho^x,n^x}\vartheta _{\delta}^y)^{p_{0}}+(B_{\rho^x,n^x}\vartheta _{\delta}^x)^{p_{0}}-(B_{\rho^x,n^x}\vartheta _{\delta}^y)^{p_{0}}\big{)}\Psi_{\delta,k}dxdyd\tau,\\
&I_{6}^4:=-\int_{\mathbb{Q}_{\zeta,t}}\overline{\mathcal{K}}_{h}(x-y) \big{(} P(A_{\rho^x,n^x}\vartheta _{\delta}^x,B_{\rho^x,n^x}\vartheta _{\delta}^x)-P(A_{\rho^x,n^x}\vartheta _{\delta}^y,B_{\rho^x,n^x}\vartheta _{\delta}^y)\big{)}\Psi_{\delta,k}dxdyd\tau .
\end{split}
\end{equation}
It follows from (\ref{P}), (\ref{AB1})-(\ref{AB2}), (\ref{P11}), $(\ref{wbound})_{1}$ and (\ref{egg}) that
\begin{equation}\nonumber
\begin{split}
&I_{1}^4\leq C\int_{\mathbb{Q}_{\zeta,t}}\overline{\mathcal{K}}_{h}(x-y)  \big{[} \big{(} (\rho_{\delta}^x+A_{\rho^x,n^x}\vartheta _{\delta}^x)^{\widetilde{\gamma}-1}+\delta(\rho_{\delta}^x+A_{\rho^x,n^x}\vartheta _{\delta}^x)^{p_{0}-1} +1\big{)} |\rho_{\delta}^x-A_{\rho^x,n^x}\vartheta _{\delta}^x|\\
&\quad\quad+\big{(} (n_{\delta}^x+B_{\rho^x,n^x}\vartheta ^x_{\delta})^{\widetilde{\alpha}-1}+\delta(n_{\delta}^x+B_{\rho^x,n^x}\vartheta ^x_{\delta})^{p_{0}-1}+1\big{)} |n_{\delta}^x-B_{\rho^x,n^x}\vartheta _{\delta}^x |\big{]}( \vartheta_{\delta}^x+\vartheta_{\delta}^y) \mathbf{1}_{\vartheta_{\delta}^x\leq k}  \mathbf{1}_{\vartheta_{\delta}^y\leq k} dxdyd\tau\\
&\quad\leq Ck^{p_{0}}\zeta.
\end{split}
\end{equation}
Similarly, one has
\begin{equation}\nonumber
\begin{split}
&I_{2}^4\leq Ck^{p_{0}}\zeta.
\end{split}
\end{equation}
For the term $I_{3}^4$, we have by (\ref{P}), $(\ref{AB2})$, $(\ref{P11})$ and $(\ref{wbound})_{1}$ that
\begin{equation}\nonumber
\begin{split}
&I_{3}^4\leq Ck^{p_{0}}\big{(} L_{h,1}( A_{\rho,n})+L_{h,1}( B_{\rho,n}) \big{)}.
\end{split}
\end{equation}
Additionally, it is easy to prove
\begin{equation}\nonumber
\begin{split}
&I_{4}^4+I_{5}^{4}\leq C k^{p_{0}}\delta^{\beta} .
\end{split}
\end{equation}

The key term $I_{6}^{4}$ can be controlled by the damping term in $(\ref{w})_{1}$ with respect to $x$, so we need to overcome the difficulties caused by the possible dependence of $I_{6}^{5}$ on $y$. To this end, in the spirit of \cite{bresch1}, we analyze $I_{6}^{4}$ in the following three cases: 

\vspace{1ex}
\begin{itemize}
\item{Case 1: $ \big{(} P(A_{\rho^x,n^x}\vartheta _{\delta}^x,B_{\rho^x,n^x}\vartheta _{\delta}^x)-P(A_{\rho^x,n^x}\vartheta _{\delta}^y,B_{\rho^x,n^x}\vartheta _{\delta}^y)  \big{)}\Lambda [\vartheta_{\delta}]\geq 0.$}
\end{itemize}
Note that if $P(\rho,n)$ is monotone increasing in both $\rho$ and $n$, then all the cases reduce to Case 1. In this case, since $P(A_{\rho^x,n^x}\vartheta _{\delta}^x,B_{\rho^x,n^x}\vartheta _{\delta}^x)-P(A_{\rho^x,n^x}\vartheta _{\delta}^y,B_{\rho^x,n^x}\vartheta _{\delta}^y)$, $\Lambda [\vartheta_{\delta}]$ and $\chi'(\Lambda [\vartheta_{\delta}])$ have a same sign, it holds by $(\ref{chi1})_{1}$ that
\begin{equation}\nonumber
\begin{split}
&-\big{(} P(A_{\rho^x,n^x}\vartheta _{\delta}^x,B_{\rho^x,n^x}\vartheta _{\delta}^x)-P(A_{\rho^x,n^x}\vartheta _{\delta}^y,B_{\rho^x,n^x}\vartheta _{\delta}^y)\big{)}\Psi_{\delta,k}\\
& \quad\leq \big{ |} P(A_{\rho^x,n^x}\vartheta _{\delta}^x,B_{\rho^x,n^x}\vartheta _{\delta}^x)-P(A_{\rho^x,n^x}\vartheta _{\delta}^y,B_{\rho^x,n^x}\vartheta _{\delta}^y) \big{|}  \\\
&\quad\quad\times\big{(}-\frac{1}{2}\big{|} \chi'(\Lambda[\vartheta_{\delta}]) \big{|}(\vartheta_{\delta}^{x}+\vartheta_{\delta}^{y})+\big{|}\chi(\Lambda[\vartheta_{\delta}])- \frac{1}{2}\chi'(\Lambda[\vartheta_{\delta}])\Lambda[\vartheta_{\delta}] \big{|} \big{)}\leq 0,\quad (x,y,\tau)\in  \mathbb{Q}_{\zeta,t}.
\end{split}
\end{equation}

\begin{itemize}
\item{Case 2:  $ \big{(} P(A_{\rho^x,n^x}\vartheta _{\delta}^x,B_{\rho^x,n^x}\vartheta _{\delta}^x)-P(A_{\rho^x,n^x}\vartheta _{\delta}^y,B_{\rho^x,n^x}\vartheta _{\delta}^y)  \big{)}\Lambda [\vartheta_{\delta}]<0$ and $\vartheta_{\delta}^{y}\leq 2\vartheta _{\delta}^x$.} 
\end{itemize}
We make use of (\ref{P}), $(\ref{AB2})$, (\ref{P11}) and $(\ref{egg})_{2}$ to obtain
 \begin{equation}\nonumber
 \begin{split}
 &-\big{(} P(A_{\rho^x,n^x}\vartheta_{\delta}^x,B_{\rho^x,n^x}\vartheta_{\delta}^x)-P(A_{\rho^x,n^x}\vartheta_{\delta}^y,B_{\rho^x,n^x}\vartheta_{\delta}^y) \big{)}\Psi_{\delta,k}\\
 &\quad\leq C \big{(} (A_{\rho^x,n^x}\vartheta_{\delta}^x+A_{\rho^x,n^x}\vartheta_{\delta}^y)^{\widetilde{\gamma}-1}+1\big{)}A_{\rho^{x},n^{x}} |\Lambda[\vartheta_{\delta}]| \big{(} \chi'(\Lambda[\vartheta_{\delta}])\vartheta_{\delta}^{y}+\chi(\Lambda[\vartheta_{\delta}]) \big{)}w_{\delta}^x \\
 &\quad\quad+C\big{(} (B_{\rho^x,n^x}\vartheta _{\delta}^x+B_{\rho^x,n^x}\vartheta_{\delta}^y)^{\widetilde{\alpha}-1}+1\big{)} B_{\rho^{x},n^{x}} |\Lambda[ \vartheta_{\delta}]| \big{(} \chi'(\Lambda[\vartheta_{\delta}])\vartheta_{\delta}^{y}+\chi(\Lambda[\vartheta_{\delta}]) \big{)} w_{\delta}^{x}  \\
&\quad\leq C \big{(} (A_{\rho^x,n^x}\vartheta_{\delta}^x)^{\widetilde{\gamma}}+A_{\rho^x,n^x}\vartheta_{\delta}^x+(B_{\rho^x,n^x}\vartheta_{\delta}^x)^{\widetilde{\alpha}}+B_{\rho^x,n^x}\vartheta_{\delta}^x\big{)} \big{(} \chi'(\Lambda[\vartheta_{\delta}])\Lambda[\vartheta_{\delta}]+\chi(\Lambda[\vartheta_{\delta}]) \big{)} w_{\delta}^x\\
 &\quad\leq C\big{(} (\rho_{\delta}^x+\zeta)^{\widetilde{\gamma}}+\rho_{\delta}^x+\zeta+(n_{\delta}^x+\zeta)^{\widetilde{\alpha}} +n_{\delta}^x+\zeta\big{)}\chi(\Lambda[ \vartheta _{\delta}])w_{\delta}^x ,\quad (x,y,\tau)\in  \mathbb{Q}_{\zeta,t}.
 \end{split}
 \end{equation}

\begin{itemize}
\item{Case 3:  $ \big{(} P(A_{\rho^x,n^x}\vartheta _{\delta}^x,B_{\rho^x,n^x}\vartheta _{\delta}^x)-P(A_{\rho^x,n^x}\vartheta _{\delta}^y,B_{\rho^x,n^x}\vartheta _{\delta}^y)  \big{)}\Lambda [\vartheta_{\delta}]<0$ and $\vartheta_{\delta}^{y}> 2\vartheta _{\delta}^x$.} 
\end{itemize}
In this case, one has
\begin{equation}\nonumber
\left\{
\begin{split}
&|\Lambda[ \vartheta _{\delta}]|=\vartheta_{\delta}^y-\vartheta_{\delta}^x>\vartheta _{\delta}^x\geq0 , \quad \chi'(\Lambda(\vartheta_{\delta}))<0,\quad \vartheta_{\delta}^{y}=|\Lambda[ \vartheta _{\delta}]|+\vartheta_{\delta}^{x}< 2|\Lambda[ \vartheta _{\delta}]|, \\
&P(A_{\rho^x,n^x}\vartheta _{\delta}^x,B_{\rho^x,n^x}\vartheta _{\delta}^x)>P(A_{\rho^x,n^x}\vartheta _{\delta}^y,B_{\rho^x,n^x}\vartheta _{\delta}^y),
\end{split}
\right.
\end{equation}
which gives rise to
\begin{equation}\nonumber
\begin{split}
 &-\big{(} P(A_{\rho^x,n^x}\vartheta _{\delta}^x,B_{\rho^x,n^x}\vartheta _{\delta}^x)-P(A_{\rho^x,n^x}\vartheta _{\delta}^y,B_{\rho^x,n^x}\vartheta _{\delta}^y) \big{)}\Psi_{\delta,k}\\
&\quad=\big{(} P(A_{\rho^x,n^x}\vartheta _{\delta}^x,B_{\rho^x,n^x}\vartheta _{\delta}^x)-P(A_{\rho^x,n^x}\vartheta _{\delta}^y,B_{\rho^x,n^x}\vartheta _{\delta}^y) \big{)}\big{(} \big{|} \chi'(\Lambda[\vartheta_{\delta}]) \big{|}\vartheta_{\delta}^{y}-\chi(\Lambda[\vartheta_{\delta}]) \big{)}w_{\delta}^x\\
&\quad\leq \big{(} P(A_{\rho^x,n^x}\vartheta _{\delta}^x,B_{\rho^x,n^x}\vartheta _{\delta}^x)+C_{1}\big{)}\big{(} 2 \big{|}\chi'(\Lambda[\vartheta_{\delta}]) \big{|} \big{|}\Lambda[\vartheta_{\delta}] \big{|}+\chi(\Lambda[\vartheta_{\delta}]) \big{)} w_{\delta}^{x}\\
&\quad\leq C\big{(} (\rho_{\delta}^x+\zeta)^{\gamma}+(n_{\delta}^x+\zeta)^{\alpha}+1 \big{)} \chi(\Lambda[ \vartheta _{\delta}]) w_{\delta}^{x},\quad  (x,y,\tau)\in  \mathbb{Q}_{\zeta,t}.\label{Pcase2}
\end{split}
\end{equation}
Combining the above three cases together, we get
\begin{equation}\nonumber
\begin{split}
I_{6}^4\leq C\int_{0}^{t}\int_{\mathbb{T}^{2d}}\overline{\mathcal{K}}_{h}(x-y)\big{(} (\rho_{\delta}^x)^{\gamma}+(\rho_{\delta}^x)^{\widetilde{\gamma}}+(n_{\delta}^x)^{\alpha}+(n_{\delta}^x)^{\widetilde{\alpha}}+1\big{)} \chi(\Lambda[ \vartheta _{\delta}] )w_{\delta}^{x}dxdyd\tau.
\end{split}
\end{equation}
Substituting the above estimates of $I_{i}^{4}$ ($i=1,...,6$) into $(\ref{I34})$, we derive (\ref{PF}). The proof of Lemma \ref{lemma47} is complete.
\end{proof}
\begin{remark}
Let~$\Pi(\rho)$ be the general pressure laws for compressible Navier-Stokes equations given in {\rm{\cite{bresch1}}}. For $P(\rho,n)=\Pi(\rho+n)$, one can repeat same arguments as in {\rm{\cite{bresch1}}} to estimate the pressure part $(\ref{PF})$. However, even for $P(\rho,n)=\Pi(\rho)+\Pi(n)$, we have to reduce the two variables $\rho_{\delta}$ and $n_{\delta}$ into the one variable $\vartheta_{\delta}$ by applying the decomposition {\rm{(\ref{ErPdelta})}}.
\end{remark}

Finally, different from the previous work by Bresch-Jabin \cite{bresch1}, we make use of the structures of the equations $(\ref{twodelta})$ and the commutator estimates (\ref{risz1})-(\ref{risz2}) of the Riesz operator $\mathcal{R}_{i}=(-\Delta)^{-\frac{1}{2}}\partial_{i}$ to control the effect viscous flux part.
\begin{lemma}\label{lemma48}
Let $T>0$, $p_{0}>\gamma+\widetilde{\gamma}+\alpha+\widetilde{\alpha}+1$, $h\in(0,h_{0})$, $(\rho_{\delta},n_{\delta},(\rho_{\delta}+n_{\delta})u_{\delta})$ be the weak solution to the IVP $(\ref{twodelta})$-$(\ref{ddelta})$ for $\delta\in(0,1)$ given by Proposition \ref{prop31}, $\vartheta_{\delta}:=\rho_{\delta}+n_{\delta}$, and $w_{\delta}$ be the solution to the IVP $(\ref{w})$. Then, under the assumptions of either Theorem \ref{theorem11} or Theorem \ref{theorem12}, it holds for $k\geq1$ and $t\in[0,T]$ that
\begin{equation}
\begin{split}
&\int_{0}^{t}\int_{\mathbb{T}^{2d}}\overline{\mathcal{K}}_{h}(x-y)\Lambda[F_{\delta}]\Psi_{\delta,k}dxdyd\tau\leq \frac{C\lambda_{0}k^{p_{0}}}{|\log{h}|^{\frac{1}{2}\min\{\frac{1}{2},\frac{2\gamma_{0}-d}{\gamma_{0}+d}\}}},\quad\quad \gamma_{0}:=\min\{\gamma,\alpha\},\label{EVF}
\end{split}
\end{equation}
where $F_{\delta}$, $\Lambda[f]$, $f^{x}$, $\overline{\mathcal{K}}_{h}$ and $\Psi_{\delta,k}$ are given by $(\ref{effect1})$, $(\ref{fx})_{1}$, $(\ref{fx})_{2}$, $(\ref{fx})_{3}$ and $(\ref{Psi})$, respectively, and $C>0$ is a constant independent of $\delta$, $h$, $k$ and $\lambda_{0}$.
\end{lemma}
\begin{proof}
By the argument of renormalized solutions for $(\ref{w})_{1}$ and $(\ref{varrhoapp})$, one has
\begin{equation}
\begin{split}
&\partial_{t}\Psi_{\delta,k}+\dive_{x}(\Psi_{\delta,k}u^x_{\delta})+\dive_{y}(\Psi_{\delta,k} u_{\delta}^y)=R_{\delta,k}\quad\text{in}~\mathcal{D}'(\mathbb{T}^{2d}\times(0,T)),\label{Psieq}
\end{split}
\end{equation}
with
\begin{equation}\nonumber
\begin{split}
&R_{\delta,k}:=(\Psi_{\delta,k}-\partial_{\vartheta_{\delta}^{x}}\Psi_{\delta,k})\dive_{x}u_{\delta}^{x}+(\Psi_{\delta,k}-\partial_{\vartheta_{\delta}^{y}}\Psi_{\delta,k})\dive_{y}u_{\delta}^{y}-\lambda_{0} \partial_{w_{\delta}^{x}}\Psi_{\delta,k} \Xi_{\delta}^x w_{\delta}^x.
\end{split}
\end{equation}
By $(\ref{w})_{2}$ and $(\ref{wbound})_{1}$, it is easy to verify
\begin{equation}\label{PsiR}
\left\{
\begin{split}
&\|\Psi_{\delta,k}\|_{L^{\infty}(0,T;L^{\infty}(\mathbb{T}^{2d}))}\leq Ck,\\
& \|R_{\delta,k}\|_{L^{\infty}(0,T;L^{\infty}(\mathbb{T}^{2d}))}\leq C\lambda_{0}k^{p_{0}}\big{(} |\dive_{x}u_{\delta}^x|+|\dive_{y}u_{\delta}^y|+M|\nabla_{x} u_{\delta}^x|+1\big{)},
\end{split}
\right.
\end{equation}
where the constant $C>0$ is independent of $\delta, k,$ and $\lambda_{0}$ and the localized maximal operator $M: L^{p}(\mathbb{T}^{d})\rightarrow L^{p}(\mathbb{T}^{d})$ for any $p\in[1,\infty)$ is defined by (\ref{Maximal}). From $(\ref{effect11})$, we have
\begin{equation}\label{psi11}
\begin{split}
&\Lambda[F_{\delta}]=\partial_{t}\Lambda[(-\Delta)^{-1}\dive(\vartheta_{\delta} u_{\delta})]+\Lambda[(-\Delta)^{-1}\dive\dive (\vartheta_{\delta}u_{\delta}\otimes u_{\delta})].
\end{split}
\end{equation}
One deduces by (\ref{Psieq}), $(\ref{psi11})$ and integration by parts that
\begin{equation}\label{481}
\begin{split}
&\int_{0}^{t}\int_{\mathbb{T}^{2d}}\overline{\mathcal{K}}_{h}(x-y)\Lambda[F_{\delta}]\Psi_{\delta,k}dxdyd\tau=\sum_{i=1}^{4}I_{i}^{5},
\end{split}
\end{equation}
where $I_{i}^{5}$ $(i=1,2,3,4)$ are given by
\begin{equation}\nonumber
\begin{split}
&I_{1}^{5}:=\int_{\mathbb{T}^{2d}} \overline{\mathcal{K}}_{h}(x-y) \Lambda[(-\Delta)^{-1}\dive (\vartheta_{\delta} u_{\delta})] \Psi_{\delta,k} dxdy \Big{|}^{t}_{0} ,\\
&I_{2}^5:=-\int_{0}^{t}\int_{\mathbb{T}^{2d}} \overline{\mathcal{K}}_{h}(x-y) \Lambda[(-\Delta)^{-1}\dive (\vartheta_{\delta} u_{\delta})]   R_{\delta,k}dxdyd\tau ,\\
&I_{3}^5:=\int_{0}^{t}\int_{\mathbb{T}^{2d}}\nabla  \overline{\mathcal{K}}_{h}(x-y) \Lambda[ u_{\delta}] \cdot\Lambda[(-\Delta)^{-1}\dive (\vartheta_{\delta} u_{\delta})] \Psi_{\delta,k}dxdyd\tau ,\\
&I_{4}^5:=\sum_{i,j=1}^{d}\int_{0}^{t}\int_{\mathbb{T}^{2d}}  \overline{\mathcal{K}}_{h}(x-y)  \Lambda[\mathcal{R}_{i}\mathcal{R}_{j}((\vartheta_{\delta} u^{i}_{\delta}) u^{j}_{\delta})-u^{j}_{\delta} \mathcal{R}_{i}\mathcal{R}_{j}(\vartheta_{\delta} u^{i}_{\delta})] \Psi_{\delta,k}dxdyd\tau.
\end{split}
\end{equation}
First, we estimate the term $I_{1}^5$. It follows from (\ref{r0}), (\ref{LNW}) and the boundedness of the operator $\nabla (-\Delta)^{-1} \dive$ that
\begin{equation}\label{Lh1}
\begin{split}
&L_{h,p}\big{(} (-\Delta)^{-1}\dive (\vartheta_{\delta} u_{\delta})\big{)}\\
&\quad\leq \frac{C\|(-\Delta)^{-1}\dive (\vartheta_{\delta} u_{\delta})\|_{L^{p}}}{|\log{h}|^{\frac{1}{2}}}+\frac{C\|\nabla(-\Delta)^{-1}\dive (\vartheta_{\delta} u_{\delta})\|_{L^{\frac{2\gamma_{0}}{\gamma_{0}+1} }}}{|\log{h}|^{pd\min\{0, \frac{1}{p}-\frac{\gamma_{0}+1}{2\gamma_{0}}\}+p}}\\
&\quad\leq C\big{(} \frac{1}{|\log{h}|^{\frac{1}{2}}}+\frac{1}{|\log{h}|^{pd\min\{0, \frac{1}{p}-\frac{\gamma_{0}+1}{2\gamma_{0}}\}+p}} \big{)},\quad\quad p\in[1,\frac{2\gamma_{0} d(\gamma_{0}+1)}{d(\gamma_{0}+1)-2\gamma_{0}}),
\end{split}
\end{equation}
which together with $(\ref{dappu})$, (\ref{r0}) and $(\ref{PsiR})_{2}$ implies
\begin{equation}\nonumber
\begin{split}
&I_{1}^5\leq Ck ~\underset{t\in[0, T]}{{\rm{ess~sup}}}~  L_{h,1}\big{(} (-\Delta)^{-1}\dive (\vartheta_{\delta} u_{\delta})\big{)}+ Ck L_{h,1} \big{(} (-\Delta)^{-1}\dive( (\rho_{0,\delta}+n_{0,\delta})u_{0,\delta} )\big{)}\leq \frac{Ck}{|\log{h}|^{\frac{1}{2}}}.
\end{split}
\end{equation}
To control the term $I_{2}^5$, we obtain by $(\ref{rdelta1})$, (\ref{r0}), $(\ref{PsiR})_{2}$ and (\ref{Lh1}) that
\begin{equation}\nonumber
\begin{split}
&I_{2}^5\leq Ck^{p_{0}}\lambda_{0}\int_{0}^{t}\int_{\mathbb{T}^{2d}} \overline{\mathcal{K}}_{h}(x-y) |\Lambda[(-\Delta)^{-1}\dive (\vartheta_{\delta} u_{\delta})]| \big{(} |\dive_{x}u_{\delta}^x|+|\dive_{y}u_{\delta}^y| +M|\nabla_{x} u_{\delta}^x|+1\big{)}dxdyd\tau\\
&\quad\leq Ck^{p_{0}}\lambda_{0}\|\nabla u_{\delta}\|_{L^2(0,T;L^2)} \underset{t\in[0, T]}{{\rm{ess~sup}}}~ L_{h,2}\big{(} (-\Delta)^{-1}\dive (\vartheta_{\delta} u_{\delta})\big{)}^{\frac{1}{2}}\\
&\quad\quad+Ck^{p_{0}}\lambda_{0}~\underset{t\in[0, T]}{{\rm{ess~sup}}}~ L_{h,1}\big{(} (-\Delta)^{-1}\dive (\vartheta_{\delta} u_{\delta})\big{)}\leq \frac{Ck^{p_{0}}\lambda_{0}}{|\log{h}|^{\min\{\frac{1}{4},\frac{2\gamma_{0}-d}{2\gamma_{0}}\}}}.
\end{split}
\end{equation}
Similarly to the estimate of $I^3_{1}$ in (\ref{ddte}), since $(-\Delta)^{-1}\dive (\vartheta_{\delta}  u_{\delta} )$ is uniformly bounded in $L^2(0,T;L^2(\mathbb{T}^{d}))$ due to $\gamma_{0}>\frac{d}{2}$, one can prove
\begin{equation}\label{i26}
\begin{split}
&I_{3}^5\leq Ck\int_{0}^{t}\int_{\mathbb{T}^{2d}}\overline{\mathcal{K}}_{h}(x-y) |\Lambda[(-\Delta)^{-1}\dive (\vartheta_{\delta}  u_{\delta} )]|  M|\nabla_{x} u_{\delta}^x| dxdyd\tau+\frac{Ck}{|\log{h}|^{\frac{1}{2}}}.
\end{split}
\end{equation}
By  (\ref{rdelta1}), (\ref{r0}) and (\ref{Lh1}), we have
\begin{equation}\label{pdxx1}
\begin{split}
&\int_{0}^{t}\int_{\mathbb{T}^{2d}} \overline{\mathcal{K}}_{h}(x-y) |\Lambda[(-\Delta)^{-1}\dive (\vartheta_{\delta}  u_{\delta} )]|  M|\nabla_{x} u_{\delta}^x| dxdyd\tau\\
&\quad\leq C\| M|\nabla u_{\delta}|\|_{L^2(0,T;L^2)} \underset{t\in[0, T]}{{\rm{ess~sup}}}~ L_{h,2}\big{(} (-\Delta)^{-1}\dive (\vartheta_{\delta} u_{\delta}) \big{)}^{\frac{1}{2}}\\
&\quad\leq\frac{C}{|\log{h}|^{\min\{\frac{1}{4},\frac{2\gamma_{0}-d}{2\gamma_{0}}\}}}.
\end{split}
\end{equation}
We obtain from $(\ref{i26})$-$(\ref{pdxx1})$ that
\begin{equation}\nonumber
\begin{split}
&I_{3}^5\leq\frac{Ck}{|\log{h}|^{\min\{\frac{1}{4},\frac{2\gamma_{0}-d}{2\gamma_{0}}\}}}.
\end{split}
\end{equation}

Finally, we are ready to estimate the last term $I_{4}^5$. Inspired by \cite{lions2}, we make use of the commutator estimate (\ref{risz2}) to have
\begin{equation}\label{6161}
\begin{split}
&\|\nabla\big{(}\mathcal{R}_{i}\mathcal{R}_{j}(\vartheta_{\delta} u_{\delta}^i u_{\delta}^j)-u_{\delta}^j \mathcal{R}_{i}\mathcal{R}_{j}(\vartheta_{\delta} u_{\delta}^i)\big{)}(t)\|_{L^{p}}\leq C\|\nabla u_{\delta}^j(t)\|_{L^2}\|\vartheta_{\delta}  u_{\delta}^i(t)\|_{L^q},\quad \text{a.e.}~t\in(0,T),
\end{split}
\end{equation}
where the constants $p\in(1,\infty)$ and $q\in(2,\infty)$ satisfy $\frac{1}{2}+\frac{1}{q}=\frac{1}{p}<1$. Note that one may not apply (\ref{6161}) directly for $\gamma_{0}\leq d$. For example, $\vartheta_{\delta} u_{\delta}$ is uniformly bounded in $L^2(0,T;L^{\frac{6\gamma_{0}}{6+\gamma_{0}}}(\mathbb{T}^{d}))$ for $d=3$, but we have $\frac{1}{2}+\frac{6+\gamma_{0}}{6\gamma_{0}}\geq1$ provided $\gamma_{0}\leq 3$. To overcome this difficulty, we truncate $\vartheta_{\delta}$ by using $1=\mathbf{1}_{\vartheta_{\delta}\leq L}+\mathbf{1}_{\vartheta_{\delta}\geq L}$ for $L\geq1$ leading to $I_{4}^5=I_{1}^6+I_{2}^6$ with
\begin{equation}\nonumber
\begin{split}
&I_{1}^6:=\sum_{i,j=1}^{d}\int_{0}^{t}\int_{\mathbb{T}^{2d}}\overline{\mathcal{K}}_{h}(x-y)\Lambda\big{[}\mathcal{R}_{i}\mathcal{R}_{j} ( \vartheta_{\delta}\mathbf{1}_{\vartheta_{\delta}\leq L}  u^{i}_{\delta}u^{j}_{\delta})-u^{j}_{\delta}\mathcal{R}_{i}\mathcal{R}_{j}( \vartheta_{\delta} \mathbf{1}_{\vartheta_{\delta}\leq L} u^i_{\delta})\big{]} \Psi_{\delta,k} dxdyd\tau,\\
&I_{2}^6:=\sum_{i,j=1}^{d}\int_{0}^{t}\int_{\mathbb{T}^{2d}}\overline{\mathcal{K}}_{h}(x-y)\Lambda\big{[}\mathcal{R}_{i}\mathcal{R}_{j} ( \vartheta_{\delta}\mathbf{1}_{\vartheta_{\delta}\geq L}  u^{i}_{\delta}u^{j}_{\delta})-u^{j}_{\delta}\mathcal{R}_{i}\mathcal{R}_{j}( \vartheta_{\delta}\mathbf{1}_{\vartheta_{\delta}\geq L} u^i_{\delta})\big{]} \Psi_{\delta,k} dxdyd\tau.
\end{split}
\end{equation}
The terms $I_{1}^6$ and $I_{2}^6$ can be estimated in the case of $d\geq 3$ and the case of $d=2$ as follows:

\vspace{1ex}
\begin{itemize}
\item{Case 1: $d\geq 3.$}
\end{itemize}
We show after using $(\ref{PsiR})_{1}$, (\ref{LNW}), (\ref{risz2}) and $H^1(\mathbb{T}^{d}) \hookrightarrow L^{\frac{2d}{d-2}}(\mathbb{T}^{d}) $ that
\begin{equation}\nonumber
\begin{split}
&I_{1}^6\leq Ck\sum_{i,j=1}^{d}\int_{0}^{t} L_{h,1}\big{(} \mathcal{R}_{i}\mathcal{R}_{j} (\vartheta_{\delta}\mathbf{1}_{\vartheta_{\delta}\leq L} u^{i}_{\delta}u^{j}_{\delta})-u^{j}_{\delta}\mathcal{R}_{i}\mathcal{R}_{j}( \mathbf{1}_{\vartheta_{\delta}\leq L}\vartheta_{\delta} u^i_{\delta})\big{)}d\tau\\
&\quad\leq Ck\Big{(}\frac{1}{|\log{h}|^{\frac{1}{2}}}+\frac{1}{|\log{h}|^2}\sum_{i,j=1}^d\|\nabla [\mathcal{R}_{i}\mathcal{R}_{j} \big{(}\vartheta_{\delta}\mathbf{1}_{\vartheta_{\delta}\leq L} u^{i}_{\delta}u^{j}_{\delta}\big{)}-u^{j}_{\delta}\mathcal{R}_{i}\mathcal{R}_{j}\big{(} \mathbf{1}_{\vartheta_{\delta}\leq L}\vartheta_{\delta}u^i_{\delta}\big{)}]\|_{L^1(0,T;L^{\frac{d}{d-1}})}\Big{)}\\
&\quad\leq \frac{Ck}{|\log{h}|^{\frac{1}{2}}}\big{(}1+\|\nabla u_{\delta}\|_{L^{2}(0,T;L^2)}\|\vartheta_{\delta}\mathbf{1}_{\vartheta_{\delta}\leq L} u_{\delta}\|_{L^2(0,T;L^{\frac{2d}{d-2}})}\big{)}\\
&\quad\leq \frac{CkL}{|\log{h}|^{\frac{1}{2}}}.
\end{split}
\end{equation}
For the term $I_{2}^6$, it holds
\begin{equation}\nonumber
\begin{split}
&I_{2}^6\leq Ck\sum_{i,j=1}^{d}\big{(}\|\mathcal{R}_{i}\mathcal{R}_{j} (\vartheta_{\delta}\mathbf{1}_{\vartheta_{\delta}\geq L} u^{i}_{\delta}u^{j}_{\delta})\|_{L^{1}(0,T;L^{p_{2}})}+\|u^{j}_{\delta}\mathcal{R}_{i}\mathcal{R}_{j}(\vartheta_{\delta} \mathbf{1}_{\vartheta_{\delta}\geq L}u^{i}_{\delta})\|_{L^{1}(0,T;L^{p_{2}})}\big{)}\\
&\quad\leq Ck\| \vartheta_{\delta} \mathbf{1}_{\vartheta_{\delta}\geq L}\|_{L^{\infty}(0,T;L^{p_{3}})}\|u_{\delta}\|_{L^2(0,T;L^{\frac{2d}{d-2}})}^2 \\
&\quad\leq\frac{Ck\|\vartheta_{\delta}\|_{L^{\infty}(0,T;L^{\gamma_{0}})}^{\frac{\gamma_{0}}{p_{3}}}\|u_{\delta}\|_{L^2(0,T;H^1)}^2}{L^{\frac{1}{d}(\gamma_{0}-\frac{d}{2})}}\leq  \frac{Ck}{L^{\frac{1}{d}(\gamma_{0}-\frac{d}{2})}},
\end{split}
\end{equation}
with the constants $p_{2}\in(1,\infty)$ and $p_{3}\in (\frac{d}{2},\gamma_{0})$ satisfying
$$
\frac{1}{p_{2}}+\frac{d-2}{d}=\frac{1}{p_{2}},\quad \quad \frac{\gamma_{0}-p_{3}}{p_{3}}=\frac{1}{d}(\gamma_{0}-\frac{d}{2}).
$$
 \begin{itemize}
\item{Case 2: $d=2.$}
\end{itemize}
By $(\ref{PsiR})_{1}$, (\ref{LNW}), (\ref{risz2}) and $H^1(\mathbb{T}^2)\hookrightarrow L^{p}(\mathbb{T}^2)$ for any $p\in[1,\infty)$, the term $I_{1}^6$ can be estimated by
\begin{equation}\nonumber
\begin{split}
&I_{1}^6\leq Ck\sum_{i,j=1}^{2}\int_{0}^{t}L_{h,1}\big{(} \mathcal{R}_{i}\mathcal{R}_{j} (\vartheta_{\delta}\mathbf{1}_{\vartheta_{\delta}\leq L} u^{i}_{\delta}u^{j}_{\delta})-u^{j}_{\delta}\mathcal{R}_{i}\mathcal{R}_{j}( \mathbf{1}_{\vartheta_{\delta}\leq L}\vartheta_{\delta} u^i_{\delta})\big{)}d\tau\\
&\quad\leq Ck\Big{(}\frac{1}{|\log{h}|^{\frac{1}{2}}}+\frac{1}{|\log{h}|^{\frac{3}{2}}}\sum_{i,j=1}^2\|\nabla\big{ [}\mathcal{R}_{i}\mathcal{R}_{j} \big{(}\vartheta_{\delta}\mathbf{1}_{\vartheta_{\delta}\leq L} u^{i}_{\delta}u^{j}_{\delta}\big{)}-u^{j}_{\delta}\mathcal{R}_{i}\mathcal{R}_{j}\big{(} \mathbf{1}_{\vartheta_{\delta}\leq L}\vartheta_{\delta}u^i_{\delta}\big{)} \big{]}\|_{L^1(0,T;L^{\frac{4}{3}})}\Big{)}\\
&\quad\leq \frac{Ck}{|\log{h}|^{\frac{1}{2}}}\big{(}1+\|\nabla u_{\delta}\|_{L^{2}(0,T;L^2)}\|\vartheta_{\delta}\mathbf{1}_{\vartheta_{\delta}\leq L} u_{\delta}\|_{L^2(0,T;L^{4})}\big{)}\\
&\quad\leq \frac{Ck(1+L\|u_{\delta}\|_{L^2(0,T;H^1)}^2)}{|\log{h}|^{\frac{1}{2}}}\leq \frac{CkL}{|\log{h}|^{\frac{1}{2}}}.
\end{split}
\end{equation}
 For the term $I_{2}^6$, we make use of $(\ref{PsiR})_{1}$, (\ref{LNW}), (\ref{risz1}) and $H^1(\mathbb{T}^2)\hookrightarrow \mathcal{BMO}(\mathbb{T}^2)\hookrightarrow L^{p}(\mathbb{T}^{2})$ for any $p\in[1,\infty)$ to obtain
\begin{equation}\nonumber
\begin{split}
&I_{2}^6\leq Ck\sum_{i,j=1}^2\| \mathcal{R}_{i}\mathcal{R}_{j} (\vartheta_{\delta}\mathbf{1}_{\vartheta_{\delta}\leq L} u^{i}_{\delta}u^{j}_{\delta})-u^{j}_{\delta}\mathcal{R}_{i}\mathcal{R}_{j}( \mathbf{1}_{\vartheta_{\delta}\leq L}\vartheta_{\delta} u^i_{\delta})\|_{L^1(0,T;L^{\frac{4\gamma_{0}}{3\gamma_{0}+1}})}\\
&\quad\leq Ck\|u_{\delta}\|_{L^2(0,T;\mathcal{BMO})}\|\vartheta_{\delta}\mathbf{1}_{\vartheta_{\delta}\geq L} u_{\delta}\|_{L^2(0,T;L^{\frac{4\gamma_{0}}{3\gamma_{0}+1}})}\\
&\quad\leq Ck\| u_{\delta}\|_{L^{2}(0,T;H^1)}\|\vartheta_{\delta}\mathbf{1}_{\vartheta_{\delta}\geq L}\|_{L^{\infty}(0,T;L^{\frac{2\gamma_{0}}{\gamma_{0}+1}})}\|u_{\delta}\|_{L^2(0,T;L^{\frac{4\gamma_{0}}{\gamma_{0}-1}})}\\
&\quad \leq \frac{Ck}{L^{\frac{\gamma_{0}-1}{2}}}\|u_{\delta}\|_{L^2(0,T;H^1)}^2\|\vartheta_{\delta}\|_{L^{\infty}(0,T;L^{\gamma_{0}})}^{\frac{\gamma_{0}+1}{2}}\leq \frac{Ck}{L^{\frac{\gamma_{0}-1}{2}}}
\end{split}
\end{equation}
Combining the above two cases about the estimates of $I_{i}^6$ $(i=1,2)$ together, we derive
\begin{equation}\label{I466}
\begin{split}
&I_{4}^5\leq Ck\big{(}\frac{L}{|\log{h}|^{\frac{1}{2}}}+\frac{1}{L^{\frac{1}{d}(\gamma_{0}-\frac{d}{2})}} \big{)},\quad\quad d\geq 2.
\end{split}
\end{equation}
Thence one can choose $L=|\log{h}|^{\frac{1}{2(\frac{\gamma_{0}}{d}+\frac{1}{2})}}$ in (\ref{I466}) to have
 \begin{equation}\nonumber
 \begin{split}
 &I_{4}^5\leq \frac{Ck}{|\log{h}|^{\frac{1}{2}\min\{1, \frac{2\gamma_{0}-d}{2\gamma_{0}+d}\}}}.
 \end{split}
 \end{equation}
 Substituting the above estimates of $I_{i}^{5}$ $(i=1,2,3,4)$ into (\ref{481}), we prove $(\ref{EVF})$. The proof of Lemma \ref{lemma48} is complete.
  \end{proof}
 
 \begin{remark}
 It should be noted that Lemma \ref{lemma48} about the compactness of effect viscous flux indeed holds for any adiabatic constants $\gamma,\alpha>\frac{d}{2}$.
 \end{remark}
 
 \underline{\it\textbf{Proof of Theorem \ref{theorem11} and Theorem \ref{theorem12}:~}} Let $T>0$, $(\rho_{\delta},n_{\delta},(\rho_{\delta}+n_{\delta})u_{\delta})$ be the weak solution to the IVP $(\ref{twodelta})$-$(\ref{ddelta})$ for $\delta\in(0,1)$ given by Proposition \ref{prop31}, $\vartheta_{\delta}:=\rho_{\delta}+n_{\delta}$, $w_{\delta}$ be the solution to the IVP $(\ref{w})$, and $(\rho, u,(\rho+n)u)$ be the limit obtained by Lemma \ref{lemma42}. Then, with the help of the uniform estimates obtained from Lemmas \ref{lemma45} to \ref{lemma48}, for any $\zeta>0$, there exists a constant $\delta_{1}(\zeta)\in(0,1)$ such that for $\delta\in (0,\delta_{1}(\zeta))$, $h\in(0,h_{0})$, $k\geq1$ and $\lambda_{0}\geq1$, we have
 \begin{equation}\nonumber
\begin{split}
&\int_{\mathbb{T}^{2d}}\overline{\mathcal{K}}_{h}(x-y)\chi(\Lambda[\vartheta_{\delta}])(w^{x}_{\delta}+w^{y}_{\delta})dxdy\\
&\quad\leq  \frac{C}{k}+ C\lambda_{0}k^{p_{0}}h_{1}(\delta,\zeta,h) +(C-2\lambda_{0})\int_{0}^{t}\int_{\mathbb{T}^{2d}}\overline{\mathcal{K}}_{h}(x-y)|\Lambda[\vartheta_{\delta}]| \Xi_{\delta}^{x}w_{\delta}^{x}dxdyd\tau,
 \end{split}
\end{equation}
where $C>0$ is a constant independent of $\delta$, $h$, $\zeta$ and $\lambda_{0}$ and $h_{1}(\delta,\zeta,h)\in(0,1)$ is given by
 \begin{equation}\nonumber
\begin{split}
&h_{1}(\delta,\zeta,h):= L_{h,1}(\rho_{0,\delta}+n_{0,\delta})+\frac{1}{|\log{h}|^{\frac{1}{2}\min\{\frac{1}{2},\frac{2\gamma_{0}-d}{\gamma_{0}+d}\}}}\\
&\quad\quad\quad\quad\quad\quad+\delta^{\beta}+\zeta+\int_{0}^{T}\big{(}L_{h,1}(A_{\rho,n})+L_{h,1}(B_{\rho,n})\big{)}dt,\quad \quad \gamma_{0}:=\min\{\gamma,\alpha\}.
 \end{split}
\end{equation}
Choosing
  \begin{equation}\nonumber
\begin{split}
&\lambda_{0}=\frac{C}{2}+1,\quad\quad k=h_{1}(\delta,\zeta,h)^{-\frac{1}{p_{0}+1}},
  \end{split}
\end{equation}
 we have
  \begin{equation}\nonumber
\begin{split}
&\int_{\mathbb{T}^{2d}}\overline{\mathcal{K}}_{h}(x-y)\chi(\Lambda[\vartheta_{\delta}])(w^{x}_{\delta}+w^{y}_{\delta})dxdy\leq C h_{1}(\delta,\zeta,h)^{\frac{1}{p_{0}+1}},
  \end{split}
\end{equation}
 which together with $(\ref{Ew})$ implies for any $\sigma_{*}>0$ that
 \begin{equation}\nonumber
 \begin{split}
 &\big{(} L_{h,1}(\vartheta_{\delta}) \big{)}^2\leq \frac{C}{\log{(1+|\log{\sigma_{*}}|})}+\frac{Ch_{1}(\delta,\zeta,h)^{\frac{1}{p_{0}+1}}}{\sigma_{*}}.
 \end{split}
 \end{equation}
 Thus, one can choose $\sigma_{*}=h_{1}(\delta,\zeta,h)^{\frac{1}{2(p_{0}+1)}}$ to obtain
  \begin{equation}\nonumber
 \begin{split}
  &\big{(} L_{h,1}(\vartheta_{\delta}) \big{)}^2\leq \frac{C}{\log{(1+|\log{h_{1}(\delta,\zeta,h)}|)}}+h_{1}(\delta,\zeta,h)^{\frac{1}{2(p_{0}+1)}}.
  \end{split}
 \end{equation}
 Since it follows $\lim_{ h \rightarrow 0}\limsup_{\delta\rightarrow 0} h_{1}(\delta,\zeta,h)=\zeta$ for any small $\zeta>0$, we gain
 \begin{equation}\label{final}
 \begin{split}
 &\lim_{ h \rightarrow 0}\limsup_{\delta\rightarrow 0} \underset{t\in[0, T]}{{\rm{ess~sup}}}~L_{h,1}(\vartheta _{\delta})=0.
 \end{split}
 \end{equation}
Due to $(\ref{rhonifdelta})$, (\ref{final}), Lemma \ref{lemma63} below and the fact that $\partial_{t}\vartheta_{\delta}$ is uniformly bounded in $L^{\infty}(0,T;W^{-1,\frac{2\gamma_{0}}{\gamma_{0}+1}}(\mathbb{T}^{d}))$, it holds up to a subsequence (still denoted by $(\rho_{\delta},n_{\delta})$) that
\begin{equation}\label{strongf}
\begin{split}
(\rho_{\delta},n_{\delta})\rightarrow (\rho,n)\quad\text{in}~L^1(0,T;L^1(\mathbb{T}^{d}))\times L^1(0,T;L^1(\mathbb{T}^{d}))\quad\text{as}~\delta\rightarrow 0,
 \end{split}
 \end{equation}
 which together with Lemma \ref{lemma41} and the Egorov theorem gives rise to
 \begin{equation}\label{strongfinal}
 \begin{split}
 &P_{\delta}(\rho_{\delta},n_{\delta})\rightarrow P(\rho,n)\quad\text{in}~L^1(0,T;L^1(\mathbb{T}^{d}))\quad\text{as}~\delta\rightarrow 0.
 \end{split}
 \end{equation}
 By $(\ref{limitdelta1})$ and $(\ref{strongfinal})$,  one can show that $(\rho, n, (\rho+n)u)$ satisfies the properties (1)-(3) in Definition \ref{defn11}. Finally, it is easy to show the energy inequality $(\ref{energy})$ by the lower semi-continuity of weak limits. Theorem \ref{theorem11} and Theorem \ref{theorem12} are proved.

\section{Appendix}
In this section, we present some technical lemmas which are useful for our analysis.

\begin{lemma}\label{lemma61}
Let $p\in[1,\infty)$, $\varepsilon>0$ and $f_{\varepsilon}$ be uniformly bounded in $L^{p}(\mathbb{T}^{d})$. Then, for any $h\in(0,h_{0})$, it holds
\begin{align}
&\sup_{|z|<h}\int_{\mathbb{T}^{d}}|f_{\varepsilon}(x+z)-f_{\varepsilon}(x)|^{p}dx\leq C\Big{(}\frac{\|f_{\varepsilon}\|_{L^1}^p}{|\log{h}|^{p}}+L_{h^{\frac{1}{d+1}},p}(f_{\varepsilon})\Big{)},\label{NL}\\
&L_{h,p}(f_{\varepsilon})\leq C\Big{(}\frac{\|f_{\varepsilon}\|_{L^p}^p}{|\log{h}|^{\frac{1}{2}}}+\sup_{ |z|<\frac{1}{|\log{h}|}}\int_{\mathbb{T}^{d}}|f_{\varepsilon}(x+z)-f_{\varepsilon}(x)|^{p}dx \Big{)}, \label{LN}
\end{align}
where $L_{h,p}(f)$ is defined by $(\ref{L})$ and $C>0$ is a constant independent of $h$ and $\varepsilon$.

Furthermore, if $f_{\varepsilon}$ is uniformly bounded in $W^{1,q}(\mathbb{T}^{d})$ for $q\in [1,d)$, then we have
\begin{equation}
\begin{split}
L_{h,p}(f_{\varepsilon})\leq C\Big{(}\frac{\|f_{\varepsilon}\|_{L^p}^p}{|\log{h}|^{\frac{1}{2}}}+\frac{\|\nabla f_{\varepsilon}\|_{L^{q}}^p}{|\log{h}|^{pd\min\{0,\frac{1}{p}-\frac{1}{q}\}+p}}\Big{)},\quad  p\in [1,\frac{qd}{d-q}). \label{LNW}
\end{split}
\end{equation}
\end{lemma}
\begin{proof}
 First, for any $|z|<h$ and $\sigma\in (0,1)$, it can be shown by (\ref{K}) and the fact $\|\nabla \mathcal{K}_{\sigma}\|_{L^{\infty}}\leq \frac{C}{\sigma^{d+1}}$ that
\begin{equation}\nonumber
\begin{split}
&\int_{\mathbb{T}^{d}}|f_{\varepsilon}(x+z)-f_{\varepsilon}(x)|^pdx\\
&\quad\leq C\int_{\mathbb{T}^d}\Big{|}  \frac{\mathcal{K}_{\sigma}}{\|\mathcal{K}_{h}\|_{L^1}}\ast f_{\varepsilon}(x+z)-\frac{\mathcal{K}_{\sigma}}{\|\mathcal{K}_{h}\|_{L^1}}\ast f_{\varepsilon}(x) \Big{|}^pdx+C\int_{\mathbb{T}^d} \Big{|} f_{\varepsilon}(x)-\frac{\mathcal{K}_{\sigma}}{\|\mathcal{K}_{h}\|_{L^1}}\ast f_{\varepsilon}(x) \Big{|}^pdx\\
&\quad\leq \frac{C}{\|\mathcal{K}_{\sigma}\|_{L^1}^{p}}\int_{\mathbb{T}^d} \Big{|} \int_{\mathbb{T}^d}\big{(} \mathcal{K}_{\sigma}(x+z-y)-\mathcal{K}_{\sigma}(x-y) \big{)}f_{\varepsilon}(y)dy \Big{|}^pdx\\
&\quad\quad+\frac{C}{\|\mathcal{K}_{\sigma}\|_{L^1}^{p}}\int_{\mathbb{T}^d} \Big{|} \int_{\mathbb{T}^d}\mathcal{K}_{\sigma}(x-y)\big{(}f_{\varepsilon}(x)-f_{\varepsilon}(y)\big{)}dy \Big{|}^pdx\\
&\quad\leq \frac{C|z|^p}{|\log{\sigma}|^p}\|\nabla \mathcal{K}_{\sigma}\|_{L^{\infty}}^p\|f_{\varepsilon}\|_{L^1}^p+\frac{C}{|\log{\sigma}|}\int_{\mathbb{T}^{2d}}\mathcal{K}_{\sigma}(x-y)|f_{\varepsilon}(x)-f_{\varepsilon}(y)|^pdxdy\\
&\quad\leq \frac{Ch^p\|f_{\varepsilon}\|_{L^1}^p}{\sigma^{p(d+1)}|\log{\sigma}|^p}+L_{\sigma,p}(f_{\varepsilon}).
\end{split}
\end{equation}
 One may choose $\sigma=h^{\frac{1}{d+1}}$ in the above inequality to prove (\ref{NL}).

Next, we show (\ref{LN}). It follows
\begin{equation}\label{ppdda}
\begin{split}
L_{h,p}(f_{\varepsilon})&= \int_{\{|x-y|\geq \frac{1}{2}\}}\overline{\mathcal{K}}_{h}(x-y)|f_{\varepsilon}(x)-f_{\varepsilon}(y)|^pdxdy\\
&\quad+\int_{\{\frac{1}{|\log{h}|}\leq |x-y|<\frac{1}{2}\}}\overline{\mathcal{K}}_{h}(x-y)|f_{\varepsilon}(x)-f_{\varepsilon}(y)|^pdxdy\\
&\quad+\int_{\{0\leq |x-y|<\frac{1}{|\log{h}|}\}}\overline{\mathcal{K}}_{h}(x-y)|f_{\varepsilon}(x)-f_{\varepsilon}(y)|^pdxdy ,\\
\end{split}
\end{equation}
Due to (\ref{K}) and $\mathcal{K}_{h}(x)\leq C$ for any $|x|\geq \frac{1}{2}$, one can bound $I_{1}^8$ by
\begin{equation}\nonumber
\begin{split}
\int_{\{|x-y|\geq \frac{1}{2}\}}\overline{\mathcal{K}}_{h}(x-y)|f_{\varepsilon}(x)-f_{\varepsilon}(y)|^pdxdy\leq\frac{C\|f_{\varepsilon}\|_{L^p}^p}{|\log{h}|} .
\end{split}
\end{equation}
For the second term on the right-hand side of \eqref{ppdda}, we have
\begin{equation}\nonumber
\begin{split}
&\int_{\{\frac{1}{|\log{h}|}\leq |x-y|<\frac{1}{2}\}}\overline{\mathcal{K}}_{h}(x-y)|f_{\varepsilon}(x)-f_{\varepsilon}(y)|^pdxdy\\
&\quad\leq C\sum_{n=0}^{\log{|\log{h}|}-1} \frac{1}{|\log{h}|}\int_{\{\frac{1}{2}e^{-n-1}\leq |z|\leq \frac{1}{2}e^{-n}\} }e^{n d}\int_{\mathbb{T}^d}|f_{\varepsilon}(x)-f_{\varepsilon}(x-z)|^pdxdz\\
&\quad\leq \frac{C\|f_{\varepsilon}\|_{L^p}^p}{|\log{h}|^{\frac{1}{2}}}\sum_{n=0}^{\infty}e^{-\frac{n}{2}}\leq \frac{C\|f_{\varepsilon}\|_{L^p}^p}{|\log{h}|^{\frac{1}{2}}},
\end{split}
\end{equation}
where we have used the facts $\mathcal{K}_{h}(z)\leq \frac{C}{|z|^{d}}$ for any $|z|>0$ and $|\log{h}|^{-\frac{1}{2}}\leq e^{-\frac{1}{2}(n+1)}$ for any integer $0\leq n\leq \log{|\log{h}|}-1$. Finally, it is easy to obtain
\begin{equation}\nonumber
\begin{split}
&\int_{\{0\leq |x-y|<\frac{1}{|\log{h}|}\}}\overline{\mathcal{K}}_{h}(x-y)|f_{\varepsilon}(x)-f_{\varepsilon}(y)|^pdxdy\\
&\quad=\int_{\{0\leq|z|< \frac{1}{|\log{h}|}\}}\overline{\mathcal{K}}_{h}(z)dz\int_{\mathbb{T}^d}|f_{\varepsilon}(x)-f_{\varepsilon}(x-z)|^pdx\leq \sup_{ |z|< \frac{1}{|\log{h}|}}\int_{\mathbb{T}^{d}}|f_{\varepsilon}(x+z)-f_{\varepsilon}(x)|^{p}dx .
\end{split}
\end{equation}
Combing the above estimates together, we derive (\ref{LN}).

Moreover, for any $|z|<h$, due to the Sobolev inequality and
$$
f_{\varepsilon}(x+z)-f_{\varepsilon}(x)=\int_{0}^{1}z \cdot \nabla f_{\varepsilon}(x+\omega z)d\omega,
$$
it holds
\begin{equation}\nonumber
\begin{split}
&\big{(} \int_{\mathbb{T}^{d}}|f_{\varepsilon}(x+z)-f_{\varepsilon}(x)|^{p}dx \big{)}^{\frac{1}{p}}\leq \big{(} \int_{\mathbb{T}^{d}}|f_{\varepsilon}(x+z)-f_{\varepsilon}(x)|^{q}dx \big{)}^{\frac{1}{q}} \leq  |z| \|\nabla f_{\varepsilon}\|_{L^q},\quad p\in[1,q],
\end{split}
\end{equation}
and
\begin{equation}\nonumber
\begin{split}
&\big{(} \int_{\mathbb{T}^{d}}|f_{\varepsilon}(x+z)-f_{\varepsilon}(x)|^{p}dx \big{)}^{\frac{1}{p}}\\
&\quad\leq \Big{(}\int_{\mathbb{T}^{d}}|f_{\varepsilon}(x+z)-f_{\varepsilon}(x)|^{\frac{qd}{d-q}}dx\Big{)}^{\frac{(d-q)\theta }{qd}}  \Big{(}\int_{\mathbb{T}^{d}}\int_{0}^{1}|z|^{q} | \nabla f_{\varepsilon}(x+\omega z)|^{q}d\omega dx\Big{)}^{\frac{1-\theta}{q}}\\
& \quad \leq C|z|^{1-\theta} \|f_{\varepsilon}- \frac{1}{|\mathbb{T}^{d}|}\int_{\mathbb{T}^{d}} fdx\|_{L^{\frac{qd}{d-q}}}^{\theta } \|\nabla f_{\varepsilon}\|_{L^{q}}^{1-\theta}\\
&\quad\leq C|z|^{1-\theta}\|\nabla f_{\varepsilon}\|_{L^q},\quad p\in(q,\frac{qd}{d-q}), 
\end{split}
\end{equation}
where the constant $\theta\in(0,1)$ satisfies $
\frac{1}{p}=\frac{d-q}{qd}\theta+\frac{1}{q}(1-\theta)$. Thus, (\ref{LNW}) follows by (\ref{LN}) and the above two estimates. The proof of Lemma \ref{lemma61} is complete.
\end{proof}

The following compactness criterion is a consequence of Lemma \ref{lemma61} and the Riesz-Fr$\rm{\acute{e}}$chet-Kolmogorov criterion (cf. \cite[page 111]{brezis1}).
\begin{lemma}\label{lemma62}
Let $\varepsilon>0$ and $f_{\varepsilon}$ be uniformly bounded in $L^{p}(\mathbb{T}^{d})$ for $p\in[1,\infty)$. Then the sequence $f_{\varepsilon}$ is strongly compact in $L^{p}(\mathbb{T}^{d})$ if and only if
\begin{equation}
\begin{split}
\lim_{h\rightarrow 0}\limsup_{\varepsilon\rightarrow 0}L_{h,p}(f_{\varepsilon})=0,\label{knormc1}
\end{split}
\end{equation}
where $L_{h,p}(f_{\varepsilon})$ is defined by $(\ref{L})$.
\end{lemma}

For the compactness in time and space, we have the following lemma of Lions-Aubin type.
\begin{lemma}\label{lemma63}
Let $T>0$, $\varepsilon>0$, $f_{\varepsilon}$ be uniformly bounded in $L^{p}(0,T;L^p(\mathbb{T}^{d}) )$ for $p\in[1,\infty)$ and $\partial_{t}f_{\varepsilon}$ be uniform bounded in $L^{q}(0,T;W^{-m,1}(\mathbb{T}^{d}) )$ for $q>1$ and $m\geq0$. Then $f_{\varepsilon}$ is strongly compact in $L^{p}(0,T;L^{p}(\mathbb{T}^{d}))$ if and only if
\begin{equation}
\begin{split}
\lim_{h\rightarrow 0}\limsup_{\varepsilon\rightarrow 0}\int_{0}^{T}L_{h,p}(f_{\varepsilon})dt=0,\label{knormc}
\end{split}
\end{equation}
where $L_{h,p}(f_{\varepsilon})$ is defined by $(\ref{L})$.
\end{lemma}
\begin{proof}
For any $|z|<h<\frac{T}{2}$, it holds
\begin{equation}\nonumber
\begin{split}
&\|f_{\varepsilon}(x+z,t+h)-f_{\varepsilon}(x,t)\|_{L^{p}(0,\frac{T}{2};L^{p})}\\
&\quad=\|f_{\varepsilon}(x+z,t+h)-f_{\varepsilon}(x,t+h)\|_{L^{p}(0,\frac{T}{2};L^{p})}+\|f_{\varepsilon}\ast \eta_{\delta}(x,t+h)-f_{\varepsilon}\ast \eta_{\delta} (x,t)\|_{L^{p}(0,\frac{T}{2};L^{p})}\\
&\quad\quad+\|f_{\varepsilon}(x,t+h)-f_{\varepsilon}\ast \eta_{\delta}(x,t+h)\|_{L^{p}(0,\frac{T}{2};L^{p})}+\|f_{\varepsilon}(x,t)-f_{\varepsilon}\ast \eta_{\delta}(x,t)\|_{L^{p}(0,\frac{T}{2};L^{p})}
\end{split}
\end{equation}
where $\eta_{\delta}\in C^{\infty}_{c}(\mathbb{T}^{d})$ for $\delta\in(0,1)$ is the Friedrichs mollifier. It can be verified by $(\ref{NL})$ and (\ref{knormc1}) that
\begin{equation}\nonumber
\begin{split}
&\lim_{h\rightarrow 0}\limsup_{\varepsilon\rightarrow 0}\|f_{\varepsilon}(x+z,t+h)-f_{\varepsilon}(x,t+h)\|_{L^{p}(0,\frac{T}{2};L^{p})}=0,\\
&\lim_{h\rightarrow 0}\limsup_{\varepsilon\rightarrow 0} \|f_{\varepsilon}\ast \eta_{\delta}(x,t+h)-f_{\varepsilon}\ast \eta_{\delta} (x,t)\|_{L^{p}(0,\frac{T}{2};L^{p})}\\
&\quad\leq \lim_{h\rightarrow 0}\limsup_{\varepsilon\rightarrow 0} h^{1-\frac{1}{q}} \|\partial_{t} f_{\varepsilon}\ast \eta_{\delta}\|_{L^{q}(0,T;L^{p})}=0,\quad \delta\in(0,1),\\
\end{split}
\end{equation}
and
\begin{equation}\nonumber
\begin{split}
&\lim_{\delta\rightarrow 0}\lim_{h\rightarrow 0}\limsup_{\varepsilon\rightarrow 0}\big(\|f_{\varepsilon}(x,t+h)-f_{\varepsilon}\ast \eta_{\delta}(x,t+h)\|_{L^{p}(0,\frac{T}{2};L^{p})}+\|f_{\varepsilon}(x,t)-f_{\varepsilon}\ast \eta_{\delta}(x,t)\|_{L^{p}(0,\frac{T}{2};L^{p})}\big)\\
&\quad\leq 2\lim_{\delta\rightarrow 0}\limsup_{\varepsilon\rightarrow 0} \sup_{|z|<\delta}\|f_{\varepsilon}(x+z,t)-f_{\varepsilon}(x,t)\|_{L^{p}(0,T;L^{p})}=0.
\end{split}
\end{equation}
By the above estimates and the Riesz-Fr$\rm{\acute{e}}$chet-Kolmogorov criterion, $f_{\varepsilon}$ is strongly compact in $L^{p}(0,\frac{T}{2};L^{p}(\mathbb{T}^{d}))$. Repeating same arguments for $\tilde{f}_{\varepsilon}(x,t):=f_{\varepsilon}(x,T-t)$, we derive the compactness of $f_{\varepsilon}$ in $L^{p}(\frac{T}{2},T;L^{p}(\mathbb{T}^{d}))$.
\end{proof}

The next lemma is useful in Lemmas \ref{lemma46} and \ref{lemma48}.
\begin{lemma}[\!\!\cite{belgacem1, bresch3}]\label{lemma64}
For any $f\in H^1(\mathbb{T}^{d})$, it holds
\begin{equation}
\begin{split}
|f(x)-f(y)|\leq C|x-y|\big{(} D_{|x-y|}f(x)+D_{|x-y|} f(y) \big{)},\label {D}
\end{split}
\end{equation}
where $C>0$ is a constant depending only on $d$ and $D_{r}f(x)$ is denoted as
\begin{equation}\label{Maximal}
\begin{split}
D_{r}f(x):=\frac{1}{r}\int_{|z|\leq r}\frac{|\nabla f(x+z)|}{|z|^{d-1}}dz,
\end{split}
\end{equation}
which satisfies
\begin{equation}
\begin{split}
D_{r} f(x)\leq CM|\nabla f|(x),\label{DM}
\end{split}
\end{equation}
where the localized maximal operator $M: L^p(\mathbb{T}^{d})\rightarrow L^p(\mathbb{T}^{d})$ for any $p\in[1,\infty)$ is defined by
\begin{equation}\label{M}
\begin{split}
Mf(x):=\sup_{r\in(0,1]}\frac{1}{|B(0,r)|}\int_{B(0,r)}f(x+z)dz.
\end{split}
\end{equation}
In addition, we have
\begin{equation}
\begin{split}
\int_{\mathbb{T}^{d}}\mathcal{K}_{h}(z)\|D_{|z|}f-D_{|z|}f(\cdot +z)\|_{L^2}dz\leq C\|f\|_{H^1}|\log{h}|^{\frac{1}{2}},\label{DH1}
\end{split}
\end{equation}
where the kernel $\mathcal{K}_{h}$ is given by $(\ref{mathcalK})$.
\end{lemma}

The following lemma about the commutator estimates of the Riesz operator $\mathcal{R}_{i}:=(-\Delta)^{-\frac{1}{2}}\partial_{i}$ was introduced by Coifman-Rochberg-Weiss \cite{coifman2} and Coifman-Meyer \cite{coifman1} and applied to the global existence of weak solution for isentropic compressible Navier-Stokes equations by Lions \cite{lions2}.
\begin{lemma}[\!\!\cite{coifman2,coifman1}]\label{lemma66}
 For $p\in(1,\infty)$, it holds
\begin{equation}
\begin{split}
\|f\mathcal{R}_{i}\mathcal{R}_{j} g-\mathcal{R}_{i}\mathcal{R}_{j}(fg)\|_{L^{p}}\leq C\|f\|_{\mathcal{BMO}}\|g\|_{L^{p}},\quad  f\in \mathcal{BMO}(\mathbb{T}^{d}) ,~g\in  L^{p}(\mathbb{T}^{d}) ,\label{risz1}
\end{split}
\end{equation}
where $\mathcal{BMO}(\mathbb{T}^{d})$ denotes the bounded mean oscillation space and  $C>0$ is a constant depending only on $p$ and $d$.

Moreover, for $q_{i}\in (1,\infty)~(i=1,2,3)$ satisfying $\frac{1}{q_{1}}=\frac{1}{q_{2}}+\frac{1}{q_{3}}$, it holds
\begin{equation}
\begin{split}
\|\nabla[f\mathcal{R}_{i}\mathcal{R}_{j}g-\mathcal{R}_{i}\mathcal{R}_{j}(fg)]\|_{L^{q_{1}}}\leq C\|\nabla f\|_{L^{q_{2}}}\|g\|_{L^{q_{3}}}\quad  f\in W^{1,q_{2}}(\mathbb{T}^{d}) ,~g\in  L^{q_{3}}(\mathbb{T}^{d}) ,\label{risz2}
\end{split}
\end{equation}
where  $C>0$ is a constant depending only on $q_{i}~(i=1,2)$ and $d$.
\end{lemma}


\noindent
\textbf{Acknowledgments.} The authors would like to thank the referees for their valuable suggestions and comments on the manuscript. The research of the paper is supported by National Natural Science Foundation of China (No.11931010, 11671384 and 11871047) and by the key research project of Academy for Multidisciplinary Studies, Capital Normal University and by the Capacity Building for Sci-Tech Innovation-Fundamental Scientific Research Funds (No.007/20530290068).





\end{document}